\documentclass[a4paper]{amsart}
\usepackage{amsfonts}\usepackage{amssymb}\usepackage{amsmath}\usepackage{amsthm}
\usepackage{latexsym}
\usepackage{graphicx}
\usepackage{layout}
\usepackage[colorlinks=true,urlcolor=blue,citecolor=red,linkcolor=blue,linktocpage,pdfpagelabels,bookmarksnumbered,bookmarksopen]{hyperref}

\newcommand{\sezione}[1]{\section{#1}\setcounter{equation}{0}}

\def\R{\mathbb{R}}

\def\di12{\mathcal{D}^{1,2}(\R^n)}

\def\l{{\lambda}}

\def\0l{_{0,\l}}
\def\1l{_{1,\l}}
\def\2l{_{2,\l}}
\def\3l{_{3,\l}}
\def\4l{_{4,\l}}

\def\sideremark#1{\ifvmode\leavevmode\fi\vadjust{\vbox to0pt{\vss
 \hbox to 0pt{\hskip\hsize\hskip1em
 \vbox{\hsize2.1cm\tiny\raggedright\pretolerance10000
  \noindent #1\hfill}\hss}\vbox to15pt{\vfil}\vss}}}%

\newtheorem{teo}{Theorem}
\newtheorem{prop}[teo]{Proposition}
\newtheorem{cor}[teo]{Corollary}
\newtheorem{lem}[teo]{Lemma}
\theoremstyle{definition}
\newtheorem{defin}{Definition}

\theoremstyle{remark}
\newtheorem{oss}{Remark}

\begin{document}
\title[Morse Index of Solutions to Sinh-Poisson Equation]{Morse Index of Multiple Blow-Up Solutions
to the Two-Dimensional Sinh-Poisson Equation}
\author{Ruggero Freddi}

\address{Sbai, Universit\`a di Roma ``La Sapienza", Via Antonio Scarpa, 14, Palazzina E, 00161 Roma, Italy.}
\email{ruggero.freddi@uniroma1.it}

\maketitle
\begin{abstract}
In this paper we consider the Dirichlet problem
\begin{equation} \label{iniz}
\begin{cases} -\Delta u = \rho^2 (e^{u} - e^{-u}) & \text{ in } \Omega\\
u=0 & \text{ on } \partial \Omega, 
\end{cases}
\end{equation}
where $\rho$ is a small parameter and $\Omega$ is a $C^2$ bounded domain in $\R^2$. \cite{io} proves the existence of a $m$-point blow-up solution $u_\rho$ jointly with its asymptotic behaviour. we compute the Morse index of $u_\rho$ in terms of the Morse index of the associated Hamilton function of this problem. In addition, we give an asymptotic estimate for the first $4m$ eigenvalues and eigenfunctions.\\

\end{abstract}

\sezione{Introduction}\label{s0}
We are concerned with the study of the Morse index of the Dirichlet problem
\begin{equation} \label{iniz}
\begin{cases} -\Delta u = \rho^2 (e^{u} - e^{-u}) & \text{ in } \Omega\\
u=0 & \text{ on } \partial \Omega, 
\end{cases}
\end{equation}
where $\rho$ is a small parameter and $\Omega$ is a $C^2$ bounded domain in $\R^2$. This equation has been widely studied as it is strictly related to the vortex-type configuration for 2D turbulent Euler flows (see \cite{cho}, \cite{chor}, and \cite{mar}). Its importance is due to the fact that, suitably adapted, it describes interesting phenomena in widely different areas like liquid helium, meteorology and oceanography; it highlights effects that are important in all those subjects. Moreover, its dynamics are isomorphic to those of the electrostatic guiding-center plasma, which have been widely extended to describe strongly magnetized plasmas (see for instance \cite{sch}, \cite{mon}, and \cite{spi}).

It has been known since Kirchhoff \cite{kir} that, if we let $\xi_i \in \Omega$, $i=1, \ldots, m$, be the centres of the vorticity blobs, then the $\xi_i$'s obey an approximate Hamiltonian dynamic associated to the Hamiltonian function
\begin{equation} \label{ham}
\mathcal{F}(\xi_1, \ldots, \xi_m)=\frac{1}{2} \sum_{k=1}^m R(\xi_k) + \frac{1}{2} \sum_{1\leqslant k,j \leqslant m} \alpha_k \alpha_j G(\xi_k,\xi_j),
\end{equation}
with $\alpha_k \in \lbrace 1, -1 \rbrace$, $k=1, \ldots, m$, depending on the sign of the corresponding vorticity blob.
Through two different approaches, Joyce \cite{joyce} and Montgomery \cite{montg} proved at heuristic level that, if we let $\omega$ be the vorticity, $\psi$ the flow's steam function, $\beta \in \R$ the inverse of the temperature, and $Z>0$ an appropriate normalization constant, then we have $\omega(\psi)=\frac{2}{Z}\sinh (-\psi)$ for a flow with total vorticity equal to zero, i.e. $\int_{\Omega}\omega = 0$, and $\omega(\psi)=\frac{1}{Z}e^{-\beta \psi}$ for a flow with total vorticity equal to one, i.e. $\int_{\Omega}\omega = 1$. Setting $u=-\vert \beta \vert \psi$, the 2D Euler equation in stationary form
\begin{equation} \label{montg}
\begin{split}
\begin{cases} \textbf{w} \cdot \nabla \omega (\psi) = 0 & \text{in } \Omega\\
-\Delta \psi=0 & \text{in } \Omega\\
\textbf{w} \cdot \nu = 0 & \text{on } \partial \Omega ,\end{cases}
\end{split}
\end{equation}
where $\textbf{w}$ is the velocity field, reduces to the sinh-Poisson equation \eqref{iniz}, where $\rho=\sqrt{\frac{\vert \beta \vert}{Z}}$. More recently, in \cite{cag92} and \cite{cag95} was rigorously proved that for any $\beta \geqslant -9 \pi$, $\omega(\psi) = \frac{1}{Z}e^{-\beta \psi}$ is the mean field-limit vorticity for both the micro-canonical and the canonical equilibrium statistic distributions for the Hamiltonian point-vortex model. 
The solutions to \eqref{montg} with $\beta <0$ of the type of those suggested in \cite{joyce} (`negative temperature' states) have shown to represent very well the numerical experiment on the Navier-Stokes equations with high Reynolds number (\cite{mat}, \cite{mat1}, and \cite{mont1}). For further details and recent developments in the study of this problem see \cite{cag95}, \cite{maj}, and \cite{yin}.
Due to the just mentioned results much effort has been put into finding out explicit solutions for the Euler equations with Joy-Montgomery vorticity. Among the most relevants there are the Mallier-Maslowe \cite{mal} counter rotating vortices, and their generalization (\cite{cho98} and \cite{cho03}). The Mallier-Maslowe vortices are sign changing solutions to $-\Delta u = \rho^2 \sinh (u)$, with $1$-periodic boundary conditions, one absolute maxima and minima and two nodal domains in each periodic cell, the resulting Euler flow is composed of symmetric and disjoint regions where the velocity fields are counter directed. This type of solutions are a suitable initial data for numerical computations (\cite{yin} and \cite{maj}). Moreover, their explicit expression led to recent results which gave some insight of the properties of the non-linear dynamical stabilities of periodic array of vortices (\cite{jul}, \cite{gur}, and \cite{dau}).

A lot of effort has been put into the study of equation \eqref{iniz}. In \cite{spr} it has been introduced an analytic point of view into the study of this equation and it has been studied the behaviour of non-negative solutions as $\rho$ goes to zero, when $\Omega$ is a rectangle. In \cite{jos} some important properties of the solutions around the points $\xi_i$ have been proved. In \cite{esp} it has been built a positive solution whose concentration points converges to critical point of \eqref{ham}, as $\rho$ goes to zero. More recently, in \cite{io} and \cite{bart} it has been proved the existence of a sign changing solution whose velocity field converges to a sum of Dirac deltas centred at critical points of \eqref{ham}. Furthermore, in \cite{gro} it has been constructed a solution that converges to a sum of $k$ Dirac deltas with alternate sign all centred at the same critical point of the Robin's function. 

In \cite{io} it is proved that, under some assumptions on the points $\xi_1, \ldots, \xi_m$, for each $\rho$ sufficiently small there exists a solution $u_\rho$ to the equation \eqref{iniz}, and its profile is given (see Theorem \ref{ang} for additional properties of this solution). \\
In this paper we want to prove some additional properties of the $m$-peak solutions studied in  \cite{io}, namely its Morse index and related properties of the linearized operator. So let us introduce the following eigenvalue problem,
\begin{equation} \label{lin1}
\begin{cases} -\Delta v =\mu_\rho\rho^2 (e^{u_\rho} + e^{-u_\rho})v & \text{ in } \Omega\\
v=0 & \text{ on } \partial \Omega, 
\end{cases}
\end{equation}
and denote by $\mu_{\rho,j}$ and $v_{\rho,j}$ be respectively the $j$-th eigenvalue and the $j$-th eigenfunction of problem \eqref{lin1}. 
Let us recall that the Morse index of the solution $u_\rho$ is the sum of the dimensions of the eigenspaces relative to eigenvalues $\mu_\rho<1$ of the linearized equation \eqref{lin1}.

Moreover, let
\[ \tilde{v}_{\rho,j}^{(k)} = \chi_{B_{\frac{\sqrt{8} \epsilon}{\tau_k \rho}}(0)} (x) v_{\rho,j}\left(\frac{\tau_k \rho x}{\sqrt{8}}+ \xi_{\rho,k} \right),\quad k=1,..,m,
\]
be the rescaled function of the $j$-th eigenfunction around $\xi_{\rho,k}$. We begin by proving some results about the asymptotic behaviour of the eigenvalues $\mu_{\rho,j}$, the asymptotic behaviour of the eigenfunctions $v_{\rho,j}$ away from the points $\xi_1, \ldots, \xi_m$, and the asymptotic behaviour of the rescaled eigenfunctions $\tilde{v}^{(k)}_{\rho,j}$.
These estimates allow to compute the {\em Morse index} of the solution $u_\rho$. It will be strictly related with  Morse index of the Hamiltonian function \eqref{ham}.

Let us now state the main results of this work. We start by the first $m$ eigenvalues and eigenfunctions
\begin{teo} \label{main1}
For every $j=1, \ldots, m$ there exists an $k \in \lbrace 1, \ldots, m \rbrace$ and nonzero real constants $C_j^{k}\in\R$ such that
\begin{itemize}
\item [1.] $\mu_{\rho,j} <- \frac{1}{2\log(\rho)}$, for $\rho$ small enough;
\item [2.] $\lim_{\rho \to 0} \tilde{v}^{(k)}_{\rho,j} = C^{(k)}_j$ in $C^2_{\mathrm{loc}}(\R^2)$;
\item [3.] $\lim_{\rho \to 0} \frac{v_{\rho,j}(x)}{\mu_{\rho,j}} = 8 \pi \sum_{k=1}^{m} C_j^{(k)} G(x, \xi_{k})$ in $C^1_{\mathrm{loc}}(\bar{\Omega} \setminus \lbrace \xi_1, \cdots, \xi_m \rbrace)$.
\end{itemize}
\end{teo}

By the previous result we see that the $j$-th eigenfunction also concentrates at the points $\xi_{k}$, $k=1,..,m$. If we consider higher eigenvalues, we again have concentration at $k$ points but a different behaviour occurs. In order to describe it let us denote by  $\eta_j$ , for $j=1, \ldots, 2m$, the eigenvalues of the Hessian matrix of the Hamilton function $\mathcal{F}$ at the point $( \xi_1, \ldots , \xi_m)$. We will always work under the assumption that $\eta_j \neq 0$, for $j=1, \ldots, 2m$.
Now we are in position to state the result about the behaviour of $\mu_{\rho,j}$ and $v_{\rho,j}$ for $j=m+1,..,3m$.

\begin{teo} \label{main2}
For every $j=m+1, \ldots, 3m$ there exists an $k \in \lbrace 1, \ldots, m \rbrace$ such that
\begin{itemize}
\item [1.] $\mu_{\rho,j}=1-3 \pi \rho^2 \eta_j (1+o(1))$, for $\rho$ small enough;
\item [2.] $\lim_{\rho \to 0} \tilde{v}_{\rho,j}^{(k)} = \frac{s_{1,j}^{(k)} x_1 + s_{2,j}^{(k)} x_2}{8+\vert x \vert^2}$ in $C^2_{\mathrm{loc}}(\R^2)$ for some $s=(s_{1,j}^{(k)}, s_{2,j}^{(k)})) \neq(0,0)$;
\item [3.]$\lim_{\rho \to 0} \frac{v_{\rho,j}}{\rho} = 2 \pi \sum_{k=1}^m \frac{\tau_k}{\sqrt{8}} \left( s_{1,j}^{(k)} \frac{\partial G}{\partial x_1} (x,\xi_k)+ s_{2,j}^{(k)} \frac{\partial G}{\partial x_2} (x,\xi_k) \right)$ in $C^1_{\mathrm{loc}}(\bar{\Omega} \setminus \lbrace \xi_1, \ldots, \xi_m \rbrace)$.
\end{itemize}
\end{teo}
Our final results concerns the study of $\mu_{\rho,j}$ and $v_{\rho,j}$ for $j=3m,..,4m$. This result, jointly with the previous theorem, is crucial for the computation of the Morse index of the solution.

Finally, we have

\begin{teo} \label{main3}
For every $j=3m+1, \ldots, 4m$ there exists an $k \in \lbrace 1, \ldots, m \rbrace$ such that
\begin{itemize}
\item [1.] $\mu_{\rho,j}=1-\frac{3}{2}\frac{1}{\log(\rho)}(1+o(1))$, for $\rho$ small enough;
\item [2.] $\lim_{\rho \to 0} \tilde{v}_{\rho,j}^{(k)}(x) = t_j^{(k)} \frac{8-\vert x \vert^2}{8+ \vert x \vert^2}$ in $C^2_{\mathrm{loc}}(\R^2)$ for some $t_j^{(k)} \neq 0$; 
\item [3.] $\log\rho\ v_{\rho,j}(x) \to 2 \pi \sum_{k=1}^m t_j^{(k)} G(x,\xi_{k})$ in $C^1_{\mathrm{loc}}(\bar{\Omega} \setminus \lbrace \xi_{1}, \ldots, \xi_{m} \rbrace)$.
\end{itemize}
\end{teo}

Now let us denote by $\mathcal{M}(u_{\rho})$ the Morse index of the solution $u_\rho$. We therefore get our main result.

\begin{teo} \label{finefinefine}
For a sufficiently small $\rho$ we have
$$\mathcal{M}(u_{\rho})=3m-\mathcal{M}(\mathrm{Hess} \,\mathcal{F}),$$
where by $\mathcal{M}(\mathrm{Hess} \,\mathcal{F})$ we denote the number of negative eigenvalues of the Hessian matrix of $\mathcal{F}$ at the point $( \xi_1, \ldots , \xi_m)$.
\end{teo}

From this, we also deduce that:

\begin{cor} \label{corol}
For a sufficiently small $\rho$, we have that
$$m \leqslant \mathcal{M}(u_{\rho})\leqslant 3m.$$
\end{cor}

It is interesting to notice that the above results are similar to those obtained in \cite{GF} where the Morse index for positive multiple blow-up solutions to the Gelfand problem is calculated . The biggest difference between this work and \cite{GF} is that the solutions to the Gelfand's problem can not have sign changing blow-up solutions, while this type of solutions are possible in our case. 

Another problem where some similarities occur  is given by 
\begin{equation}\label{dip} 
\begin{cases} -\Delta u = \vert u \vert^{p-1}u & \text{ in } B\\
u=0 & \text{ on } \partial B, 
\end{cases}
\end{equation}
where $B$ is the ball of radius one centered at the origin and $p>1$ is large.
A comparison of our results with those obtained in \cite{fil} shows how different they are. In our case, the Morse index of a solution depends on the Morse index of the Hamiltonian function, while in \cite{fil} is proved that if $u$ is the least energy sign-changing radial solution to \eqref{dip} then its Morse index is twelve.

Finally, we want to mention the results obtained in \cite{bar} where the equation
\begin{align*}
\begin{cases} -\Delta u = f(u) & \text{ in } \Omega\\
u=0 & \text{ on } \partial \Omega, 
\end{cases}
\end{align*}
is studied. In \cite{bar} are given proper hypothesis on the regularity and the growth rate of the non-linear part $f$ for the existence of sign changing solutions with Morse index at most equal to one and for sign changing solution with Morse index equal to two.

The paper is organized as follows: in Section 2, we recall some known facts on
problem \ref{iniz} and we introduce some notations. In Section 3, we prove Theorem \ref{main1}. In Section 4, we study the asymptotic behaviour of the eigenvalues and eigenfunctions from $m+1$ to $3m$ and we prove some additional properties that hold for the eigenfunctions and eigenvalues from $m+1$ to $4m$. Finally, in Section 5 we prove Theorem \ref{main2}, Theorem \ref{main3} and Corollary \ref{corol}.

\section{Preliminaries and Notations}

In the following work, we will denote by $C$ a constant which may possibly change from step to step, with $o(1)$ a function which  goes to zero  in a suitable function space (which we will specify every time) and with $\Omega$ a $C^2$ bounded smooth domain in $\R^2$ or a convex polygon with corner points $\lbrace \zeta_1, \ldots, \zeta_n \rbrace \subset \partial \Omega$. Moreover, we will let $\epsilon > 0$ be such that $B_{2\epsilon}(\xi_{\rho,i}) \subset \Omega$ and $ B_{2\epsilon}(\xi_{\rho,i}) \cap B_{2\epsilon}(\xi_{\rho,j}) =\emptyset $ for any $i,j=1, \ldots, m$ and $i \neq j$, and for $\rho > 0$ sufficiently small.

We will denote by $\psi_{\rho,i}$ a function in $C^{\infty}(\R^2)$ such that $0 \leqslant \psi_{\rho,i}(x) \leqslant 1$ and $\psi_{\rho,i}(x)=1$ for each $x$ in $B_{\epsilon}(\xi_{\rho, i})$, and $\psi_{\rho,i}(x)=0$ for each $x$ in $\R^2 \setminus B_{2\epsilon}(\xi_{\rho,i})$.

Furthermore, we will denote by $U_{\tau, \xi}$ the function
$$U_{\tau, \xi}=\log\frac{8\tau^2}{\left(\tau^2 \rho^2 + \vert x-\xi \vert^2 \right)^2},$$
which verifies
\begin{equation} \label{ese}
\begin{cases} -\Delta U(x)= \rho^2 e^{U(x)}= \frac{8 \tau^2 \rho^2}{\left(\tau^2 \rho^2+ \vert x-\xi\vert^2 \right)^2} & \hbox{in }\R^2\\
\int_{\R^2} e^{U} <\infty.
\end{cases}
\end{equation}

From now on, $G(\cdot , \cdot)$ and $H(\cdot , \cdot )$ will denote respectively the Green  function and its regular part of $-\Delta$ in $\Omega$ with zero Dirichlet boundary condition. Moreover, we will denote by $R(x)=H(x,x)$ the Robin function.

We will consider the equation
\begin{equation} \label{ese}
\begin{cases} -\Delta u = \rho^2 (e^{u} - e^{-u}) & \text{ in } \Omega\\
u=0 & \text{ on } \partial \Omega \\\end{cases}
\end{equation}
and the Hamiltonian function
$$\mathcal{F}(x_1, \ldots, x_m)=\frac{1}{2} \sum_{k=1}^m R(x_k) + \frac{1}{2} \sum_{1\leqslant k,j \leqslant m} \alpha_k \alpha_j G(x_k,x_j), \quad \alpha_k \in \lbrace 1, -1 \rbrace.$$

\begin{defin} \label{def}
Given an open set $U \subset \R^k$, $F:U \to \R$  a $C^1$ function and $K \subset U$ e a bounded set of critical points for $F$ we say that $K$ is \emph{$C^1$-stable} for $F$ if for any $F_n \to F$ in $C^1(U)$, there exists at least one critical point $y_n \in U$ for $F_n$, and $y \in K$, such that $y_n \to y$ as $n \to \infty$.
\end{defin}

\begin{oss} It can be verified that a set $K$ of critical points for $F$ is $C^1$-stable if either one of the following conditions is satisfied:
\begin{itemize}
\item[(i)] $K$ is either a strict local maximum or a strict local minimum set for $F$;
\item[(ii)] the Brouwer degree $\mathrm{deg}(\nabla F, B_{\epsilon}(K),0)$ is non-zero, for any $\epsilon >0$ small enough, where $B_{\epsilon}(K)= \lbrace x \in U \, : \, \vert x-K \vert \leqslant \epsilon \rbrace$.
\end{itemize}
\end{oss}

The following result holds.

\begin{teo}[{\cite[Theorem 1.1]{io}}]\label{ang}
Let $(\xi_{1}, \ldots, \xi_{m})$ be a $C^1$-stable critical point for the function $\mathcal{F}$. Then, there exist $\rho_0>0$ such that for any $\rho \in (0, \rho_0 )$, the equation \eqref{ese} has a solution $u_{\rho}$ such that 
$$u_{\rho}(x) \to 8 \pi \sum_{k=1}^m \alpha_k G(x, \xi_{k}) \quad \textit{ as } \rho \to 0$$
in $C^{\infty}_{\mathrm{loc}} \left(\Omega \setminus \lbrace \xi_1, \ldots, \xi_m \rbrace \right) \cap C^{1, \sigma}_{\mathrm{loc}} \left(\bar{\Omega} \setminus \lbrace \xi_1, \ldots, \xi_m \rbrace \right)$, for some $\sigma \in (0,1)$ and $\alpha_k \in \lbrace 1, -1 \rbrace$.
\end{teo}

In \cite{io} it is also proved that the solution of the theorem above has the form
\begin{equation} \label{sol}
u_{\rho}(x)= \sum_{k=1}^m \alpha_k PU_{\tau_{\rho,k}, \xi_{\rho,k}}(x) + \varphi_{\rho}(x),
\end{equation}  
where we denoted by $PU_{\tau_{\rho,k}, \xi_{\rho,k}}$ the projection of $U_{\tau_{\rho,k}, \xi_{\rho,k}}$ onto $H_0^1(\Omega)$. The parameter point $(\xi_{1,\rho}, \ldots, \xi_{m,\rho})\in \Omega^m$ converges to $(\xi_{1}, \ldots, \xi_{m})\in \Omega^m$ and 
$$\tau_{\rho,k}= \frac{e^{4 \pi \left[ H(\xi_{\rho,k}, \xi_{\rho,k})+\sum_{i=1,i \neq k}^m \alpha_k \alpha_i G(\xi_{\rho,k}, \xi_{\rho,i})\right]}}{\sqrt{8}} \to  \frac{e^{4 \pi \left[ H(\xi_{k}, \xi_{k})+\sum_{i=1,i \neq k}^m \alpha_k \alpha_i G(\xi_{k}, \xi_{i})\right]}}{\sqrt{8}}= \tau_k$$
as $\rho$ goes to zero. For this reason, $\tau_{\rho,k}$ will be called inappropriately $\tau_k$. Let's set $U_k=U_{\tau_k, \xi_k}$. 

We have:

\begin{teo}[{\cite[Theorem 5.2]{io}},{\cite[Theorem 3.2]{io}},{\cite[Theorem 5.1]{io}}] \label{unoo}
Let $\varphi_{\rho}$ be as in \eqref{sol}. Then 
$$\lim_{\rho \to 0} \varphi_{\rho }=0$$
in $C^{\infty}_{\mathrm{loc}} (\Omega \setminus \lbrace \xi_1, \ldots, \xi_m \rbrace) \cap C^{1, \sigma} (\bar{\Omega} \setminus \lbrace \xi_1, \ldots, \xi_m \rbrace) \cap C^{0, \alpha} ( \bar{\Omega} ) \cap W^{2,q}(\Omega) $, for a suitable $\sigma \in (0,1)$, any $\alpha \in (0,1)$, and any $q \in [1,2)$. Moreover, for any $p \in(1, \frac{4}{3})$, there exist $\rho_0>0$ and $R>0$ such that, for any $\rho \in (0,\rho_0)$, we have $\Vert \varphi_{\rho} \Vert_{H_0^1(\Omega)} \leqslant R \rho^{\frac{2-p}{p}} \vert \log(\rho) \vert$. Finally, 
let $PU_{k}$ be as in \eqref{sol}. Then
\begin{align*}
PU_{k}(x) &= U_{k}(x) + 8 \pi H(x, \xi_{\rho,k}) - \log (8 \tau_k^2)  + O(\rho^2)\\
&= \log \left( \frac{1}{( \tau_k^2 \rho^2+ \vert x-\xi_{\rho,k} \vert^2)^2} \right) + 8 \pi H(x, \xi_{\rho,k})+  O(\rho^2)
\end{align*}
in $C^{\infty}_{\mathrm{loc}} (\Omega) \cap C^{1, \sigma}(\bar{\Omega})$.

\end{teo}

In the following, we will consider $\beta$ to be a fixed constant in $(\frac{1}{2},1)$ and we will call $\xi_{\rho,1}, \ldots, \xi_{\rho,m}$ the blow-up points.

We are interested in calculating the Morse index of the solutions to \eqref{ese}. 
The equation for the $j$-th eigenvalue and the $j$-th eigenfunction of the linearised equation of \eqref{ese} around $u_{\rho}$ of the form \eqref{sol} is
\begin{equation} \label{autovvv}
\begin{cases} -\Delta v_{\rho,j} = \mu_{\rho,j} \rho^2 (e^{u_{\rho}} + e^{-u_{\rho}})v_{\rho,j} & \text{ in } \Omega \\
v_{\rho,j}=0 & \text{ on } \partial \Omega \\
\Vert v_{j, \rho}(x) \Vert_{L^{\infty}(\Omega)}=1. \\\end{cases}
\end{equation}

Let $\chi_A$ be the characteristic function of the set $A$. We will denote by \[ \tilde{v}_{\rho,j}^{(k)} = \chi_{B_{\frac{\sqrt{8} \epsilon}{\tau_k \rho}}(0)} (x) v_{\rho,j}\left(\frac{\tau_k \rho x}{\sqrt{8}}+ \xi_{\rho,k} \right) \]
the rescaled function of the $j$-th eigenfunction around $\xi_{\rho,k}$.
We have
\begin{align*}
u_{\rho}(x)&= \sum_{i =1}^m \alpha_i \left( -2\log\left(\tau_i^2 \rho^2 + \vert x-\xi_{\rho, i} \vert^2 \right) + 8 \pi H(x,\xi_{\rho,i}) \right) +\varphi_{\rho}(x)+o(1)\\
&=-2 \alpha_k \log\left(\tau_k^2 \rho^2+ \vert x-\xi_k \vert^2 \right)+ \ell_{\rho,k}(x) + \varphi_{\rho}(x)
\end{align*}
in $C^{\infty}_{\mathrm{loc}} (\Omega) \cap C^{1, \sigma}(\bar{\Omega})$, where
$$\ell_{\rho,k}(x)=-2 \sum_{i=1,i \neq k}^m \alpha_i \log\left(\tau_i^2 \rho^2 + \vert x-\xi_{\rho, i} \vert^2 \right) + 8 \pi \sum_{i=1}^m \alpha_i H(x,\xi_{\rho,i})+o(1)$$
in $C^{\infty}_{\mathrm{loc}} (\Omega) \cap C^{1, \sigma}(\bar{\Omega})$. Moreover,
\begin{align*}
\chi_{B_{\frac{2\sqrt{8} \epsilon}{\tau_k \rho}}(0)} (x) \ell_{\rho,k}\left(\frac{\tau_k \rho x}{\sqrt{8}}+ \xi_{\rho,k} \right)= \ell_k+o(1)
\end{align*}
in $C^{\infty}_{\mathrm{loc}} (\R^2) \cap C^{1, \sigma}(\R^2)$, where
\begin{align*}
\ell_k&=- 4 \sum_{i=1,i \neq k}^m \alpha_i \log(\vert \xi_k-\xi_i \vert ) + 8 \pi \sum_{i=1}^m \alpha_i H(\xi_i, \xi_k)=8 \pi \alpha_k \left( H(\xi_k, \xi_k ) + \sum_{i=1,i \neq k}^m \alpha_k \alpha_i G (\xi_i, \xi_k) \right).
\end{align*} 
Furthermore, 
\begin{align*}
\chi_{B_{\frac{2\sqrt{8} \epsilon}{\tau_k \rho}}(0)} (x) u_{\rho}\left(\frac{\tau_k \rho x}{\sqrt{8}}+ \xi_{\rho,k} \right)=-4 \alpha_k \log(\tau_k \rho) -2 \alpha_k \log \left( 1+ \frac{\vert x \vert^2}{8} \right) + \ell_{k}(x) + o(1)
\end{align*}
in $C^{0,\alpha}_{\mathrm{loc}}(\R^2)$.
For the rescaled function of the derivative of $u_{\rho}$ we have
\begin{align*}
&\chi_{B_{\frac{2\sqrt{8} \epsilon}{\tau_k \rho}}(0)} (x) \frac{\partial u_{\rho}}{\partial x_i}\left(\frac{\tau_k \rho x}{\sqrt{8}}+ \xi_{\rho,k} \right)\\
&= - \frac{4 \alpha_k}{\sqrt{8}\tau_k \rho} \frac{x_i}{\left( 1+ \frac{\vert x \vert^2}{8} \right)} + \frac{\partial \ell_{\rho,k}}{\partial x_i}\left(\frac{\tau_k \rho x}{\sqrt{8}}+ \xi_{\rho,k} \right) + \frac{\partial \varphi_{\rho}}{\partial x_i}\left(\frac{\tau_k \rho x}{\sqrt{8}}+ \xi_{\rho,k} \right) +o(1) 
\end{align*}
in $C^{\infty}_{\mathrm{loc}}(\R^2) \cap C^{0,\sigma}(\R^2)$. 
Finally,
\begin{align*}
\chi_{B_{\frac{\sqrt{8} \epsilon}{\tau_k \rho}}(0)} (x)e^{\pm u_{\rho} \left(\frac{\tau_k \rho x}{\sqrt{8}}+ \xi_{\rho,k} \right)}=\left(\frac{\tau_k^2 \rho^4}{8}  \left(1+ \frac{\vert x \vert^2}{8}\right)^2 \right)^{\mp \alpha_k} \left( 1+o(1) \right)
\end{align*}
in $C^{0,\alpha}(\R^2)$.

The previous estimates allow to use the Lebesgue convergence Theorem in different situations.

\section{Eigenvalues from $1$ to $m$}

In this chapter, we will study the behaviour of the first $m$ eigenvalues and eigenfunctions. In particular, we will prove that the first $m$ eigenvalues go to zero as $\rho$ goes to zero, and hence the Morse index is greater or equal than $m$. In addition, we will provide an estimate for the asymptotic behaviour of the first $m$ eigenfunctions, which will also be of fundamental importance in the next chapter.

We start by studying the asymptotic behaviour of the eigenfunctions rescaled around the blow-up points. We prove the following.
\begin{prop} \label{t1}
The function $\tilde{v}_{\rho,j}^{(k)}=\chi_{B_{\frac{\epsilon}{\tau_k \rho}}(0)}(x) v \left( \frac{\tau_k \rho x}{\sqrt{8}} + \xi_{\rho,k} \right)$, converges in $C^2_{\mathrm{loc}}(\R^2)$ to the function $\tilde{v}^{(k)}_{j}$, solution to the equation
\begin{equation} \label{eqlim}
-\Delta \tilde{v}_{j}^{(k)}(x) = \mu_{j} \frac{1}{\left(1+\frac{\vert x \vert^2}{8}\right)^2} \tilde{v}_{j}^{(k)}(x), \quad x \in \R^2,
\end{equation}
where $\mu_j= \lim_{\rho \to 0} \mu_{\rho,j}$.
Furthermore, there exists at least an $k \in \lbrace 1, \ldots, m \rbrace$ such that $\tilde{v}_j^{(k)} \neq 0$.
\end{prop}
\begin{proof}
For any $x \in B_{\frac{\epsilon}{\tau_k \rho}(0)}$ we have
\begin{equation} \label{eq1}
\begin{split}
&-\Delta \tilde{v}^{(k)}_{\rho,j} (x)\\
&= \mu_{\rho,j} \frac{\tau_k^2 \rho^4}{8} \left(e^{u_{\rho} \left( \frac{\tau_k \rho x}{\sqrt{8}} + \xi_{\rho,k} \right)} + e^{-u_{\rho}\left( \frac{\tau_k \rho x}{\sqrt{8}} + \xi_{\rho,k} \right) } \right) v_{\rho,j}\left( \frac{\tau_k \rho x}{\sqrt{8}} + \xi_{\rho,k} \right)\\
&=\mu_{\rho,j} e^{o(1)}\left(\frac{1}{\left(1+\frac{\vert x \vert^2}{8} \right)^2} + \frac{\tau_k^4  \rho^8}{64} \left(1+\frac{\vert x \vert^2}{8} \right)^2 \right) \tilde{v}_{\rho,j}^{(k)} (x).\\
\end{split}
\end{equation}
Note that
\begin{equation} \label{sti}
\frac{\tau_k^4  \rho^8}{64} \left(1+\frac{\vert x \vert^2}{8} \right)^2 \leqslant C \rho^8 \frac{1}{\rho^4}  =o(1), \quad \forall x \in B_{\frac{\epsilon}{\tau_k \rho}(0)},
\end{equation}
and
\begin{equation}
\begin{split}
&\Vert \tilde{v}^{(k)}_{\rho,j} \Vert^2_{H_0^1(\R^2)}\\
&=- \int_{B_{\frac{\epsilon}{\tau_k \rho}}(0)}  \tilde{v}^{(k)}_{\rho,j}(x) \Delta \tilde{v}^{(k)}_{\rho,j}(x) dx\\
&=\mu_{\rho,j} e^{o(1)} \int_{B_{\frac{\epsilon}{\tau_k \rho}}(0)}\frac{ \left(\tilde{v}_{\rho,j}^{(k)} (x)\right)^2}{\left(1+\frac{\vert x \vert^2}{8} \right)^2} dx +\\
& +\mu_{\rho,j} e^{o(1)} \frac{\tau_k^4  \rho^8}{64} \int_{B_{\frac{\epsilon}{\tau_k \rho}}(0)} \left(1+\frac{\vert x \vert^2}{8} \right)^2 \left( \tilde{v}_{\rho,j}^{(k)} (x)\right)^2 dx\\
&\leqslant \mu_{\rho,j} e^{o(1)} \int_{B_{\frac{\epsilon}{\tau_k \rho}}(0)}\frac{1}{\left(1+\frac{\vert x \vert^2}{8} \right)^2} dx + \mu_{\rho,n} e^{o(1)} \frac{\tau_k^4  \rho^8}{64} \int_{B_{\frac{\epsilon}{\tau_k \rho}}(0)} \left(1+\frac{\vert x \vert^2}{8} \right)^2 dx\\
&\hbox{using \eqref{sti}}\\
&\leqslant C,
\end{split}
\end{equation}
Since the right hand side \eqref{eq1} is bounded in $\R^2$ uniformly in $\rho$, and since $\lbrace \nabla v_{\rho,j} \rbrace$ is bounded in $L^2(\R^2)$, by the standard theory of elliptic equations, we have that $\tilde{v}_{\rho,j}^{(k)} \to \tilde{v}_{j}^{(k)}$ in $C^2_{\mathrm{loc}}(\R^2)$, where $\tilde{v}_{j}^{(k)}$ is as in \eqref{eqlim}.

For the proof of the second part of the theorem we refer to \cite[Proposition 2.11]{GF}.
\end{proof}
We can now give an estimate for the first $m$ eigenvalues.
\begin{prop} \label{t2}
We have that 
$$\mu_{\rho,1} < -\frac{1}{2\log(\rho)}.$$
\end{prop}
\proof
Using the formula for the Rayleigh quotient we have
\begin{equation} \label{eq2}
\begin{split}
\mu_{\rho,1} &= \inf_{v \in H_0^1(\Omega), v \neq 0} \frac{\int_{\Omega} \vert \nabla v \vert^2 dx }{\rho^2 \int_{\Omega} \left(e^{u_{\rho}(x)} + e^{-u_{\rho}(x)} \right) v^2(x) dx} \\
&\leqslant \frac{\int_{\Omega} \vert \nabla (U_1(x) \psi_{\rho,1}(x)) \vert^2 dx }{\rho^2 \int_{\Omega}\left(e^{u_{\rho}(x)} + e^{-u_{\rho}(x)} \right) U_1^2(x) \psi_{\rho,1}^2(x) dx}.
\end{split}
\end{equation}
We start by estimating the numerator.
\begin{equation} \label{eq3}
\begin{split}
&\int_{\Omega} \left| \nabla (U_k(x) \psi_{\rho,k}(x)) \right|^2 dx \\
&=-2\int_{\Omega} \nabla U_k(x) \nabla \psi_{\rho,k}(x) U_k(x) \psi_{\rho,k}(x) dx - \int_{\Omega} \Delta U_k(x) U_k(x) \psi_{\rho,k}^2(x)dx +\\
&+ \int_{\Omega} U_k^2(x) \vert \nabla \psi_{\rho,k}(x)\vert^2 dx + 2 \int_{\Omega} \nabla U_k(x) \nabla \psi_{\rho,k}(x) U_k(x) \psi_{\rho,k} dx\\
&\leqslant \int_{B_{2\epsilon}(\xi_{\rho,k})} \frac{8 \tau_k^2 \rho^2}{\left(\tau_k^2 \rho^2+ \vert x-\xi_{\rho,k}\vert^2 \right)^2}\log\left( \frac{8\tau_k^2}{\left(\tau_k^2 \rho^2 + \vert x-\xi_{\rho,k} \vert^2 \right)^2}\right) \psi_{\rho,k}^2\left(x \right) dx +\\
&+ C\int_{B_{2\epsilon}(\xi_{\rho,k}) \setminus B_{\epsilon}(\xi_{\rho,k})} U_k^2(x) dx\\
&= \int_{B_{\frac{2\sqrt{8}\epsilon}{\tau_k \rho}}(0)} \frac{\left(-4 \log(\rho) - \log\left(\frac{\tau_k^2}{8}\right) -2\log\left(1+\frac{\vert x \vert^2}{8}\right) \right)\psi_{\rho,k}^2\left(\frac{\tau_k \rho x}{\sqrt{8}} + \xi_{\rho,k} \right)}{\left(1+\frac{\vert x \vert^2}{8}\right)^2}  dx+C\\
&\hbox{using Lebesgue theorem}\\
&=-32 \pi \log (\rho)(1 +o(1)).
\end{split}
\end{equation}
For the denominator we have
\begin{equation} \label{eq4}
\begin{split}
&\rho^2 \int_{\Omega}\left(e^{u_{\rho}(x)} + e^{-u_{\rho}(x)} \right) U_k^2(x) \psi_{\rho,k}^2(x) dx\\ 
&\geqslant \frac{\tau_k^2 \rho^4 }{8} \frac{1}{\tau_k^4 \rho^4} \int_{B_{\frac{\sqrt{8}\epsilon}{\tau_k \rho}}(0)} \frac{e^{\ell_{\rho,k}\left(\frac{\tau_k \rho x}{\sqrt{8}} + \xi_{\rho,k} \right) + \varphi_{\rho}\left(\frac{\tau_k \rho x}{\sqrt{8}} + \xi_{\rho,k} \right) }}{\left(1+\frac{\vert x \vert^2}{8} \right)^2} \\
& \left( -4 \log(\rho) - \log\left(\frac{\tau_k^2}{8}\right) -2\log\left(1+\frac{\vert x \vert^2}{8}\right) \right)^2dx+\\
&+\frac{\tau_k^2 \rho^4 }{8} \tau_k^4 \rho^4 \int_{B_{\frac{\sqrt{8}\epsilon}{\tau_k \rho}}(0)} e^{-\ell_{\rho,k}\left(\frac{\tau_k \rho x}{\sqrt{8}} + \xi_{\rho,k} \right) - \varphi_{\rho}\left(\frac{\tau_k \rho x}{\sqrt{8}} + \xi_{\rho,k} \right) }\left(1+\frac{\vert x \vert^2}{8} \right)^2 \\
& \left( -4 \log(\rho) - \log\left(\frac{\tau_k^2}{8}\right) -2\log\left(1+\frac{\vert x \vert^2}{8}\right) \right)^2dx\\
&\hbox{using Lebesgue theorem}\\
&=128 \pi \log^2 (\rho)(1 +o(1)).
\end{split}
\end{equation}
Letting $k = 1$ in \eqref{eq3} and \eqref{eq4}, and substituting in \eqref{eq2}, we get 
$$\mu_{\rho,1} \leqslant \frac{-32 \pi \log (\rho)(1 +o(1))}{128 \pi \log^2 (\rho)(1 +o(1))} = -\frac{1+o(1)}{4 \log(\rho)}< -\frac{1}{2\log(\rho)}.$$
\endproof
We have the following
\begin{prop} \label{t3}
We have
\begin{itemize}
\item [1.] $\mu_{\rho,j} <- \frac{1}{2\log(\rho)}$ for every $j=1,\ldots, m$;
\item [2.] $\tilde{v}^{(k)}_{\rho,j} \to\tilde{v}^{(k)}_{j}= C^{(k)}_j$ in $C^2_{\mathrm{loc}}(\R^2)$ for every $j=1, \ldots, m$ and for a $k$ depending on $j$;
\item [3.] $\mu_{\rho,j} \nrightarrow 0$ for every $j>m$.
\end{itemize}
\end{prop}

\proof 

We proceed by induction on $m$. By Proposition \ref{t2} we have that $\mu_{\rho,1} < -\frac{1}{2\log(\rho)}$. Let us assume that $\mu_{\rho,j} < - \frac{1}{2\log(\rho)}$ for every $j=1, \ldots, m-1$.
By Proposition \ref{t1} we have that $\tilde{v}^{(k)}_{\rho,j} \to \tilde{v}^{(k)}_{j} $ in $C^2_{\mathrm{loc}}(\R^2)$, and that $\tilde{v}^{(k)}_{j}$ is a solution to
$$-\Delta \tilde{v}_{j}^{(k)} = 0; \quad x \in \R^2.$$
Then $\tilde{v}^{(k)}_{\rho,j} \to C^{(k)}_j$ in $C^2_{\mathrm{loc}}(\R^2)$ for every $j=1, \ldots, m-1$ and every $k=1, \ldots, m$. Furthermore, we know that there exists $k \in \lbrace 1, \ldots, m \rbrace$ such that $C_j^{(k)} \neq 0$.
Set 
$$\Psi_{\rho}(x)=\sum_{k=1}^m \lambda_k U_k(x)\psi_{\rho,k}(x) + \sum_{k=1}^{m-1} a_{\rho,k} v_{\rho, k}(x),$$ 
where 
\begin{equation} \label{akk1}
a_{\rho,j} = - \frac{(\sum_{k=1}^m \lambda_k U_k \psi_{\rho,k}, v_{\rho,j})_{H_0^1(\Omega)}}{(v_{\rho,j},v_{\rho,j})_{H_0^1(\Omega)}},
\end{equation}
for $\lambda_k \in \R$, and for each $j=1, \ldots, m$.
It follows immediately that $(\Psi_{\rho}, v_{\rho,j})_{H_0^1(\Omega)}=0$ for every $j=1, \ldots, m-1$.

Let us now show that for a suitable choice of $\lambda_k$, we have $a_{\rho,j} =o\left( \log(\rho) \right)$ for every $j=1, \ldots, m-1$.
Let us start by estimating the numerator of \eqref{akk1}.
\begin{equation} \label{io1}
\begin{split}
&\frac{1}{\mu_{\rho,j}}\left( \sum_{k=1}^m \lambda_k U_k \psi_{\rho,k}, v_{\rho,j}\right)_{H_0^1(\Omega)} \\
&= \sum_{k=1}^m \lambda_k \rho^2 \frac{\tau_k^2 \rho^2}{8} \frac{1}{\tau_k^4 \rho^4} \int_{B_{\frac{2\sqrt{8}\epsilon}{\tau_k \rho}}(0)}  \frac{e^{\ell_{\rho,k}\left(\frac{\tau_k \rho x}{\sqrt{8}} + \xi_{\rho,k} \right) + \varphi_{\rho}\left(\frac{\tau_k \rho x}{\sqrt{8}} + \xi_{\rho,k} \right) }}{\left(1+\frac{\vert x \vert^2}{8} \right)^2}  \tilde{v}^{(k)}_{\rho,j}(x)\\
& \psi_{\rho,k}\left(\frac{\tau_k \rho x}{\sqrt{8}} + \xi_{\rho,k} \right) \left( -4 \log(\rho) - \log\left( \frac{\tau_i^2}{8} \right) -2\log\left(1+\frac{\vert x \vert^2}{8}\right) \right)dx+\\
&+ \sum_{k=1}^m \lambda_k \rho^2 \frac{\tau_k^2 \rho^2}{8} \tau_k^4 \rho^4 \int_{B_{\frac{2\sqrt{8}\epsilon}{\tau_k \rho}}(0)} e^{-\ell_{\rho,k}\left(\frac{\tau_k \rho x}{\sqrt{8}} + \xi_{\rho,k} \right) - \varphi_{\rho}\left(\frac{\tau_k \rho x}{\sqrt{8}} + \xi_{\rho,k} \right) } \left(1+\frac{\vert x \vert^2}{8} \right)^2 \\
&\tilde{v}^{(k)}_{\rho,j}(x)\psi_{\rho,k}\left(\frac{\tau_k \rho x}{\sqrt{8}} + \xi_{\rho,k} \right) \left( -4 \log(\rho) - \log\left( \frac{\tau_i^2}{8}\right) -2\log\left(1+\frac{\vert x \vert^2}{8}\right) \right)dx\\
&\hbox{using Lebesgue theorem}\\
&=-32 \pi \log(\rho) \left(\sum_{k=1}^m \lambda_k  C_j^{(k)} + o(1) \right).
\end{split}
\end{equation}
Hence, taking $\lambda_k$ such that $\sum_{k=1}^m \lambda_k C_j^{(k)} = 0$ for $j=1, \ldots, m-1$, we have
\begin{equation} \label{xx}
\frac{1}{\mu_{\rho,j}} \left( \sum_{k=1}^m \lambda_k  U_k \psi_{\rho,k}, v_{\rho,j}\right)_{H_0^1(\Omega)} = -32 \pi \log(\rho)o(1)= o\left( \log(\rho) \right).
\end{equation}
For the denominator we have
\begin{equation} \label{limitato}
\begin{split}
&\Vert v_{\rho,j} \Vert^2_{H_0^1(\Omega)}\\
&=\mu_{\rho,j}\Biggl( \sum_{k=1}^m \rho^2 \frac{\tau_k^2 \rho^2}{8} \frac{1}{\tau_k^4 \rho^4}\int_{B_{\frac{\sqrt{8}\epsilon}{\tau_k \rho}}(0)} \frac{e^{\ell_{\rho,k}\left(\frac{\tau_k \rho x}{\sqrt{8}} + \xi_{\rho,k} \right) + \varphi_{\rho}\left(\frac{\tau_k \rho x}{\sqrt{8}} + \xi_{\rho,k} \right) }}{\left(1+\frac{\vert x \vert^2}{8}\right)^2} \left(\tilde{v}^{(k)}_{\rho,j}(x) \right)^2 dx +o(1)\Biggr)+\\
&+\mu_{\rho,j} \Biggl( \sum_{k=1}^m \rho^2 \frac{\tau_k^2 \rho^2}{8} \tau_k^4 \rho^4\\
&\int_{B_{\frac{\sqrt{8}\epsilon}{\tau_k \rho}}(0)} e^{-\ell_{\rho,k}\left(\frac{\tau_k \rho x}{\sqrt{8}} + \xi_{\rho,k} \right) - \varphi_{\rho}\left(\frac{\tau_k \rho x}{\sqrt{8}} + \xi_{\rho,k} \right) }\left(1+\frac{\vert x \vert^2}{8}\right)^2 \left(\tilde{v}^{(k)}_{\rho,j}(x) \right)^2 dx +o(1) \Biggr) \\
& \hbox{using Lebesgue theorem}\\
&=8 \pi \mu_{j} \left(\sum_{k=1}^m (C_j^{(k)})^2+o(1)\right).
\end{split}
\end{equation}

Using \eqref{xx} and \eqref{limitato} in \eqref{akk1} we get
\begin{equation} \label{xxx}
a_{\rho,j} = \frac{o\left( \log(\rho) \right)}{8\pi \sum_{i=1}^m \left(C_j^{(i)} \right)^2 + o(1)} = o \left( \log(\rho) \right).
\end{equation}

Using the formula for the Rayleigh quotient, we have
\begin{equation} \label{app}
\begin{split}
\mu_{\rho,m} &= \inf_{\substack{v \in H_0^1(\Omega), v \neq 0, \\ v \perp v_{\rho,1}, \ldots, v \perp v_{\rho,m-1}}} \frac{\int_{\Omega} \left| \nabla v \right|^2 dx}{\rho^2 \int_{\Omega} \left(e^{u_{\rho}(x)} +e^{-u_{\rho}(x)} \right) v^2(x) dx}\\
&\leqslant \frac{\int_{\Omega} \left| \nabla \Psi_{\rho}(x) \right|^2 dx}{\rho^2 \int_{\Omega} \left(e^{u_{\rho}(x)} +e^{-u_{\rho}(x)} \right) \Psi^2_{\rho}(x)dx }\\
&= \frac{\int_{\Omega} \left| \nabla \left( \sum_{k=1}^m \lambda_k U_k(x)\psi_{\rho,k}(x) + \sum_{k=1}^{m-1} a_{\rho,k} v_{\rho, k}(x) \right) \right|^2 dx}{\rho^2 \int_{\Omega} \left(e^{u_{\rho}(x)} +e^{-u_{\rho}(x)} \right) \left(\sum_{k=1}^m \lambda_k U_k(x)\psi_{\rho,k}(x) + \sum_{k=1}^{m-1} a_{\rho,k} v_{\rho, k}(x) \right)^2 dx }
\end{split}
\end{equation}

Let us start by studying the numerator.
\begin{equation} \label{app1}
\begin{split}
&\int_{\Omega} \left| \nabla \left( \sum_{k=1}^m \lambda_k U_k(x) \psi_{\rho,k}(x) + \sum_{k=1}^{m-1} a_{\rho,k} v_{\rho, k}(x) \right) \right|^2 dx \\
&= \sum_{k=1}^m \lambda_k^2 \int_{\Omega} \left| \nabla \left( U_k(x) \psi_{\rho,k}(x)\right) \right|^2 dx +\sum_{k=1}^{m-1} a_{\rho,k}^2 \int_{\Omega} \left| \nabla v_{\rho, k}(x) \right|^2 dx +\\
& +2 \sum_{j=1}^{m-1} a_{\rho,j} \left( \sum_{k=1}^m \lambda_k U_k \psi_{\rho,k}, v_{\rho, j} \right)_{H_0^1(\Omega)}\\
&=-32 \pi \log(\rho)\sum_{k=1}^m \lambda_k^2 (1+o(1)).
\end{split}
\end{equation}
Since by \eqref{eq3} we have
$$ \int_{\Omega} \left| \nabla \left( U_k(x) \psi_{\rho,k}(x)\right) \right|^2 dx = -32 \pi \log(\rho)(1+o(1)),$$
by \eqref{limitato} and \eqref{xxx}, and since by hypothesis $\mu_{\rho,j} < -\frac{1}{2\log(\rho)}$, we have
$$a_{\rho,k}^2 \int_{\Omega} \left| \nabla v_{\rho, k}(x) \right|^2 dx =8 \pi \mu_j a_{\rho,k}^2 \left( \sum_{k=1}^m \left(C_j^{(k)} \right)^2 + o(1) \right) = o( \log(\rho)),$$
and by \eqref{xx}, \eqref{xxx}, and since by hypothesis $\mu_{\rho,j} < -\frac{1}{2\log(\rho)}$, we have
$$a_{\rho,k} \left( \sum_{k=1}^m \lambda_k U_k \psi_{\rho,k}, v_{\rho, j} \right)_{H_0^1(\Omega)} =  a_{\rho,k} \mu_{\rho,j} \log(\rho) o(1)=o(\log(\rho)).$$

For the denominator we have
\begin{equation} \label{app2}
\begin{split}
&\rho^2 \int_{\Omega} \left(e^{u_{\rho}(x)} +e^{-u_{\rho}(x)} \right) \left( \sum_{k=1}^m \lambda_k U_k(x) \psi_{\rho,k}(x) + \sum_{k=1}^{m-1} a_{\rho,k} v_{\rho, k}(x) \right)^2 dx\\
&=\sum_{k=1}^m \lambda_k^2 \rho^2 \int_{\Omega} \left(e^{u_{\rho}(x)} +e^{-u_{\rho}(x)} \right) U_k^2(x) \psi_{\rho,k}^2(x) dx +\\
&+ \sum_{k=1}^m a_{\rho,k}^2 \rho^2 \int_{\Omega} \left(e^{u_{\rho}(x)} +e^{-u_{\rho}(x)} \right) v_{\rho, k}^2(x)  dx +2 \sum_{j=1}^{m-1} a_{\rho,j} \frac{1}{\mu_{\rho,j}} \left( \sum_{k=1}^m \lambda_k U_k \psi_{\rho,k}, v_{\rho, j} \right)_{H_0^1(\Omega)}\\
&= 128 \log^2(\rho)\sum_{k=1}^m \lambda_k^2 (1+o(1)),
\end{split}
\end{equation}

Since by \eqref{eq4} we have
$$\rho^2 \int_{\Omega} \left(e^{u_{\rho}(x)} +e^{-u_{\rho}(x)} \right) U_k(x)^2(x) \psi_{\rho,k}^2(x) dx = 128 \log^2(\rho)(1+o(1)),$$
by \eqref{limitato}, \eqref{xxx}, and since by hypothesis $\mu_{\rho,j} < -\frac{1}{2\log(\rho)}$, we have
\begin{align*}
&a^2_{\rho,k}\rho^2 \int_{\Omega} \left(e^{u_{\rho}(x)} +e^{-u_{\rho}(x)} \right) v_{\rho, k}^2(x)  dx\\
& < - 8 \pi \frac{1}{2\log(\rho)} \log^2(\rho) o(1) \left( \sum_{j=1}^m \left(C_k^{(j)} \right)^2 +o(1) \right)\\
&=o(\log(\rho)),
\end{align*}
and by \eqref{xx} and \eqref{xxx}, we have
$$\frac{a_{\rho,j}}{\mu_{\rho,j}}\left( \sum_{k=1}^m \lambda_k U_k \psi_{\rho,k}, v_{\rho,j} \right)_{H_0^1(\Omega)}=o(\log^2(\rho)).$$

Using \eqref{app1} and \eqref{app2} in \eqref{app} we get 
$$\mu_{\rho,m} \leqslant \frac{-32\log(\rho)\sum_{k=1}^m \lambda_k^2 (1+o(1))}{128\log^2(\rho)\sum_{k=1}^m \lambda_k^2 (1+o(1))}<-\frac{1}{2 \log(\rho)}.$$

Since $\mu_{\rho, j} \to 0$ for every $j=1, \ldots, m$, proceeding as in the beginning of this proof we have that $\tilde{v}^{(k)}_j \to C^{(k)}_j$ in $C^2_{\mathrm{loc}}(\R^2)$ per every $k,j = 1, \ldots, m$.

Let us finally prove that $\mu_{\rho,i} \nrightarrow 0$ for every $i >m$.
Let $i$ and $j$ be such that $\mu_{\rho,i} \to 0$ and $\mu_{\rho,j} \to 0$, and let $C_j=(C_j^{(1)}, \ldots, C_j^{(n)})$.

By the orthogonality of the eigenfunctions we get
$$0=(v_{\rho,i} , v_{\rho,j})_{H_0^1(\Omega)}= \mu_{\rho,i} \rho^2 \int_{\Omega}\left(e^{u_{\rho}(x)} +e^{-u_{\rho}(x)} \right) v_{\rho,i}(x)v_{\rho,j}(x) dx,$$
so that 
\begin{align*}
0&=\rho^2 \int_{\Omega}\left(e^{u_{\rho}(x)} +e^{-u_{\rho}(x)} \right) v_{\rho,i}(x)v_{\rho,j}(x) dx\\
&=\sum_{k=1}^m \frac{\tau_k^2 \rho^4}{8} \frac{1}{\tau_k^4 \rho^4} \int_{B_{\frac{\sqrt{8}\epsilon}{\tau_k \rho}}(0)} \frac{e^{\ell_{\rho,k}\left(\frac{\tau_k \rho x}{\sqrt{8}}+\xi_{\rho,k} \right)+\varphi_{\rho} \left(\frac{\tau_k \rho x}{\sqrt{8}}+\xi_{\rho,k} \right)}}{\left(1+\frac{\vert x \vert^2}{8}\right)^2} \tilde{v}^{(k)}_{\rho,i}(x)\tilde{v}^{(k)}_{\rho,j}(x) dx+\\
&+\sum_{k=1}^m \frac{\tau_k^2 \rho^4}{8} \tau_k^4\rho^4\\
& \int_{B_{\frac{\sqrt{8}\epsilon}{\tau_k \rho}}(0)} e^{-\ell_{\rho,k}\left(\frac{\tau_k \rho x}{\sqrt{8}}+\xi_{\rho,k} \right)-\varphi_{\rho} \left(\frac{\tau_k \rho x}{\sqrt{8}}+\xi_{\rho,k} \right)}\left(1+\frac{\vert x \vert^2}{8}\right)^2 \tilde{v}^{(k)}_{\rho,i}(x)\tilde{v}^{(k)}_{\rho,j}(x) dx+o(1)\\
&\hbox{using Lebesgue theorem}\\
&= 8\pi (C_i, C_j)_{\R^m}+ o(1).
\end{align*}
Therefore $(C_i,C_j)_{\R^m}=0$. Since in $\R^m$ there are at most $m$ orthogonal vectors, then there are at most $m$ eigenvalues which go to zero. Since in 1 we proved that the first $m$ eigenvalues go to zero, we get the conclusion.
\endproof
The next theorem gives an estimate for the asymptotic behaviour of the first $m$ eigenfunctions away from the blow-up points.
\begin{lem} \label{t4}
For every $1\leqslant j \leqslant m$ we have
$$\lim_{\rho \to 0} \frac{v_{\rho,j}(x)}{\mu_{\rho,j}} = 8 \pi \sum_{k=1}^{m} C_j^{(k)} G(x, \xi_{k})$$
in $C^1_{\mathrm{loc}}(\bar{\Omega} \setminus \lbrace \xi_1, \cdots, \xi_m \rbrace)$.
\end{lem}
\proof
By Proposition \ref{t3} we know that $\tilde{v}_{\rho,j}^{(k)} \to C_j^{(k)}$ in $C^2_{\mathrm{loc}}(\R^2)$ for every $j,k=1, \ldots,m$.
For each $x \in \bar{\Omega} \setminus \lbrace \xi_1, \ldots, \xi_m \rbrace$, using Green's representation formula we get
\begin{equation} \label{y0}
\begin{split}
\frac{v_{\rho,j}(x)}{\mu_{\rho,j}}&= \rho^2 \int_{\Omega} \left( e^{u_{\rho}(y)} +e^{-u_{\rho}(y)} \right) v_{\rho,j}(y) G(x,y) dy\\
&=\sum_{k=1}^m G(x,\xi_{k}) \rho^2 \int_{B_{\epsilon}(\xi_{\rho,k})} \left( e^{u_{\rho}(y)} +e^{-u_{\rho}(y)} \right) v_{\rho,j}(y) dy+\\
&+\sum_{k=1}^m \rho^2 \int_{B_{\epsilon}(\xi_{\rho,k})} \left( e^{u_{\rho}(y)} +e^{-u_{\rho}(y)} \right) v_{\rho,j}(y) (G(x,y)-G(x, \xi_{k})) dy + o(1).\\
\end{split}
\end{equation}
Let us study the integrals in the sums separately. For the first integral we have
\begin{equation} \label{y}
\begin{split}
&\rho^2 \int_{B_{\epsilon}(\xi_{\rho,k})} \left( e^{u_{\rho}(y)} +e^{-u_{\rho}(y)} \right) v_{\rho,j}(y) dy \\
&=\rho^2 \frac{\tau_k^2 \rho^2}{8} \frac{1}{\tau_k^4 \rho^4} \int_{B_{\frac{\sqrt{8} \epsilon}{\tau_k \rho}}(0)} \frac{e^{\ell_{\rho,k}\left(\frac{\tau_k \rho x}{\sqrt{8}}+\xi_{\rho,k} \right)+ \varphi\left(\frac{\tau_k \rho x}{\sqrt{8}}+\xi_{\rho,k} \right)}}{\left(1+\frac{\vert x \vert^2}{8}\right)^2} \tilde{v}^{(k)}_{\rho,j}(x) dx +\\
&+ \rho^2 \frac{\tau_k^2 \rho^2}{8} \tau_k^4 \rho^4 \int_{B_{\frac{\sqrt{8} \epsilon}{\tau_k \rho}}(0)} e^{-\ell_{\rho,k}\left(\frac{\tau_k \rho x}{\sqrt{8}}+\xi_{\rho,k} \right)- \varphi\left(\frac{\tau_k \rho x}{\sqrt{8}}+\xi_{\rho,k} \right)}\left(1+\frac{\vert x \vert^2}{8}\right)^2 \tilde{v}^{(k)}_{\rho,j}(x) dx\\
&\hbox{using Lebesgue theorem}\\
&=8 \pi C_j^{(k)}+o(1).
\end{split}
\end{equation}
Let us now move to the second integral. For every compact subset $K \subset \bar{\Omega} \setminus \lbrace \xi_{\rho,1}, \ldots, \xi_{\rho,m} \rbrace$ and for every $k=1,\ldots,m$, there exists $\epsilon'_k$ such that, for $\rho$ small enough, we have
$$\mathrm{dist}(K,\xi_{\rho,k})= \inf_{x \in K} \vert x-\xi_{\rho,k} \vert = \epsilon'_{\rho,k}>\epsilon'_k.$$ 
Then, for each $\lambda <\min(\epsilon, \epsilon'_1, \ldots, \epsilon'_m)$, if $x \in K$ then $x \notin B_{\lambda}(\xi_{\rho,k})$. Therefore
\begin{align*}
&\rho^2 \int_{B_{\epsilon}(\xi_{\rho,k})} \left( e^{u_{\rho}(y)} +e^{-u_{\rho}(y)} \right) v_{\rho,j}(y) (G(x,y)-G(x, \xi_{k})) dy \\
&=\rho^2 \int_{B_{\epsilon}(\xi_{\rho,k}) \setminus B_{\lambda}(\xi_{\rho,k})} \left( \frac{e^{\ell_{\rho,k}(y) + \varphi_{\rho}(y)}}{\left(\tau_k^2 \rho^2 + \vert y-\xi_{\rho,k}\vert^2\right)^2} +e^{-\ell_{\rho,k}(y) - \varphi_{\rho}(y)}\left(\tau_k^2 \rho^2 + \vert y-\xi_{\rho,k}\vert^2 \right)^2 \right)\\
& v_{\rho,j}(y) (G(x,y)-G(x, \xi_{k})) dy + \\
&+\rho^2 \int_{B_{\lambda}(\xi_{\rho,k})} \left( e^{u_{\rho}(y)} +e^{-u_{\rho}(y)} \right) v_{\rho,j}(y) (G(x,y)-G(x, \xi_{k})) dy. \\
\end{align*}

For the first of these integrals we have
\begin{align*}
&\rho^2 \Bigg| \int_{B_{\bar{\epsilon}}(\xi_{\rho,k}) \setminus B_{\lambda}(\xi_{\rho,k})} \left( \frac{e^{\ell_{\rho,k}(y) + \varphi_{\rho}(y)}}{\left(\tau_k^2 \rho^2 + \vert y-\xi_{\rho,k}\vert^2\right)^2} +e^{-\ell_{\rho,k}(y) - \varphi_{\rho}(y)}\left(\tau_k^2 \rho^2 + \vert y-\xi_{\rho,k}\vert^2 \right)^2 \right)\\
& v_{\rho,j}(y) (G(x,y)-G(x, \xi_{k})) dy \Bigg|\\
&\leqslant \left( \frac{C \rho^2}{(\rho^2 + \lambda^2)^2} + C \rho^2 \right)  \int_{B_{\bar{\epsilon}}(\xi_{\rho,k}) \setminus B_{\lambda}(\xi_{\rho,k})} \vert G(x,y) - G(x,\xi_{k}) \vert dy \\
&\hbox{since $G(x,\cdot) \in L^1(\Omega)$}\\
&\leqslant C\left( \frac{\rho^2}{(\rho^2 + \lambda^2)^2} + \rho^2  \right).
\end{align*}

Choosing $\lambda= \rho^{\gamma}$ with $\gamma < \frac{1}{2}$, this integral converges to zero, uniformly in $x \in K$.

Let us now study the second integral.
\begin{align*}
&\rho^2 \left| \int_{B_{\lambda}(\xi_{\rho,k})} \left( e^{u_{\rho}(y)} +e^{-u_{\rho}(y)} \right) v_{\rho,j}(y) (G(x,y)-G(x, \xi_{k})) dy \right| \\
&\leqslant \sup_{y \in B(\lambda)(\xi_{\rho,k})} \vert (G(x,y)-G(x, \xi_{k}))\vert \rho^2 \int_{B_{\epsilon}(\xi_{\rho,k})} \left( e^{u_{\rho}(y)} +e^{-u_{\rho}(y)} \right) dy\\
&\hbox{using the uniform continuity of $G$ in $K \times B_{\epsilon'}(\xi_{\rho,k})$ and \eqref{y}}\\
&=o(1).
\end{align*}
Therefore 
\begin{equation} \label{y1}
\rho^2 \int_{B_{\bar{\epsilon}}(\xi_{\rho,k})} \left( e^{u_{\rho}(y)} +e^{-u_{\rho}(y)} \right) v_{\rho,j}(y) (G(x,y)-G(x, \xi_{k})) dy =o(1).
\end{equation}
Using \eqref{y} and \eqref{y1} in \eqref{y0} we get
$$\frac{v_{\rho,j}(x)}{\mu_{\rho,j}} = 8 \pi \sum_{k=1}^{m}C_j^{(k)}G(x,\xi_{\rho,k}) + o(1)$$
in $C^0(\bar{\Omega} \setminus \lbrace \xi_1, \ldots, \xi_m \rbrace)$.
Using the fact that
$$\frac{\partial v_{\rho,j}}{\partial x_i} = \rho^2 \int_{\Omega} \left(e^{u_{\rho}(y)} + e^{-u_{\rho}(y)} \right) v_{\rho, j} (y) \frac{\partial G}{\partial x_i}(x,y) dy,$$
and proceeding as above, we get the estimate in $C^1(\bar{\Omega} \setminus \lbrace \xi_1, \ldots, \xi_m \rbrace)$.
\endproof
\proof[Proof of Theorem \ref{main1}] 
Point $1$ of Proposition \ref{t3} proves $1$. Point $2$ of Proposition \ref{t3} proves that
$$\lim_{\rho \to 0} \tilde{v}^{(k)}_{\rho,j} = C^{(k)}_j,$$
while Proposition \ref{t1} proves $C^{(k)}_j \neq  0$ for some $k \in \lbrace 1, \ldots,m\rbrace$. Finally, Lemma \ref{t4} proves $3$.
\endproof
\section{Eigenvalues from $m+1$ to $4m$}

In this chapter, we will consider the eigenvalues and the eigenfunctions from $m+1$ to $4m$. We will prove that such eigenvalues go to one as $\rho$ approaches to zero. Being interested in the eigenvalues smaller than one, we will need to determine if they go to one from below or from above. For this purpose, we will study the asymptotic behaviour of the eigenfunctions under different conditions on the behaviour of the eigenfunctions themselves, rescaled around the blow-up points $\lbrace \xi_{1}, \ldots, \xi_{m} \rbrace$. Finally, we will prove that the eigenvalues from $3m+1$ to $4m$ go to one from above. This gives that the Morse Index is smaller or equal than $3m$.

We begin by studying the asymptotic behaviour of the eigenfunctions rescaled around the blow-up points, under the assumption that the relative eigenvalues are smaller than $1+o(1)$. This condition will be proved to hold for every eigenvalue from $m+1$ to $4m$.

We prove the following:

\begin{prop} \label{uno1}
For every $j>m$, if $\mu_{\rho,j} \leqslant 1+o(1)$ then $\mu_{\rho,j} \to 1$, and
\begin{equation} \label{auto}
\tilde{v}_{\rho,j}^{(k)}(x) \to \tilde{v}_{j}^{(k)}(x) = \frac{s_{1,j}^{(k)} x_1 + s_{2,j}^{(k)} x_2}{8+\vert x \vert^2} + t_j^{(k)} \frac{8-\vert x \vert^2}{8+ \vert x \vert^2},
\end{equation}

in $C^2_{\mathrm{loc}}(\R^2)$, and there exists $k\in \lbrace 1, \ldots, m \rbrace$ such that $(s_{1,j}^{(k)}, s_{2,j}^{(k)}, t_j^{(k)}) \neq (0,0,0)$.
\end{prop}
\proof
In Proposition \ref{t1} we proved that $\mu_{\rho,j}\to \mu_{j}$, where $\mu_j$ is an eigenvalue of \eqref{eqlim}. The proof in \cite[Lemma 4.3]{elg}, suitably modified, shows that the first eigenvalues of \eqref{eqlim} are $0$ and $1$. By assumption $\mu_{j}\leqslant1$, and by Proposition
 \ref{t3} we have that $\mu_{j} \neq 0$ for every $j>m$, so that $\mu_j=1$. In \cite[Lemma 4.3]{elg} is proved that the eigenfunctions for the eigenvalue $1$ of \eqref{eqlim} are of the form \eqref{auto}. Finally, by Proposition \ref{t1} $\tilde{v}_{\rho, j}^{(k)}$ can not be zero for all $k$. 
\endproof

Let us now prove the following lemma.

\begin{lem} \label{lemma2}
For $j$ such that $\mu_{\rho,j} \to 1$, we have that $v_{\rho,j}(x) \to 0$ in $C^1_{\mathrm{loc}}( \bar{\Omega} \setminus \lbrace \xi_1, \ldots, \xi_m \rbrace)$.
\end{lem}

\proof
Since $\mu_j \to 1$, by Proposition \ref{uno1} we have that
$$\tilde{v}_{\rho,j}^{(k)}(x) \to \tilde{v}_{j}^{(k)}(x) = \frac{s_{1,j}^{(k)} x_1 + s_{2,j}^{(k)} x_2}{8+\vert x \vert^2} + t_j^{(k)} \frac{8-\vert x \vert^2}{8+ \vert x \vert^2}$$
in $C^2_{\mathrm{loc}}(\R^2)$. Let $K \subset \bar{\Omega} \setminus \lbrace \xi_1, \ldots, \xi_m \rbrace$ be compact. Then, by Green's representation formula, for each $x \in K$ we have that
\begin{align*}
\frac{v_{\rho,j}}{\mu_{\rho,j}}(x) &= \rho^2 \int_{\Omega} \left(e^{u_{\rho}(y)} + e^{-u_{\rho}(y)} \right) v_{\rho, j} (y) G(x,y) dy\\
&=\sum_{k=1}^m G(x,\xi_k) \rho^2 \int_{B_{\epsilon}(\xi_{\rho,k})} \left(e^{u_{\rho}(y)} + e^{-u_{\rho}(y)} \right) v_{\rho, j} (y) dy +\\
&+ \sum_{k=1}^m \rho^2 \int_{B_{\epsilon}(\xi_{\rho,k})} \left(e^{u_{\rho}(y)} + e^{-u_{\rho}(y)} \right) v_{\rho, j} (y) \left( G(x,y) - G(x,\xi_k) \right) dy + o(1)\\
&=o(1)
\end{align*}
in $C_{\mathrm{loc}}(\bar{\Omega} \setminus \lbrace \xi_1, \ldots, \xi_m \rbrace)$.

In fact, for the fist integral we have
\begin{align*}
&\rho^2 \int_{B_{\epsilon}(\xi_{\rho,k})} \left(e^{u_{\rho}(x)} + e^{-u_{\rho}(x)} \right) v_{\rho, j} (x) dy\\
&=\rho^2 \frac{\tau_k^2 \rho^2}{8} \frac{1}{\tau_k^4 \rho^4}\int_{B_{\frac{\sqrt{8}\epsilon}{\tau_k \rho}}(0)} \frac{e^{\ell_{\rho,k} \left(\frac{\tau_k \rho x}{\sqrt{8}} +\xi_{\rho,k}\right) + \varphi_{\rho}\left(\frac{\tau_k \rho x}{\sqrt{8}} +\xi_{\rho,k}\right)}}{\left(1+\frac{\vert x \vert^2}{8} \right)^2} \tilde{v}^{(k)}_{\rho, j} (x) dy+\\
&+\rho^2 \frac{\tau_k^2 \rho^2}{8}\tau_k^4 \rho^4\int_{B_{\frac{\sqrt{8}\epsilon}{\tau_k \rho}}(0)} e^{-\ell_{\rho,k} \left(\frac{\tau_k \rho x}{\sqrt{8}} +\xi_{\rho,k}\right) - \varphi_{\rho}\left(\frac{\tau_k \rho x}{\sqrt{8}} +\xi_{\rho,k}\right)}\left(1+\frac{\vert x \vert^2}{8} \right)^2 \tilde{v}^{(k)}_{\rho, j} (x) dy\\
&\hbox{using Lebesgue theorem}\\
&=\int_{\R^2} \frac{1}{\left(1+ \frac{\vert x \vert^2}{8} \right)^2} \left( \frac{s_{1,j}^{(k)} x_1 + s_{2,j}^{(k)} x_2}{8+\vert x \vert^2} + t_j^{(k)} \frac{8-\vert x \vert^2}{8+ \vert x \vert^2} \right) dx + o(1)\\
&=o(1).
\end{align*}

For the second integral, proceeding as in the proof of Lemma \ref{t4}, for every $\lambda <\min\lbrace \inf_{\rho,k} \mathrm{dist}(K, \xi_{\rho,k}), \epsilon \rbrace$ we have
\begin{align*}
&\rho^2 \left| \int_{B_{\epsilon}(\xi_{\rho,k})} \left(e^{u_{\rho}(y)} + e^{-u_{\rho}(y)} \right) v_{\rho, j} (y) \left( G(x,y) - G(x,\xi_k) \right) dy \right|\\
&\leqslant C\left(\frac{\rho^2}{\left(\rho^2 + \lambda^2 \right)^2} + \rho^2 \right) \int_{B_{\epsilon}(\xi_{\rho,k}) \setminus B_{\lambda}(\xi_{\rho,k})} \vert G(x,y)-G(x,\xi_{k}) \vert dy +\\
&+ \sup_{y \in B_{\lambda}(\xi_{\rho,k})} \vert G(x,y) - G(x,\xi_k) \vert \rho^2\int_{B_{\lambda}(\xi_{\rho,k})}  \left(e^{u_{\rho}(y)} + e^{-u_{\rho}(y)} \right) dy\\
&\hbox{using the uniform continuity of $G(x,\cdot)$ in $K$, and taking $\lambda=\rho^{\gamma}$ with $\gamma < \frac{1}{2}$}\\
&=o(1)
\end{align*}
in $C_{\mathrm{loc}}(\bar{\Omega} \setminus \lbrace \xi_1, \ldots, \xi_m \rbrace)$.
Using the fact that
$$\frac{\partial v_{\rho,j}}{\partial x_i}(x) = \rho^2 \int_{\Omega} \left(e^{u_{\rho}(y)} + e^{-u_{\rho}(y)} \right) v_{\rho, j} (y) \frac{\partial G}{\partial x_i}(x,y) dy,$$
and proceeding as above, we get the estimate in $C^1_{\mathrm{loc}}(\bar{\Omega} \setminus \lbrace \xi_1, \ldots, \xi_m \rbrace)$.
We remark that this also proves that for every $j$ we have
\begin{equation} \label{stima1}
v_{\rho,j}(x) = \mu_{\rho,j} \sum_{k=1}^m G(x,\xi_{k}) \rho^2 \int_{B_{\epsilon}(\xi_{\rho,k})} \left(e^{u_{\rho}(x)} + e^{-u_{\rho}(x)}\right) v_{\rho,j}(x) dx +o(1)
\end{equation}
in $C_{\mathrm{loc}}^1(\bar{\Omega} \setminus \lbrace \xi_1, \ldots, \xi_m \rbrace)$.
\endproof

Let us now study the behaviour of the eigenfunctions away from the blow-up points, under a condition that will be proved to select the eigenfunctions from $3m+1$ to $4m$.
\begin{lem} \label{due2}
If $\mu_{\rho,j} \to 1$ and $t_j=(t_j^{(1)}, \ldots, t_j^{(m)}) \neq (0, \ldots, 0)$ then
$$\log(\rho) v_{\rho,j}(x) \to 2 \pi \sum_{k=1}^m t_j^{(k)} G(x,\xi_{k})$$
in $C^1_{\mathrm{loc}}(\bar{\Omega} \setminus \lbrace \xi_{1}, \ldots, \xi_{m} \rbrace)$.
\end{lem}
\proof 
We have
\begin{equation} \label{eqq1}
\begin{split}
&\rho^2 \int_{B_{\epsilon} (\xi_{\rho,k})} \left( e^{u_{\rho}(x)}-e^{-u_{\rho}(x)} \right) v_{\rho,j}(x) dx+o(1)\\
&=\rho^2 \int_{B_{2\epsilon} (\xi_{\rho,k})} \left( e^{u_{\rho}(x)}-e^{-u_{\rho}(x)} \right) \psi_{\rho,k}(x) v_{\rho,j}(x) dx+\\
&+ \int_{B_{2\epsilon}(\xi_{\rho,k}) \setminus B_{\epsilon}(\xi_{\rho,k})} (-\Delta \psi_{\rho,k})(x) u_{\rho}(x)v_{\rho,j}(x)dx +\\
&+ 2 \int_{B_{2\epsilon}(\xi_{\rho,k}) \setminus B_{\epsilon}(\xi_{\rho,k})} \nabla \psi_{\rho,k}(x) \nabla u_{\rho}(x)v_{\rho,j}(x)dx\\ 
&=\int_{\Omega} -\Delta (u_{\rho}(x)\psi_{\rho,k}(x)) v_{\rho,j}(x) dx \\
&=\int_{\Omega} -\Delta v_{\rho,j}(x) u_{\rho}(x) \psi_{\rho,k}(x)dx\\
&=\mu_{\rho,j} \rho^2 \int_{B_{\epsilon}(\xi_{\rho,k})} \left( e^{u_{\rho}(x)} + e^{-u_{\rho}(x)} \right) v_{\rho,j} (x) u_{\rho}(x) dx,
\end{split}
\end{equation}
where we used the fact that $\psi$ and its derivatives are bounded in $\R^2$, $\Vert v_{\rho,j} \Vert_{C(\Omega)} \leqslant 1$, $\Vert u_{\rho} \Vert_{C^1(B_{2\epsilon}(\xi_{\rho,k}) \setminus B_{\epsilon}(\xi_{\rho,k}))} \leqslant C$, were $C$ does not depend on $\rho$, and that by Lemma \ref{lemma2} we have $v_{\rho,j} \to 0$ in $C_{\mathrm{loc}}^1(\bar{\Omega} \setminus \lbrace \xi_{1}, \ldots, \xi_{m} \rbrace)$. 

Let us start by studying the first integral. By Proposition \ref{uno1} we know that
$$\tilde{v}_{\rho,j}^{(k)}(x) \to \tilde{v}_{j}^{(k)}(x) = \frac{s_{1,j}^{(k)} x_1 + s_{2,j}^{(k)} x_2}{8+\vert x \vert^2} + t_j^{(k)} \frac{8-\vert x \vert^2}{8+ \vert x \vert^2}$$
in $C^2_{\mathrm{loc}}(\R^2)$. Hence
\begin{equation} \label{eqq2}
\begin{split}
&\rho^2 \int_{B_{\epsilon} (\xi_{\rho,k})} \left( e^{u_{\rho}(x)}-e^{-u_{\rho}(x)} \right) v_{\rho,j}(x) dx\\
&=\alpha_k \rho^2 \frac{\tau_k^2 \rho^2}{8} \frac{1}{\tau_k^4 \rho^4} \int_{B_{\frac{2\sqrt{8}\epsilon}{\tau_k \rho}}(0)} \frac{e^{\ell_{\rho,k}\left(\frac{\tau_k \rho x}{\sqrt{8}} +\xi_{\rho,k}\right) + \varphi_{\rho}\left(\frac{\tau_k \rho x}{\sqrt{8}} +\xi_{\rho,k}\right)}}{\left(1+\frac{\vert x \vert^2}{8} \right)^2}  \tilde{v}_{\rho,j}^{(k)}(x)dx +\\
&-\alpha_k \rho^2 \frac{\tau_k^2 \rho^2}{8} \tau_k^4 \rho^4 \int_{B_{\frac{2\sqrt{8}\epsilon}{\tau_k \rho}}(0)} e^{-\ell_{\rho,k}\left(\frac{\tau_k \rho x}{\sqrt{8}} +\xi_{\rho,k}\right) - \varphi_{\rho}\left(\frac{\tau_k \rho x}{\sqrt{8}} +\xi_{\rho,k}\right)}\left(1+\frac{\vert x \vert^2}{8} \right)^2  \tilde{v}_{\rho,j}^{(k)}(x)dx \\
&\hbox{using Lebesgue theorem}\\
&= \alpha_k \int_{\R^2} \frac{1}{\left(1+\frac{\vert x \vert^2}{8} \right)^2} \left(\frac{s_{1,j}^{(k)} x_1 + s_{2,j}^{(k)} x_2}{8+\vert x \vert^2} + t_j^{(k)} \frac{8-\vert x \vert^2}{8+ \vert x \vert^2}\right) dx+o(1)\\
&=o(1).
\end{split}
\end{equation}

Let us now study the second integral.
\begin{equation} \label{eqq3}
\begin{split}
&\mu_{\rho,j} \rho^2\int_{B_{\epsilon}(\xi_{\rho,k})} \left( e^{u_{\rho}(x)} + e^{-u_{\rho}(x)} \right) v_{\rho,j} (x) u_{\rho}(x) dx\\
&=-4 \alpha_k \log(\tau_k \rho) \mu_{\rho,j} \rho^2 \int_{B_{\epsilon}(\xi_{\rho,k})} \left( e^{u_{\rho}(x)} + e^{-u_{\rho}(x)} \right) v_{\rho,j} (x) dx+\\
&+ \mu_{\rho,j} \rho^2 \frac{\tau_k^2 \rho^2}{8} \frac{1}{\tau_k^4 \rho^4} \int_{B_{\frac{\sqrt{8}\epsilon}{\tau_k \rho}}(0)} \frac{e^{\ell_{\rho,k} \left(\frac{\tau_k \rho x}{\sqrt{8}} + \xi_{\rho,k} \right)+ \varphi_{\rho} \left(\frac{\tau_k \rho x}{\sqrt{8}} + \xi_{\rho,k} \right)}}{\left(1+\frac{\vert x \vert^2}{8} \right)^2} \tilde{v}_{\rho,j}^{(k)}(x) \\
&\left(- 2 \alpha_k \log\left( 1 + \frac{\vert x \vert^2}{8} \right)+ \ell_{\rho,k}\left(\frac{\tau_k \rho x}{\sqrt{8}} + \xi_{\rho,k} \right) + \varphi_{\rho}\left(\frac{\tau_k \rho x}{\sqrt{8}} + \xi_{\rho,k} \right) \right) dx\\
&+ \mu_{\rho,j} \rho^2 \frac{\tau_k^2 \rho^2}{8} \tau_k^4 \rho^4 \int_{B_{\frac{\sqrt{8}\epsilon}{\tau_k \rho}}(0)} e^{-\ell_{\rho,k} \left(\frac{\tau_k \rho x}{\sqrt{8}} + \xi_{\rho,k} \right)- \varphi_{\rho} \left(\frac{\tau_k \rho x}{\sqrt{8}} + \xi_{\rho,k} \right)}\left(1+\frac{\vert x \vert^2}{8} \right)^2 \tilde{v}_{\rho,j}^{(k)}(x)\\
&\left(- 2 \alpha_k \log\left( 1 + \frac{\vert x \vert^2}{8} \right)+ \ell_{\rho,k}\left(\frac{\tau_k \rho x}{\sqrt{8}} + \xi_{\rho,k} \right) + \varphi_{\rho}\left(\frac{\tau_k \rho x}{\sqrt{8}} + \xi_{\rho,k} \right) \right) dx\\
&\hbox{using Lebesgue theorem}\\
&=-4 \alpha_k \log(\rho) (1+o(1)) \rho^2 \int_{B_{\epsilon}(\xi_{\rho,k})} \left( e^{u_{\rho}(x)} + e^{-u_{\rho}(x)} \right) v_{\rho,j} (x) dx+\\
&+ \int_{\R^2} \frac{1}{\left(1+\frac{\vert x \vert^2}{8} \right)^2} \left(\frac{s_{1,j}^{(k)} x_1 + s_{2,j}^{(k)} x_2}{8+\vert x \vert^2} + t_j^{(k)} \frac{8-\vert x \vert^2}{8+ \vert x \vert^2}\right)\\
&\left(- 2 \alpha_k \log\left( 1 + \frac{\vert x \vert^2}{8} \right)+ \ell_{k} \right) dx+o(1)\\
&=-4 \alpha_k \log(\rho) (1+o(1)) \rho^2 \int_{B_{\epsilon}(\xi_{\rho,k})} \left( e^{u_{\rho}(x)} + e^{-u_{\rho}(x)} \right) v_{\rho,j} (x) dx+ 8\pi \alpha_k \tau_j^{(k)} +o(1).
\end{split}
\end{equation}

Using \eqref{eqq2} and \eqref{eqq3} in \eqref{eqq1} we get
\begin{equation} \label{stima}
\rho^2 \int_{B_{\epsilon}(\xi_{\rho,k})} \left( e^{u_{\rho}(x)} + e^{-u_{\rho}(x)} \right) v_{\rho,j} (x) dx = \frac{2\pi t_j^{(k)} + o(1)}{\log(\rho) (1+o(1))}.
\end{equation}

Let us now use Green's representation formula to estimate $v_{\rho,j}$. Let $K \subset \bar{\Omega} \setminus \lbrace \xi_1, \ldots, \xi_m \rbrace$ be compact. Then, for each $x \in K$, we have
\begin{align*}
&\log(\rho) v_{\rho,j}(x) \\
&= \log(\rho)\mu_{\rho,j} \rho^2 \int_{\Omega} \left( e^{u_{\rho}(y)} + e^{-u_{\rho}(y)} \right) v_{\rho,j}(y) G(x,y) dy\\
&=\log(\rho) \mu_{\rho,j} \sum_{k=1}^m G(x,\xi_k) \rho^2 \int_{B_{\epsilon}(\xi_{\rho,k})} \left( e^{u_{\rho}(y)} + e^{-u_{\rho}(y)} \right) v_{\rho,j}(y) dy +\\
&+\log(\rho) \mu_{\rho,j} \sum_{k=1}^m \rho^2 \int_{B_{\epsilon}(\xi_{\rho,k})} \left( e^{u_{\rho}(y)} + e^{-u_{\rho}(y)} \right) v_{\rho,j}(y) \left( G(x,y)-G(x,\xi_k) \right) dy+ o(1)\\
&\hbox{using \eqref{stima}}\\
&=2 \pi \sum_{k=1}^m t_j^{(k)} G(x,\xi_k)  + o(1),
\end{align*}
where we used the fact that, proceeding as in the proof of Lemma \ref{t4} we can prove that
$$\left| \log(\rho) \rho^2 \int_{B_{\epsilon}(\xi_{\rho,k})} \left( e^{u_{\rho}(y)} + e^{-u_{\rho}(y)} \right) v_{\rho,j}(y) \left( G(x,y)-G(x,\xi_k) \right) dy \right|=o(1).$$
Using the fact that
$$\frac{\partial v_{\rho,j}}{\partial x_i} = \rho^2 \int_{\Omega} \left(e^{u_{\rho}(y)} + e^{-u_{\rho}(y)} \right) v_{\rho, j} (y) \frac{\partial G}{\partial x_i}(x,y) dy,$$
and proceeding as above, we get the estimate in $C^1(\bar{\Omega} \setminus \lbrace \xi_1, \ldots, \xi_m \rbrace)$.
\endproof

We have:

\begin{lem}[{\cite[Proposition 5.5]{OH}}] \label{lemma3}
For each $\xi,\in \R^2$, $R>0$ e $f,g \in C^2(\overline{B_R(\xi)})$ we have
\begin{equation} \label{eqlem}
\begin{split}
&\int_{B_R(\xi)} \left\lbrace \left[(x-\xi) \cdot \nabla f (x)\right] \Delta g(x) + \left[ (x-\xi) \cdot \nabla g (x)\right] \Delta f(x) \right\rbrace  dx\\
&= R\int_{\partial B_R(\xi)} \left( 2\frac{\partial f}{\partial \nu}(x)\frac{\partial g}{\partial \nu}(x)-\nabla f(x) \cdot \nabla g(x) \right) d \sigma(x).
\end{split}
\end{equation}
\end{lem}

In the following, we will denote by $o_R(f(R))$ a function such that\\ $\lim_{R \to 0} \frac{o_R(f(R))}{f(R)}=0$. We prove the following lemma.

\begin{lem} \label{lll}
Let $R>0$ be such that $B_R(\xi_k) \cap B_R(\xi_j)=\emptyset$ e $B_R(\xi_k) \cap B_R(\xi_i)=\emptyset$. If $k \neq j$ e $k \neq i$ then we have
$$\int_{\partial B_R(\xi_k)} \nabla G(x,\xi_i) \cdot \nabla G(x, \xi_j) d \sigma(x)= \int_{\partial B_R(\xi_k)} \frac{\partial G}{\partial \nu} (x,\xi_i) \frac{\partial G}{\partial \nu} (x, \xi_j) d \sigma(x)=o_R(1).$$
If $k = j \neq i$ or $k =i \neq j$ then we have
$$R\int_{\partial B_R(\xi_k)} \nabla G(x,\xi_i) \cdot \nabla G(x, \xi_j) d \sigma(x)=R\int_{\partial B_R(\xi_k)} \frac{\partial G}{\partial \nu} (x,\xi_i) \frac{\partial G}{\partial \nu} (x, \xi_j) d \sigma(x)= o_R(1).$$
If $k = j = i$ then we have
$$R\int_{\partial B_R(\xi_k)} \nabla G(x,\xi_i) \cdot \nabla G(x, \xi_j) d \sigma(x)=R\int_{\partial B_R(\xi_k)} \frac{\partial G}{\partial \nu} (x,\xi_i) \frac{\partial G}{\partial \nu}  (x, \xi_j) d \sigma(x)= \frac{1}{2\pi} + o_R(1).$$
\end{lem}

\proof
The case $k \neq i$ and $k \neq j$ is trivial. Let us study the case $k=i$.
$$\nabla G(x,\xi_k) = -\frac{1}{2\pi} \frac{(x-\xi_k)}{\vert x-\xi_k \vert^2} + \nabla H(x,\xi_k),$$
Then
\begin{align*}
R\int_{\partial B_R(\xi_k)} \nabla G(x,\xi_k) \cdot \nabla G(x, \xi_j) d \sigma(x)&=-\frac{R}{2\pi} \int_{\partial B_R(\xi_k)} \frac{(x-\xi_k)}{\vert x-\xi_k \vert^2} \cdot \nabla H(x,\xi_j) d\sigma(x) +\\
&+ \frac{R}{4 \pi^2} \int_{\partial B_R(\xi_k)} \frac{(x-\xi_k)\cdot(x-\xi_j)}{\vert x-\xi_k \vert^2 \vert x-\xi_j \vert^2} d\sigma(x)+o_R(1).
\end{align*}
Where we used the fact that $\vert \nabla H(x,\xi_k) \vert $ e $\left| \frac{1}{2\pi} \frac{(x-\xi_j)}{\vert x-\xi_j \vert^2} + \nabla H(x,\xi_j) \right|$ are bounded in $B_R(\xi_k)$ and hence are bounded in $\partial B_R(\xi_k)$, uniformly in $R$.
Let us study the two integrals separately.
\begin{align*}
\frac{R}{2\pi} \int_{\partial B_R(\xi_k)} \frac{(x-\xi_k)}{\vert x-\xi_k \vert^2} \cdot \nabla H(x,\xi_j) d\sigma(x)=\frac{R}{2\pi} \int_0^{2\pi} (\cos (\theta), \sin(\theta))\cdot \nabla H(x,\xi_j) d\theta=o_R(1).
\end{align*}
Where we used the fact that $\left| \nabla H(x,\xi_j) \right|$ is bounded in $\partial B_R(\xi_k)$, uniformly in $R$. For the second integral we have.
\begin{align*}
&\frac{R}{4 \pi^2} \int_{\partial B_R(\xi_k)} \frac{(x-\xi_k)\cdot (x-\xi_j)}{\vert x-\xi_k \vert^2 \vert x-\xi_j \vert^2} d\sigma(x)\\
&=\frac{R^3}{4 \pi^2 R^2} \int_0^{2\pi} \frac{(\cos(\theta), \sin(\theta))\cdot (R(\cos(\theta), \sin(\theta)+ (\xi_k-\xi_j))}{\vert R(\cos(\theta), \sin(\theta))+ (\xi_k-\xi_j) \vert^2} d\theta.\\
\end{align*}
If $j \neq k$ then we have
$$\frac{R^3}{4 \pi^2 R^2} \int_0^{2\pi} \frac{(\cos(\theta), \sin(\theta))\cdot (R(\cos(\theta), \sin(\theta)+ (\xi_k-\xi_j))}{\vert R(\cos(\theta), \sin(\theta)+ (\xi_k-\xi_j) \vert^2} d\theta=o_R(1),$$
where we used the fact that the function inside the integral is uniformly bounded in $R$.
If $j=k$ then we have
$$\frac{R^3}{4 \pi^2 R^2} \int_0^{2\pi} \frac{(\cos(\theta), \sin(\theta))\cdot (R(\cos(\theta), \sin(\theta)+ (\xi_k-\xi_j))}{\vert R(\cos(\theta), \sin(\theta)+ (\xi_k-\xi_j) \vert^2} d\theta = \frac{1}{2\pi}.$$
Analogous computations give the other identities.
\endproof

The next theorem shows that the hypotheses of Lemma \ref{due2} do not hold if the eigenvalues go to one too fast. Thanks to this theorem we will prove that the estimate provided by Lemma \ref{due2} does not holds for the eigenfunctions from $m+1$ to $3m$, but holds only for the eigenfunctions from $3m+1$ to $4m$ .

\begin{teo}  \label{tre3}
For every $j>m$ such that $\mu_{\rho,j} <1+C\rho^2$, in \eqref{auto} we have that $t_j=\left(t_j^{(1)}, \ldots, t_j^{(m)}\right) = (0, \ldots, 0)$, that is
$$\tilde{v}_{\rho,j}^{(k)} \to \frac{s_{1,j}^{(k)} x_1 + s_{2,j}^{(k)} x_2}{8+\vert x \vert^2}$$
in $C^2_{\mathrm{loc}}(\R^2)$ for some $s=(s_{1,j}^{(k)}, s_{2,j}^{(k)})) \neq \textbf{0}$. Furthermore, if $t_j \neq (0, \ldots, 0)$ then $\mu_{\rho,j}>1$.
\end{teo}
\proof
By Proposition \ref{uno1} we have that
$$\tilde{v}_{\rho,j}^{(k)} \to \frac{s_{1,j}^{(k)} x_1 + s_{2,j}^{(k)} x_2}{8+\vert x \vert^2} + t_j^{(k)} \frac{8-\vert x \vert^2}{8+ \vert x \vert^2}$$
in $C^2_{\mathrm{loc}}(\R^2)$. If $t_j=\textbf{0}$ then $s \neq \textbf{0}$, otherwise we would have that $v^{(k)}_{\rho,j} \to 0$. Let us assume by contradiction that $t_j \neq \textbf{0}$,that is $t_j^{(k)} \neq 0$ for some $k=1, \ldots, m$.
Using \eqref{eqlem} with $f=u_{\rho}$ and $g= v_{\rho,j}$ we get
\begin{align*}
& R\int_{\partial B_R(\xi_k)} \left( 2\frac{\partial u_\rho}{\partial \nu}(x)\frac{\partial v_{\rho,j}}{\partial \nu}(x)-\nabla u_{\rho}(x) \cdot \nabla v_{\rho,j}(x)\right) d \sigma(x)\\
&=\int_{B_R(\xi_k)} \left\lbrace \left[(x-\xi_k) \cdot \nabla u_\rho (x)\right] \Delta v_{\rho,j}(x) + \left[ (x-\xi_k) \cdot \nabla v_{\rho,j} (x)\right] \Delta u_{\rho}(x) \right\rbrace dx.
\end{align*}
For the left hand side, by Lemma \ref{due2} we have
\begin{align*}
&R\int_{\partial B_R(\xi_k)} \left( 2\frac{\partial u_\rho}{\partial \nu}(x)\frac{\partial v_{\rho,j}}{\partial \nu}(x)-\nabla u_{\rho}(x) \cdot \nabla v_{\rho,j}(x)\right) d \sigma(x)\\
&=\frac{32 \pi^2 R}{\log(\rho)}\int_{\partial B_R(\xi_k)} \frac{\partial }{\partial \nu} \left( \sum_{i=1}^m \alpha_i G(x,\xi_i) + o(1) \right) \frac{\partial}{\partial \nu} \left(  \sum_{n=1}^m t_j^{(n)}G(x,\xi_n)+o(1) \right) d \sigma(x)+\\
&- \frac{16 \pi^2 R}{\log(\rho)}\int_{\partial B_R(\xi_k)} \nabla \left( \sum_{i=1}^m \alpha_i G(x,\xi_i)+o(1) \right) \cdot \nabla \left( \sum_{n=1}^m t_j^{(n)}G(x,\xi_n)+o(1) \right) d \sigma(x)\\
&\hbox{applying Lemma \ref{lll}}\\
&=\frac{8 \pi}{\log(\rho)} \alpha_k t_j^{(k)} + o_R(1).
\end{align*}
Let us now study the right hand side.
\begin{align*}
&\int_{B_R(\xi_k)} \left\lbrace \left[(x-\xi_k) \cdot \nabla u_\rho (x)\right] \Delta v_{\rho,j}(x) + \left[ (x-\xi_k) \cdot \nabla v_{\rho,j} (x)\right] \Delta u_{\rho}(x)\right\rbrace dx\\
&=-\rho^2 \int_{B_R(\xi_k)} (x-\xi_k) \cdot \nabla \left( \left( e^{u_{\rho}(x)} - e^{-u_{\rho}(x)} \right) v_{\rho,j}(x) \right) dx +\\
&-(\mu_{\rho,j}-1) \rho^2 \int_{B_R(\xi_k)} \left[(x-\xi_k) \cdot \nabla u_\rho (x)\right] \left( e^{u_{\rho}(x)} + e^{-u_{\rho}(x)} \right) v_{\rho,j}(x) dx \\
&=-\rho^2 \int_{\partial B_R(\xi_k)} (x-\xi_k) \cdot \nu  \left( e^{u_{\rho}(x)} - e^{-u_{\rho}(x)} \right) v_{\rho,j}(x) d\sigma(x) +\\
&+ 2 \rho^2 \int_{B_R(\xi_k)} \left( e^{u_{\rho}(x)} - e^{-u_{\rho}(x)} \right) v_{\rho,j}(x) dx +\\
&+(1-\mu_{\rho,j}) \rho^2 \int_{B_R(\xi_k)} \left[(x-\xi_k) \cdot \nabla u_\rho (x)\right] \left( e^{u_{\rho}(x)} + e^{-u_{\rho}(x)} \right) v_{\rho,j}(x) dx.
\end{align*}
Let us study these three integrals separately.
For the first integral we have
\begin{align*}
&-\rho^2 \int_{\partial B_R(\xi_k)} (x-\xi_k) \cdot \nu \left( e^{u_{\rho}(x)} - e^{-u_{\rho}(x)} \right) v_{\rho,j}(x) d\sigma(x)\\
&=- R \rho^2 \int_{\partial B_R(\xi_k)} \left( e^{u_{\rho}(x)} - e^{-u_{\rho}(x)} \right) v_{\rho,j}(x) d\sigma(x)\\
&=o_R(1),
\end{align*}
since
\begin{align*}
&\rho^2 \left| \int_{\partial B_R(\xi_k)} \left( e^{u_{\rho}(x)} - e^{-u_{\rho}(x)} \right) v_{\rho,j}(x) d\sigma(x) \right|\\
&\leqslant \rho^2 \frac{\tau_k^2 \rho^2}{8} \frac{1}{\tau_k^4 \rho^4} \int_{\frac{\sqrt{8} R}{\tau_k \rho}(0)} \frac{e^{\ell_{\rho,k} \left(\frac{\tau_k \rho x}{\sqrt{8}} +\xi_{\rho,k} \right) + \varphi_{\rho}\left(\frac{\tau_k \rho x}{\sqrt{8}} +\xi_{\rho,k} \right)}}{\left(1+\frac{\vert x \vert^2}{8} \right)^2} dx + \\
&+\rho^2 \frac{\tau_k^2 \rho^2}{8} \tau_k^4 \rho^4 \int_{\frac{\sqrt{8} R}{\tau_k \rho}(0)} e^{-\ell_{\rho,k} \left(\frac{\tau_k \rho x}{\sqrt{8}} +\xi_{\rho,k} \right) - \varphi_{\rho}\left(\frac{\tau_k \rho x}{\sqrt{8}} +\xi_{\rho,k} \right)}\left(1+\frac{\vert x \vert^2}{8} \right)^2 dx\\
&\hbox{using Lebesgue theorem}\\
&=8\pi+o(1).
\end{align*}

For the second integral we have
\begin{align*}
& 2 \rho^2 \int_{B_R(\xi_k)} \left( e^{u_{\rho}(x)} - e^{-u_{\rho}(x)} \right) v_{\rho,j}(x) dx\\
&= \alpha_k 2 \rho^2 \frac{\tau_k^2 \rho^2}{8} \frac{1}{\tau_k^4 \rho^4} \int_{B_\frac{\sqrt{8}R}{\tau_k \rho}(0)} \frac{e^{\ell_{\rho,k} \left(\frac{\tau_k \rho x}{\sqrt{8}}+ \xi_{\rho,k} \right)+ \varphi_{\rho}\left(\frac{\tau_k \rho x}{\sqrt{8}}+ \xi_{\rho,k} \right)}}{\left( 1+\frac{\vert x \vert^2}{8}\right)^2}\tilde{v}^{(k)}_{\rho,j}(x) dx + \\
&+ \alpha_k 2 \rho^2 \frac{\tau_k^2 \rho^2}{8} \tau_k^4 \rho^4 \int_{B_\frac{\sqrt{8}R}{\tau_k \rho}(0)} e^{-\ell_{\rho,k} \left(\frac{\tau_k \rho x}{\sqrt{8}}+ \xi_{\rho,k} \right)- \varphi_{\rho}\left(\frac{\tau_k \rho x}{\sqrt{8}}+ \xi_{\rho,k} \right)}\left( 1+\frac{\vert x \vert^2}{8}\right)^2 \tilde{v}^{(k)}_{\rho,j}(x) dx\\
&\hbox{using Lebesgue theorem}\\
&= 2 \alpha_k \int_{\R^2} \frac{1}{\left( 1+\frac{\vert x \vert^2}{8}\right)^2} \left( \frac{s_{1,j}^{(k)} x_1 + s_{2,j}^{(k)} x_2}{8+\vert x \vert^2} + t_j^{(k)} \frac{8-\vert x \vert^2}{8+ \vert x \vert^2} \right)dx +o\left( \frac{1}{\log(\rho)} \right)\\
&=o\left( \frac{1}{\log(\rho)} \right).
\end{align*}

For the third integral we have
\begin{align*}
&\rho^2 \int_{B_R(\xi_k)} \left[(x-\xi_k) \cdot \nabla u_\rho (x)\right] \left(e^{u_{\rho}}(x) + e^{-u_{\rho}}(x) \right) v_{\rho,j}(x) dx\\
&= \frac{1+o(1)}{8 \tau_k^2} \int_{B_{\frac{\sqrt{8}R}{\tau_k \rho}}(\xi_k-\xi_{\rho,k})} \frac{e^{\ell_{\rho,k} \left( \frac{\tau_k \rho x}{\sqrt{8}} + \xi_{\rho,k} \right) + \varphi_{\rho} \left( \frac{\tau_k \rho x}{\sqrt{8}} + \xi_{\rho,k} \right) }}{\left( 1 + \frac{\vert x \vert^2}{8} \right)^2} \tilde{v}_{\rho,j}^{(k)}(x)\\
& \left(- \frac{\alpha_k}{2} \frac{\vert x \vert^2 }{1 + \frac{\vert x \vert^2}{8}} + \frac{\tau_k \rho}{\sqrt{8}} x \cdot \nabla \ell_{\rho,k}\left( \frac{\tau_k \rho x}{\sqrt{8}} + \xi_{\rho,k} \right) +\frac{\tau_k \rho}{\sqrt{8}} x \cdot \nabla \varphi_{\rho}\left( \frac{\tau_k \rho x}{\sqrt{8}} + \xi_{\rho,k} \right)  \right) dx + \\
&+ \frac{\tau_k^6 \rho^8}{8} (1+o(1)) \int_{B_{\frac{\sqrt{8}R}{\tau_k \rho}}(\xi_k-\xi_{\rho,k})} e^{-\ell_{\rho,k} \left( \frac{\tau_k \rho x}{\sqrt{8}} + \xi_{\rho,k} \right) - \varphi_{\rho} \left( \frac{\tau_k \rho x}{\sqrt{8}} + \xi_{\rho,k} \right) }\left( 1 + \frac{\vert x \vert^2}{8} \right)^2 \tilde{v}_{\rho,j}^{(k)}(x)\\
& \left(- \frac{\alpha_k}{2} \frac{\vert x \vert^2 }{1 + \frac{\vert x \vert^2}{8}} + \frac{\tau_k \rho}{\sqrt{8}} x \cdot \nabla \ell_{\rho,k}\left( \frac{\tau_k \rho x}{\sqrt{8}} + \xi_{\rho,k} \right) +\frac{\tau_k \rho}{\sqrt{8}} x \cdot \nabla \varphi_{\rho}\left( \frac{\tau_k \rho x}{\sqrt{8}} + \xi_{\rho,k} \right)  \right) dx\\
&\hbox{using Lebesgue theorem}\\
&=- \frac{\alpha_k}{2} \int_{\R^2} \frac{\vert x \vert^2}{\left(1+\frac{\vert x \vert^2}{8} \right)^3} \left( \frac{s_{1,j}^{(k)} x_1 + s_{2,j}^{(k)} x_2}{8+\vert x \vert^2} + t_j^{(k)} \frac{8-\vert x \vert^2}{8+ \vert x \vert^2} \right) dx + o\left( 1 \right)\\
&=\frac{16}{3} \pi \alpha_k t_j^{(k)} +o \left( 1 \right),
\end{align*}
where we used the fact that
\begin{align*}
& \rho^8 \Biggl| \int_{B_{\frac{\sqrt{8}R}{\tau_k \rho}}(\xi_k-\xi_{\rho,k})} e^{-\ell_{\rho,k} \left( \frac{\tau_k \rho x}{\sqrt{8}} + \xi_{\rho,k} \right) - \varphi_{\rho} \left( \frac{\tau_k \rho x}{\sqrt{8}} + \xi_{\rho,k} \right) }\left( 1 + \frac{\vert x \vert^2}{8} \right)^2 \tilde{v}_{\rho,j}^{(k)}(x) \frac{\tau_k \rho}{\sqrt{8}} x \cdot\\
&\cdot \nabla \varphi_{\rho}\left( \frac{\tau_k \rho x}{\sqrt{8}} + \xi_{\rho,k} \right) dx \Biggr|\\
&\leqslant C \rho^9 \int_{B_{\frac{\sqrt{8}R}{\tau_k \rho}}(\xi_k-\xi_{\rho,k})} \left( 1 + \frac{\vert x \vert^2}{8} \right)^2 \vert x \vert \left| \nabla \varphi_{\rho}\left( \frac{\tau_k \rho x}{\sqrt{8}} + \xi_{\rho,k} \right) \right| dx \\
&\hbox{using Cauchy-Schwarz}\\
&=C\rho^9 \left( \int_{B_{\frac{\sqrt{8}R}{\tau_k \rho}}(\xi_k-\xi_{\rho,k})} \left( \left( 1 + \frac{\vert x \vert^2}{8} \right)^2 \vert x \vert \right)^{2} dx \right)^{\frac{1}{2}}\\
& \left( \int_{B_{\frac{\sqrt{8}R}{\tau_k \rho}}(\xi_k-\xi_{\rho,k})} \left| \nabla \varphi_{\rho}\left( \frac{\tau_k \rho x}{\sqrt{8}} + \xi_{\rho,k} \right) \right|^{2} dx \right)^{\frac{1}{2}} \\
&\leqslant C \rho^2 \rho^{\beta} \vert \log(\rho)\vert\\
&=o\left( 1 \right),
\end{align*}
and that
\begin{align*}
&\frac{1}{8 \tau_k^2} \Biggl| \int_{B_{\frac{\sqrt{8}R}{\tau_k \rho}}(\xi_k-\xi_{\rho,k})} \frac{e^{\ell_{\rho,k} \left( \frac{\tau_k \rho x}{\sqrt{8}} + \xi_{\rho,k} \right) + \varphi_{\rho} \left( \frac{\tau_k \rho x}{\sqrt{8}} + \xi_{\rho,k} \right) }}{\left( 1 + \frac{\vert x \vert^2}{8} \right)^2} \tilde{v}_{\rho,j}^{(k)}(x)\frac{\tau_k \rho}{\sqrt{8}} x \cdot\\
&\cdot \nabla \varphi_{\rho}\left( \frac{\tau_k \rho x}{\sqrt{8}} + \xi_{\rho,k} \right) dx \Biggr| \\
&\leqslant C \rho \int_{B_{\frac{\sqrt{8}R}{\tau_k \rho}}(\xi_k-\xi_{\rho,k})} \frac{\vert x \vert}{\left( 1 + \frac{\vert x \vert^2}{8} \right)^2} \left| \nabla \varphi_{\rho}\left( \frac{\tau_k \rho x}{\sqrt{8}} + \xi_{\rho,k} \right) \right| dx \\
&\hbox{using Cauchy-Schwarz}\\
&\leqslant C \rho \left( \int_{\R^2} \left( \frac{\vert x \vert}{\left( 1 + \frac{\vert x \vert^2}{8} \right)^2} \right)^{2} dx \right)^{\frac{1}{2}} \left( \int_{B_{\frac{\sqrt{8}R}{\tau_k \rho}}(\xi_k-\xi_{\rho,k})} \left| \nabla \varphi_{\rho}\left( \frac{\tau_k \rho x}{\sqrt{8}} + \xi_{\rho,k} \right) \right|^{2} dx \right)^{\frac{1}{2}} \\
&\leqslant C \rho^{\beta} \vert \log(\rho)\vert\\
&=o\left( 1 \right).
\end{align*}
We then have 
$$(1-\mu_{\rho,j}) \left(1+o(1) \right) \frac{16}{3} \pi \alpha_k t_j^{(k)}+o\left( \frac{1}{\log(\rho)} \right)+o_R(1)=8 \pi \frac{1}{\log(\rho)} \alpha_k t_j^{(k)} + o_R(1).$$
Since this relation holds true for every $R$, we have
\begin{equation} \label{due}
(1-\mu_{\rho,j})=\left( \frac{3}{2} \frac{1}{\log(\rho)} + o\left( \frac{1}{\log(\rho)} \right) \right) (1+o(1))< \frac{2}{\log(\rho)}.
\end{equation}
We supposed by assumption that 
$$(1-\mu_{\rho,j}) \geqslant -C\rho^2,$$
so that $$-C \rho^2 \log(\rho) >2,$$
which gives a contradiction as $C \rho^2 \log(\rho)=o(1)$.
Furthermore, by \eqref{due} we have that if $t_j^{(k)} \neq \textbf{0}$ for some $k=1, \ldots, m$, then we have
$$\mu_{\rho,j}=1-\frac{3}{2}\frac{1}{\log(\rho)}+o(1)>1,$$
for $\rho$ sufficiently small.
Finally, by Proposition \ref{t1} $\tilde{v}_{\rho, j}^{(k)}$ can not be zero for all $k$.\endproof

We prove the following lemma.

\begin{lem} \label{lemma}
For any domain $\Omega' \subset \Omega$, and any eigenfunction $v_{\rho,n}$, the following integral identity holds.
\begin{equation} \label{formula}
\begin{split}
&(1-\mu_{\rho,n}) \rho^2 \int_{\Omega'} \left( e^{u_{\rho}(x)} + e^{-u_{\rho}(x)} \right) \frac{\partial u_{\rho}}{ \partial x_j}(x) v_{\rho,n}(x) dx\\ 
&= \int_{\partial \Omega'} \frac{\partial v_{\rho,n}}{\partial \nu }(x) \frac{\partial u_{\rho}}{ \partial x_j}(x) -v_{\rho,n}(x)\frac{\partial^2 u_{\rho}}{ \partial \nu \partial x_j}(x) d\sigma_x,
\end{split}
\end{equation}
where we denoted by $\nu$ the external normal vector to $\partial \Omega'$.
\end{lem}

\proof
Differentiating \eqref{ese} with respect to $x_j$, for $j=1,2$, we get
$$-\Delta \frac{\partial u_{\rho}}{ \partial x_j}(x) = \rho^2 \left( e^{u_{\rho}(x)} + e^{-u_{\rho}(x)} \right) \frac{\partial u_{\rho}}{ \partial x_j}(x), \; \forall x \text{ in } \Omega.$$
Multiplying this expression by $v_{\rho,n}$, and integrating both sides of the equation we get
\begin{align*}
&\int_{\Omega'} \nabla \left( \frac{\partial u_{\rho}}{ \partial x_j} \right)(x) \nabla v_{\rho,n}(x) dx - \int_{\partial \Omega'} v_{\rho,n}(x)\frac{\partial^2 u_{\rho}}{ \partial \nu \partial x_j}(x) d \sigma_x\\
&= \rho^2 \int_{\Omega'} \left( e^{u_{\rho}(x)} + e^{-u_{\rho}(x)} \right) \frac{\partial u_{\rho}}{ \partial x_j}(x) v_{\rho,n}(x)dx.
\end{align*}
On the other hand, multiplying \eqref{autovvv} by $\frac{\partial u_{\rho}}{ \partial x_j}$ we get
\begin{align*}
&\int_{\Omega'} \nabla v_{\rho,n} \nabla \left( \frac{\partial u_{\rho}}{ \partial x_j}\right)(x) dx - \int_{\partial \Omega'} \frac{\partial v_{\rho,n}}{ \partial \nu}(x)\frac{\partial u_{\rho}}{ \partial x_j}(x) d \sigma_x\\
&= \mu_{\rho,n} \rho^2 \int_{\Omega'} \left( e^{u_{\rho}(x)} + e^{-u_{\rho}(x)} \right) \frac{\partial u_{\rho}}{ \partial x_j}(x) v_{\rho,n}(x) dx.
\end{align*}
Taking the difference of these two expressions we get the conclusion.
\endproof

We can now prove that the hypothesis of Theorem \ref{tre3} holds for the eigenvalues from $m+1$ to $3m$, and therefore the estimate provided by Lemma \ref{due2} does not hold for the corresponding eigenfunctions. The next two theorems show that the eigenvalues from $m+1$ to $3m$ go to one, give an estimate for the rate of convergence, and prove that the eigenvalues after $3m$ are bigger than one, thus providing an estimate from above for the Morse index.
 
\begin{lem} \label{t6}
We have that $$\mu_{\rho,m+1} <1+C\rho^2$$ and $$\mu_{\rho,m+1} \to 1.$$ 
\end{lem}

\proof 
Set 
$$\Psi_{\rho}(x)=\frac{\partial u_{\rho}}{\partial x_1}(x) \psi_{\rho,k}(x) + \sum_{j=1}^{m} a_{\rho,j} v_{\rho, j}(x),$$ 
where
\begin{equation} \label{ajk}
a_{\rho,j} = - \frac{\left( \frac{\partial u_\rho}{\partial x_1} \psi_{\rho,k}, v_{\rho,j} \right)_{H_0^1(\Omega)}}{(v_{\rho,j},v_{\rho,j})_{H_0^1(\Omega)}},
\end{equation}
with $k=1$.
With this choice for the $a_{\rho,j}$'s we immediately get that $\Psi_{\rho}$ is orthogonal to $v_{\rho,j}$ for $j=1, \ldots, m$.
By Proposition \ref{t3} we have that $\mu_{\rho,k} = o(1)$ for any $k=1,\ldots, m$. Let us prove that $a_{\rho,j} =o(1)$.
Let us start by studying the numerator
\begin{equation} \label{N}
\begin{split}
&\frac{1}{\mu_{\rho,j}}\left( \frac{\partial u_\rho}{\partial x_1} \psi_{\rho,k}, v_{\rho,j} \right)_{H_0^1(\Omega)}\\
&=  \rho^2 \int_{B_{2\epsilon}(\xi_{\rho,k}) \setminus B_{\epsilon}(\xi_{\rho,k})} \left(e^{u_{\rho}(x)}+e^{-u_{\rho(x)}} \right) v_{\rho,j}(x) \frac{\partial u_\rho}{\partial x_1}(x) \psi_{\rho,k}(x) dx+\\
&+  \rho^2 \int_{B_{\epsilon}(\xi_{\rho,k})} \left(e^{u_{\rho}(x)}+e^{-u_{\rho(x)}} \right) v_{\rho,j}(x) \frac{\partial u_\rho}{\partial x_1}(x) dx.
\end{split}
\end{equation}
Let us study these integral separately. For the first integral we have
\begin{equation} \label{i1}
\rho^2 \int_{B_{2\epsilon}(\xi_{\rho,k}) \setminus B_{\epsilon}(\xi_{\rho,k})} \left(e^{u_{\rho}(x)}+e^{-u_{\rho(x)}} \right) v_{\rho,j}(x) \frac{\partial u_\rho}{\partial x_1}(x) \psi_{\rho,k}(x) dx =o(1).
\end{equation}

By Lemma \ref{lemma} and Lemma \ref{t4} we have that
\begin{equation} \label{i2}
\begin{split}
&\rho^2 \int_{B_{\epsilon}(\xi_{\rho,k})} \left(e^{u_{\rho}(x)}+e^{-u_{\rho(x)}} \right) v_{\rho,j}(x) \frac{\partial u_\rho}{\partial x_1}(x) dx \\
&=\frac{\mu_{\rho,j}}{1-\mu_{\rho,j}} \Bigg( \int_{\partial B_{\epsilon}(\xi_{\rho,k}))} \Bigg[ \frac{\partial}{\partial \nu }\left(8 \pi \sum_{k=1}^{m}C_j^{(k)}G(x,\xi_{k}) \right) \frac{\partial }{ \partial x_1}\left(8\pi \sum_{k=1}^{m} \alpha_k G(x,\xi_{k}) \right)+\\
& - 8 \pi \left( \sum_{k=1}^{m}C_j^{(k)}G(x,\xi_{k}) \right) \frac{\partial^2 }{ \partial \nu \partial x_1}\left(8\pi \sum_{k=1}^{m} \alpha_k G(x,\xi_{k}) \right) \Bigg] d\sigma_x + o(1) \Bigg)\\
&=o(1).
\end{split}
\end{equation}
Using \eqref{i1} and \eqref{i2} in \eqref{N} we get
\begin{equation} \label{NN}
\frac{1}{\mu_{\rho,j}}\left( \frac{\partial u_\rho}{\partial x_1} \psi_{\rho,k}, v_{\rho,j} \right)_{H_0^1(\Omega)} = o(1).
\end{equation}
By \eqref{limitato} we have that
\begin{equation} \label{D}
(v_{\rho,j},v_{\rho,j})_{H_0^1(\Omega)}= 8 \pi \mu_{\rho,j} \left(\sum_{k=1}^m \left(C_j^{(k)} \right)^2 +o(1) \right).
\end{equation}
Using \eqref{NN} and \eqref{D} in \eqref{ajk} we get that $a_{\rho,j} = o(1)$.
By the formula for the Rayleigh quotient we have that
\begin{equation} \label{RR2}
\begin{split}
\mu_{\rho,m} &= \inf_{\substack{v \in H_0^1(\Omega), v \neq 0,\\ v \perp v_{\rho,1}, \ldots, v \perp v_{\rho,m}}} \frac{\int_{\Omega} \vert \nabla v \vert^2 dx}{\rho^2 \int_{\Omega} \left(e^{u_{\rho}(x)} +e^{-u_{\rho}(x)} \right) v^2(x) dx}\\
&\leqslant \frac{\int_{\Omega} \vert \nabla \Psi_{\rho}(x) \vert^2 dx}{\rho^2 \int_{\Omega} \left(e^{u_{\rho}(x)} +e^{-u_{\rho}(x)} \right) \Psi^2_{\rho}(x)dx }\\
&=\frac{\int_{\Omega} \vert \nabla \left( \frac{\partial u_{\rho}}{\partial x_1}(x) \psi_{\rho,k}(x) + \sum_{j=1}^{m} a_{\rho,j} v_{\rho, j}(x) \right) \vert^2 dx}{\rho^2 \int_{\Omega} \left(e^{u_{\rho}(x)} +e^{-u_{\rho}(x)} \right) \left( \frac{\partial u_{\rho}}{\partial x_1}(x) \psi_{\rho,k}(x) + \sum_{j=1}^{m} a_{\rho,j} v_{\rho, j}(x) \right)^2 dx }.\\
\end{split}
\end{equation}
Let us start by studying the numerator.
\begin{align*}
&\int_{\Omega} \left| \nabla \left( \frac{\partial u_{\rho}}{\partial x_1}(x) \psi_{\rho,k}(x) + \sum_{j=1}^{m} a_{\rho,j} v_{\rho, j}(x) \right) \right|^2 dx\\
&=\int_{\Omega} \left| \nabla \left( \frac{\partial u_{\rho}}{\partial x_1}(x) \psi_{\rho,k}(x) \right) \right|^2 dx + \sum_{j=1}^m  a_{\rho,j}^2 \int_{\Omega} \left| \nabla v_{\rho, j}(x)  \right|^2 dx + 2 \sum_{j=1}^m a_{\rho,j} \left( \frac{\partial u_{\rho}}{\partial x_1} \psi_{\rho,k}, v_{\rho, j}\right)_{H_0^1(\Omega)} \\
&=\int_{\Omega} \left| \nabla \left( \frac{\partial u_{\rho}}{\partial x_1}(x) \psi_{\rho,k}(x) \right) \right|^2 dx + o(1)
\end{align*}
In fact, by \eqref{limitato}, and since $a_{\rho,k} = o(1)$, we have that 
\begin{equation} \label{risu1}
a_{\rho,k}^2 \int_{\Omega} \left| \nabla v_{\rho, k}(x)  \right|^2 dx = 8 \pi a_{\rho,k}^2 \mu_{j} \left( \sum_{k=1}^m \left(C_j^{(k)} \right)^2 + o(1) \right) =o(1),
\end{equation}
and by \eqref{NN}, and since $a_{\rho,k} = o(1)$, we have that
\begin{equation} \label{risu2}
a_{\rho,j} \left( \frac{\partial u_{\rho}}{\partial x_1} \psi_{\rho,k}, v_{\rho, j}\right)_{H_0^1(\Omega)} = a_{\rho,j} \mu_{\rho,j} o(1)=o(1).
\end{equation}
We remark that
\begin{align*}
-\Delta \frac{\partial u_{\rho}}{\partial x_1}(x) = \frac{\partial}{\partial x_1} \left(\rho^2 \left(e^{u_{\rho}(x)} -e^{-u_{\rho}(x)} \right) \right)= \rho^2 \left(e^{u_{\rho}(x)} + e^{-u_{\rho}(x)} \right) \frac{\partial u_{\rho}}{\partial x_1}(x), \quad \forall x \in \Omega.
\end{align*}
Let us go back to the numerator.
\begin{align*}
&\int_{\Omega} \left| \nabla \left( \frac{\partial u_{\rho}}{\partial x_1}(x) \psi_{\rho,k}(x) \right) \right|^2 dx\\
&=-2 \int_{\Omega} \frac{\partial u_{\rho}}{\partial x_1}(x) \psi_{\rho,k}(x) \nabla \frac{\partial u_{\rho}}{\partial x_1}(x) \nabla \psi_{\rho,k}(x) dx - \int_{\Omega} \Delta \left(\frac{\partial u_{\rho}}{\partial x_1} \right)(x) \frac{\partial u_{\rho}}{\partial x_1}(x) \psi_{\rho,k}^2(x) dx+\\
&+ \int_{\Omega} \frac{\partial u_{\rho}}{\partial x_1}^2(x) \left| \nabla \psi_{\rho,k}(x) \right|^2 dx + 2 \int_{\Omega} \frac{\partial u_{\rho}}{\partial x_1}(x) \psi_{\rho,k}(x) \nabla \frac{\partial u_{\rho}}{\partial x_1}(x) \nabla \psi_{\rho,k}(x) dx\\
&=\rho^2 \int_{\Omega} \left(e^{u_{\rho}(x)} + e^{-u_{\rho}(x)} \right) \frac{\partial u_{\rho}}{\partial x_1}^2(x) \psi_{\rho,k}^2(x) dx +\int_{\Omega} \frac{\partial u_{\rho}}{\partial x_1}^2(x) \left| \nabla \psi_{\rho,k}(x) \right|^2 dx,
\end{align*}
hence
\begin{equation} \label{N1}
\begin{split}
&\int_{\Omega} \left| \nabla \left( \frac{\partial u_{\rho}}{\partial x_1}(x) \psi_{\rho,k}(x) + \sum_{k=1}^{m} a_{\rho,k} v_{\rho, k}(x) \right) \right|^2 dx\\
&= \rho^2 \int_{\Omega} \left(e^{u_{\rho}(x)} + e^{-u_{\rho}(x)} \right) \frac{\partial u_{\rho}}{\partial x_1}^2(x) \psi_{\rho,k}^2(x) dx +\int_{\Omega} \frac{\partial u_{\rho}}{\partial x_1}^2(x) \left| \nabla \psi_{\rho,k}(x) \right|^2 dx + o(1).
\end{split}
\end{equation}
Let us now study the denominator
\begin{equation}\label{D1}
\begin{split}
&\rho^2 \int_{\Omega} \left(e^{u_{\rho}(x)} +e^{-u_{\rho}(x)} \right) \left( \frac{\partial u_{\rho}}{\partial x_1}(x) \psi_{\rho,k}(x) + \sum_{j=1}^{m} a_{\rho,j} v_{\rho, j}(x) \right)^2 dx\\
&= \rho^2  \int_{\Omega} \left(e^{u_{\rho}(x)} +e^{-u_{\rho}(x)} \right) \frac{\partial u_{\rho}}{\partial x_1}^2(x) \psi_{\rho,k}^2(x) dx +\\
&+ \sum_{j=1}^m a_{\rho,j}^2 \rho^2 \int_{\Omega} \left(e^{u_{\rho}(x)} +e^{-u_{\rho}(x)} \right) v_{\rho, j}^2(x) dx +2 \sum_{j=1}^m \frac{a_{\rho,j}}{\mu_{\rho,j}} \left(\frac{\partial u_{\rho}}{\partial x_1} \psi_{\rho,k} , v_{\rho, j} \right)_{H_0^1(\Omega)}\\
&=\rho^2  \int_{\Omega} \left(e^{u_{\rho}(x)} +e^{-u_{\rho}(x)} \right) \frac{\partial u_{\rho}}{\partial x_1}^2(x) \psi_{\rho,k}^2(x) dx +o(1).
\end{split}
\end{equation}
In fact, by \eqref{limitato}, and since $a_{\rho,k} = o(1)$, we have
\begin{align*}
a_{\rho,k}^2 \rho^2 \int_{\Omega} \left(e^{u_{\rho}(x)} +e^{-u_{\rho}(x)} \right) v_{\rho, k}^2(x) dx= 8 \pi a_{\rho,j}^2 \left( \sum_{k=1}^m \left(C_j^{(k)} \right)^2 +o(1)\right)=o(1),
\end{align*}
and by \eqref{NN}, and since $a_{\rho,k} = o(1)$, we have
$$\frac{a_{\rho,j}}{\mu_{\rho,j}} \left(\frac{\partial u_{\rho}}{\partial x_1} \psi_{\rho,k} , v_{\rho, j} \right)_{H_0^1(\Omega)}=o(1).$$
Using \eqref{N1} and \eqref{D1} in \eqref{RR2}, we get
\begin{align*}
\mu_{\rho,m} &\leqslant \frac{\rho^2 \int_{\Omega} \left(e^{u_{\rho}(x)} + e^{-u_{\rho}(x)} \right) \frac{\partial u_{\rho}}{\partial x_1}^2(x) \psi_{\rho,k}^2(x) dx +\int_{\Omega} \frac{\partial u_{\rho}}{\partial x_1}^2(x) \vert \nabla \psi_{\rho,k}(x) \vert^2 dx + o(1)}{\rho^2  \int_{\Omega} \left(e^{u_{\rho}(x)} +e^{-u_{\rho}(x)} \right) \frac{\partial u_{\rho}}{\partial x_1}^2(x) \psi_{\rho,k}^2(x) dx +o(1)}\\
&=1+ \frac{\int_{\Omega} \frac{\partial u_{\rho}}{\partial x_1}^2(x) \vert \nabla \psi_{\rho,k}(x) \vert^2 dx + o(1)}{\rho^2  \int_{\Omega} \left(e^{u_{\rho}(x)} +e^{-u_{\rho}(x)} \right) \frac{\partial u_{\rho}}{\partial x_1}^2(x) \psi_{\rho,k}^2(x) dx +o(1)}.
\end{align*}
It remains to prove that
$$\frac{\int_{\Omega} \frac{\partial u_{\rho}}{\partial x_1}^2(x) \vert \nabla \psi_{\rho,k}(x) \vert^2 dx + o(1)}{\rho^2  \int_{\Omega} \left(e^{u_{\rho}(x)} +e^{-u_{\rho}(x)} \right) \frac{\partial u_{\rho}}{\partial x_1}^2(x) \psi_{\rho,k}^2(x) dx +o(1)} \leqslant C \rho^2.$$
For the numerator we have
$$\int_{\Omega} \frac{\partial u_{\rho}}{\partial x_1}^2(x) \vert \nabla \psi_{\rho,k}(x) \vert^2 dx  = \int_{B_{2\epsilon}(\xi_{\rho,k}) \setminus B_{\epsilon}(\xi_{\rho,k})} \frac{\partial u_{\rho}}{\partial x_1}^2(x) \vert \nabla \psi_{\rho,k}(x) \vert^2 dx \leqslant C.$$
For the denominator we have
\begin{equation} \label{appox}
\begin{split}
&\rho^2  \int_{\Omega} \left(e^{u_{\rho}(x)} +e^{-u_{\rho}(x)} \right) \frac{\partial u_{\rho}}{\partial x_1}^2(x) \psi_{\rho,k}^2(x) dx \\
&\geqslant \rho^2 \frac{\tau_k^2 \rho^2}{8} \frac{1}{\tau_k^4 \rho^4} \int_{B_{\frac{\sqrt{8}\epsilon}{\tau_k \rho}}(0)} \frac{e^{\ell_{\rho,k} \left(\frac{\tau_k \rho x}{\sqrt{8}} + \xi_{\rho,k} \right) + \varphi_{\rho}\left(\frac{\tau_k \rho x}{\sqrt{8}} + \xi_{\rho,k} \right)}}{\left(1+ \frac{\vert x \vert^2}{8}\right)^2}\\
&\left(-  \frac{4\alpha_k}{\sqrt{8}\tau_k \rho}\frac{x_1}{1+\frac{\vert x \vert^2}{8}} +\frac{\partial \ell_{\rho,k}}{\partial x_1}  \left(\frac{\tau_k \rho x}{\sqrt{8}} + \xi_{\rho,k} \right) + \frac{\partial \varphi_{\rho}}{\partial x_1} \left(\frac{\tau_k \rho x}{\sqrt{8}} + \xi_{\rho,k} \right) \right)^2 dx \\
&\geqslant \frac{2}{\tau_k^2 \rho^2} \frac{1}{8\tau_k^2} \int_{B_{\frac{\sqrt{8}\epsilon}{\tau_k \rho}}(0)} \frac{e^{\ell_{\rho,k} \left(\frac{\tau_k \rho x}{\sqrt{8}} + \xi_{\rho,k} \right) + \varphi_{\rho}\left(\frac{\tau_k \rho x}{\sqrt{8}} + \xi_{\rho,k} \right)}}{\left(1+ \frac{\vert x \vert^2}{8}\right)^2} \frac{x_1^2}{\left( 1+\frac{\vert x \vert^2}{8} \right)^2} dx+ \\
&- \frac{\sqrt{8}\alpha_k}{\tau_k \rho} \frac{1}{8 \tau_k^2}\int_{B_{\frac{\sqrt{8}\epsilon}{\tau_k \rho}}(0)} \frac{e^{\ell_{\rho,k} \left(\frac{\tau_k \rho x}{\sqrt{8}} + \xi_{\rho,k} \right) + \varphi_{\rho}\left(\frac{\tau_k \rho x}{\sqrt{8}} + \xi_{\rho,k} \right)}}{\left(1+ \frac{\vert x \vert^2}{8}\right)^2} \frac{x_1}{1+\frac{\vert x \vert^2}{8}} \frac{\partial \ell_{\rho,k}}{\partial x_1}  \left(\frac{\tau_k \rho x}{\sqrt{8}} + \xi_{\rho,k} \right) dx+\\
&- \frac{\sqrt{8} \alpha_k}{\tau_k \rho} \frac{1}{8 \tau_k^2} \int_{B_{\frac{\sqrt{8}\epsilon}{\tau_k \rho}}(0)} \frac{e^{\ell_{\rho,k} \left(\frac{\tau_k \rho x}{\sqrt{8}} + \xi_{\rho,k} \right) + \varphi_{\rho}\left(\frac{\tau_k \rho x}{\sqrt{8}} + \xi_{\rho,k} \right)}}{\left(1+ \frac{\vert x \vert^2}{8}\right)^2} \frac{x_1}{1+\frac{\vert x \vert^2}{8}} \frac{\partial \varphi_{\rho}}{\partial x_1} \left(\frac{\tau_k \rho x}{\sqrt{8}} + \xi_{\rho,k} \right)dx \\
&\hbox{using Lebesgue theorem}\\
&=\frac{1}{\rho^2}\left(\frac{32 \pi}{3\tau_k^2} +o(1) \right),
\end{split}
\end{equation}
where we used the fact that
\begin{align*}
& \frac{1}{\rho} \left|\int_{B_{\frac{\sqrt{8}\epsilon}{\tau_k \rho}}(0)} \frac{e^{\ell_{\rho,k} \left(\frac{\tau_k \rho x}{\sqrt{8}} + \xi_{\rho,k} \right) + \varphi_{\rho}\left(\frac{\tau_k \rho x}{\sqrt{8}} + \xi_{\rho,k} \right)}}{\left(1+ \frac{\vert x \vert^2}{8}\right)^2} \frac{x_1}{1+\frac{\vert x \vert^2}{8}} \frac{\partial \varphi_{\rho}}{\partial x_1} \left(\frac{\tau_k \rho x}{\sqrt{8}} + \xi_{\rho,k} \right)dx \right|\\
&\leqslant \frac{C}{\rho} \int_{B_{\frac{\sqrt{8}\epsilon}{\tau_k \rho}}(0)} \frac{\vert x_1 \vert}{\left(1+\frac{\vert x \vert^2}{8} \right)^3} \left| \frac{\partial \varphi_{\rho}}{\partial x_1} \left(\frac{\tau_k \rho x}{\sqrt{8}} + \xi_{\rho,k} \right) \right| dx \\
&\hbox{using Cauchy-Schwarz}\\
&\leqslant \frac{C}{\rho} \left( \int_{B_{\frac{\sqrt{8}\epsilon}{\tau_k \rho}}(0)} \left(\frac{\vert x_1 \vert}{\left(1+\frac{\vert x \vert^2}{8} \right)^3} \right)^{2}dx \right)^{\frac{1}{2}}\left( \int_{B_{\frac{\sqrt{8}\epsilon}{\tau_k \rho}}(0)}  \left| \frac{\partial \varphi_{\rho}}{\partial x_1} \left(\frac{\tau_k \rho x}{\sqrt{8}} + \xi_{\rho,k} \right)  \right|^2 dx \right)^{\frac{1}{2}} \\
&\leqslant \frac{C}{\rho^2} \rho^{\beta} \log(\rho)\\
&=\frac{1}{\rho^2} o(1).
\end{align*}
Therefore
$$\frac{\int_{\Omega} \frac{\partial u_{\rho}}{\partial x_1}^2(x) \vert \nabla \psi_{\rho,k}(x) \vert^2 dx + o(1)}{ \rho^2  \int_{\Omega} \left(e^{u_{\rho}(x)} +e^{-u_{\rho}(x)} \right) \frac{\partial u_{\rho}}{\partial x_1}^2(x) \psi_{\rho,k}^2(x) dx +o(1)} \leqslant \frac{C}{\frac{1}{\rho^2} \sum_{k=1}^m \left(\frac{32\pi}{3\tau_k^2} +o(1) \right)} \leqslant C \rho^2.$$
We know by the proof of Proposition \ref{uno1} that the first two eigenvalues are zero and 1, and we know that $\mu_{\rho,m+1} \neq 0$.
\endproof

Furthermore,
\begin{prop} \label{penultimo}
We have
\begin{itemize}
\item[(i)] $\mu_{\rho,j} \leqslant 1+C\rho^2$ for every $m+1 \leqslant j \leqslant 3m$; 
\item[(ii)] $\mu_{\rho,j} \to 1$ for every $m+1 \leqslant j \leqslant 3m$ ;
\item[(iii)] $\mu_{\rho,j} >1$ for every $j>3m$.
\end{itemize}
\end{prop}
\proof
Let us start by proving (\textit{i}). Let us proceed by induction on $m$. We know that (\textit{i}) holds true for $j=m+1$. Let us assume that it holds true for $m+1 \leqslant j \leqslant 3m-1$. We show now that it holds true for $j=3m$.

Set 
$$\Psi_{\rho}(x)= \sum_{k=1}^m \left(\lambda_{1}^{(k)} \frac{\partial u_{\rho}}{\partial x_1}(x) + \lambda_{2}^{(k)}\frac{\partial u_{\rho}}{\partial x_2}(x) \right) \psi_{\rho,k}(x) + \sum_{j=1}^{3m-1} a_{\rho,j} v_{\rho, j}(x),$$ 
where
\begin{equation} \label{ak}
\begin{split}
a_{\rho,j} &= - \frac{\left( \sum_{k=1}^m \left(\lambda_{1}^{(k)} \frac{\partial u_{\rho}}{\partial x_1} + \lambda_{2}^{(k)}\frac{\partial u_{\rho}}{\partial x_2} \right) \psi_{\rho,k}, v_{\rho,j} \right)_{H_0^1(\Omega)}}{(v_{\rho,j},v_{\rho,j})_{H_0^1(\Omega)}}\\
&=- \frac{\sum_{k=1}^m \lambda_{1}^{(k)} \left( \frac{\partial u_{\rho}}{\partial x_1}  \psi_{\rho,k}, v_{\rho,j} \right)_{H_0^1(\Omega)}}{(v_{\rho,j},v_{\rho,j})_{H_0^1(\Omega)}} - \frac{\sum_{k=1}^m \lambda_{2}^{(k)} \left( \frac{\partial u_{\rho}}{\partial x_2}  \psi_{\rho,k}, v_{\rho,j} \right)_{H_0^1(\Omega)}}{(v_{\rho,j},v_{\rho,j})_{H_0^1(\Omega)}},
\end{split}
\end{equation}
for $m+1\leqslant j \leqslant3m-1$.
With this choice for the $a_{\rho,j}$'s it is immediate to see that $\Psi_{\rho}$ is orthogonal to $v_{\rho,j}$ for $1\leqslant j \leqslant 3m-1$.
We have seen in Lemma \ref{t6} that $a_{\rho,j}=o(1)$ for every $j \leqslant m$. Let us now prove that, for a suitable choice of the $\lambda_1^{k}$'s and $\lambda_2^{(k)}$'s we have $\rho a_{\rho,j} =o(1)$ for every $m+1 \leqslant j \leqslant 3m-1$. By inductive hypothesis and by Theorem \ref{tre3}, for every $m+1 \leqslant j \leqslant 3m-1$ and every $k=1, \ldots, m$ we have that
$$\tilde{v}^{(k)}_{\rho,j}\to \frac{s_{1,j}^{(k)} x_1 + s_{2,j}^{(k)} x_2}{8+\vert x \vert^2}$$
in $C^2_{\mathrm{loc}}(\R^2)$.
Let us start by studying the numerator of \eqref{ak}.
\begin{equation} \label{appo}
\begin{split}
& \left( \sum_{k=1}^m \left(\lambda_{1}^{(k)} \frac{\partial u_{\rho}}{\partial x_1} + \lambda_{2}^{(k)}\frac{\partial u_{\rho}}{\partial x_2} \right) \psi_{\rho,k}, v_{\rho,j}  \right)_{H_0^1(\Omega)}\\
&=\sum_{k=1}^m  \frac{\mu_{\rho,j}}{8\tau_k^2} \int_{B_{\frac{2\sqrt{8} \epsilon}{\tau_k \rho}}(0)} \frac{e^{\ell_{\rho,k} \left(\frac{\tau_k \rho x}{\sqrt{8}} + \xi_{\rho,k} \right) + \varphi_{\rho}\left(\frac{\tau_k \rho x}{\sqrt{8}} + \xi_{\rho,k} \right)}}{\left(1+\frac{\vert x \vert^2}{8} \right)^2} \tilde{v}^{(k)}_{\rho,j} (x) \psi_{\rho,k}\left(\frac{\tau_k \rho x}{\sqrt{8}} + \xi_{\rho,k} \right)\\
& \Bigg( - \frac{4 \alpha_k}{\sqrt{8}\tau_k \rho} \frac{\lambda_{1}^{(k)} x_1 + \lambda_{2}^{(k)} x_2}{8+\vert x \vert^2} + \left( \lambda_{1}^{(k)} \frac{\partial \ell_{\rho,k}}{\partial x_1} + \lambda_{2}^{(k)} \frac{\partial \ell_{\rho,k}}{\partial x_2} \right) \left(\frac{\tau_k \rho x}{\sqrt{8}} + \xi_{\rho,k} \right) +\\
&+ \left(\lambda_{1}^{(k)} \frac{\partial \varphi_{\rho}}{\partial x_1} + \lambda_{2}^{(k)}\frac{\partial \varphi_{\rho}}{\partial x_2} \right)  \left(\frac{\tau_k \rho x}{\sqrt{8}} + \xi_{\rho,k} \right) \Bigg) dx +\\
&+\sum_{k=1}^m \mu_{\rho,j} \frac{\tau_k^6 \rho^8}{8} \int_{B_{\frac{2\sqrt{8} \epsilon}{\tau_k \rho}}(0)} e^{-\ell_{\rho,k} \left(\frac{\tau_k \rho x}{\sqrt{8}} + \xi_{\rho,k} \right) - \varphi_{\rho}\left(\frac{\tau_k \rho x}{\sqrt{8}} + \xi_{\rho,k} \right)}\left(1+\frac{\vert x \vert^2}{8} \right)^2 \tilde{v}^{(k)}_{\rho,j} (x)\\
& \psi_{\rho,k}\left(\frac{\tau_k \rho x}{\sqrt{8}} + \xi_{\rho,k} \right) \Bigg( - \frac{4 \alpha_k}{\sqrt{8}\tau_k \rho} \frac{\lambda_{1}^{(k)} x_1 + \lambda_{2}^{(k)} x_2}{8+\vert x \vert^2} +\\
&+ \left( \lambda_{1}^{(k)} \frac{\partial \ell_{\rho,k}}{\partial x_1} + \lambda_{2}^{(k)} \frac{\partial \ell_{\rho,k}}{\partial x_2} \right) \left(\frac{\tau_k \rho x}{\sqrt{8}} + \xi_{\rho,k} \right) +  \left(\lambda_{1}^{(k)} \frac{\partial \varphi_{\rho}}{\partial x_1} + \lambda_{2}^{(k)}\frac{\partial \varphi_{\rho}}{\partial x_2} \right)  \left(\frac{\tau_k \rho x}{\sqrt{8}} + \xi_{\rho,k} \right) \Bigg) dx\\
&=\sum_{k=1}^m \frac{\mu_{\rho,j}}{8 \tau_k^2} \int_{B_{\frac{2\sqrt{8} \epsilon}{\tau_k \rho}}(0)} \frac{e^{\ell_{\rho,k} \left(\frac{\tau_k \rho x}{\sqrt{8}} + \xi_{\rho,k} \right) + \varphi_{\rho}\left(\frac{\tau_k \rho x}{\sqrt{8}} + \xi_{\rho,k} \right)}}{\left(1+\frac{\vert x \vert^2}{8} \right)^2} \tilde{v}^{(k)}_{\rho,j} (x)\psi_{\rho,k}\left(\frac{\tau_k \rho x}{\sqrt{8}} + \xi_{\rho,k} \right)\\
& \left(- \frac{4 \alpha_k}{\sqrt{8}\tau_k \rho} \frac{\lambda_{1}^{(k)} x_1 + \lambda_{2}^{(k)} x_2}{8+\vert x \vert^2} + \left( \lambda_{1}^{(k)} \frac{\partial \ell_{\rho,k}}{\partial x_1} + \lambda_{2}^{(k)} \frac{\partial \ell_{\rho,k}}{\partial x_2} \right) \left(\frac{\tau_k \rho x}{\sqrt{8}} + \xi_{\rho,k} \right) \right) dx + \mu_{\rho,j}\frac{o(1)}{\rho},
\end{split}
\end{equation}
where we used the fact that
\begin{equation} \label{cauchy2}
\begin{split}
&\Biggl| \int_{B_{\frac{2\sqrt{8} \epsilon}{\tau_k \rho}}(0)} \frac{e^{\ell_{\rho,k} \left(\frac{\tau_k \rho x}{\sqrt{8}} + \xi_{\rho,k} \right) + \varphi_{\rho}\left(\frac{\tau_k \rho x}{\sqrt{8}} + \xi_{\rho,k} \right)}}{\left(1+\frac{\vert x \vert^2}{8} \right)^2} \tilde{v}^{(k)}_{\rho,j} (x) \psi_{\rho,k}\left(\frac{\tau_k \rho x}{\sqrt{8}} + \xi_{\rho,k} \right) \frac{\partial \varphi_{\rho}}{\partial x_i}\left(\frac{\tau_k \rho x}{\sqrt{8}} + \xi_{\rho,k} \right) dx \Biggr|\\
&\leqslant C \int_{B_{\frac{2\sqrt{8} \epsilon}{\tau_k \rho}}(0)} \frac{1}{\left(1+\frac{\vert x \vert^2}{8} \right)^2} \left| \frac{\partial \varphi_{\rho}}{\partial x_i}\left(\frac{\tau_k \rho x}{\sqrt{8}} + \xi_{\rho,k} \right) \right| dx \\
&\hbox{using Cauchy-Schwarz}\\
&\leqslant C \left( \int_{B_{\frac{2\sqrt{8} \epsilon}{\tau_k \rho}}(0)}\left( \frac{1}{\left(1+\frac{\vert x \vert^2}{8} \right)^2} \right)^2 dx \right)^{\frac{1}{2}} \left( \int_{B_{\frac{2\sqrt{8} \epsilon}{\tau_k \rho}}(0)} \left| \frac{\partial \varphi_{\rho}}{\partial x_i}\left(\frac{\tau_k \rho x}{\sqrt{8}} + \xi_{\rho,k} \right) \right|^2 dx \right)^{\frac{1}{2}} \\
&\leqslant \frac{C}{\rho} \rho^{\beta} \vert \log(\rho) \vert\\
&=\frac{o(1)}{\rho},
\end{split}
\end{equation}
and that
\begin{equation} \label{cauchy3}
\begin{split}
&\rho^8\Bigg| \int_{B_{\frac{2\sqrt{8} \epsilon}{\tau_k \rho}}(0)} e^{-\ell_{\rho,k} \left(\frac{\tau_k \rho x}{\sqrt{8}} + \xi_{\rho,k} \right) - \varphi_{\rho}\left(\frac{\tau_k \rho x}{\sqrt{8}} + \xi_{\rho,k} \right)}\left(1+\frac{\vert x \vert^2}{8} \right)^2\\
&\tilde{v}^{(k)}_{\rho,j} (x) \psi_{\rho,k}\left(\frac{\tau_k \rho x}{\sqrt{8}} + \xi_{\rho,k} \right) \frac{\partial \varphi_{\rho}}{\partial x_i}\left(\frac{\tau_k \rho x}{\sqrt{8}} + \xi_{\rho,k} \right) dx \Bigg|\\
&\leqslant C \rho^8 \int_{B_{\frac{2\sqrt{8} \epsilon}{\tau_k \rho}}(0)} \left(1+\frac{\vert x \vert^2}{8} \right)^2 \left| \frac{\partial \varphi_{\rho}}{\partial x_i}\left(\frac{\tau_k \rho x}{\sqrt{8}} + \xi_{\rho,k} \right) \right| dx \\
&\hbox{using Cauchy-Schwarz}\\
&\leqslant C \rho^8 \left( \int_{B_{\frac{2\sqrt{8} \epsilon}{\tau_k \rho}}(0)} \left(1+\frac{\vert x \vert^2}{8} \right)^4 dx \right)^{\frac{1}{2}} \left( \int_{B_{\frac{2\sqrt{8} \epsilon}{\tau_k \rho}}(0)} \left| \frac{\partial \varphi_{\rho}}{\partial x_i}\left(\frac{\tau_k \rho x}{\sqrt{8}} + \xi_{\rho,k} \right) \right|^2 dx \right)^{\frac{1}{2}} \\
&\leqslant C \rho^2 \rho^{\beta} \vert \log(\rho) \vert\\
&=o(1).
\end{split}
\end{equation}
Moreover,
\begin{align*}
&\sum_{k=1}^m \frac{\mu_{\rho,j}}{8 \tau_k^2} \int_{B_{\frac{2\sqrt{8} \epsilon}{\tau_k \rho}}(0)} \frac{e^{\ell_{\rho,k} \left(\frac{\tau_k \rho x}{\sqrt{8}} + \xi_{\rho,k} \right) + \varphi_{\rho}\left(\frac{\tau_k \rho x}{\sqrt{8}} + \xi_{\rho,k} \right)}}{\left(1+\frac{\vert x \vert^2}{8} \right)^2} \tilde{v}^{(k)}_{\rho,j} (x)\psi_{\rho,k}\left(\frac{\tau_k \rho x}{\sqrt{8}} + \xi_{\rho,k} \right)\\
& \left( - \frac{4 \alpha_k}{\sqrt{8}\tau_k \rho} \frac{\lambda_{1}^{(k)} x_1 + \lambda_{2}^{(k)} x_2}{8+\vert x \vert^2} + \left( \lambda_{1}^{(k)} \frac{\partial \ell_{\rho,k}}{\partial x_1} + \lambda_{2}^{(k)} \frac{\partial \ell_{\rho,k}}{\partial x_2} \right) \left(\frac{\tau_k \rho x}{\sqrt{8}} + \xi_{\rho,k} \right) \right) dx + \mu_{\rho,j}\frac{o(1)}{\rho}\\
&\hbox{using Lebesgue theorem}\\
&=\sum_{k=1}^m - \frac{4 \alpha_k}{\sqrt{8}\tau_k \rho} \mu_{\rho,j}  \int_{\R^2} \frac{1}{\left(1+\frac{\vert x \vert^2}{8} \right)^2}  \frac{s_{1,j}^{(k)} x_1 + s_{2,j}^{(k)} x_2}{8+\vert x \vert^2}  \frac{\lambda_{1}^{(k)} x_1 + \lambda_{2}^{(k)} x_2}{8+\vert x \vert^2} dx (1+o(1)) +\\
&+\sum_{k=1}^m \mu_{\rho,j} \left( \lambda_{1}^{(k)} \frac{\partial \ell_{\rho,k}}{\partial x_1}(\xi_{k}) + \lambda_{2}^{(k)} \frac{\partial \ell_{\rho,k}}{\partial x_2}(\xi_{k}) \right)  \int_{\R^2} \frac{1}{\left(1+\frac{\vert x \vert^2}{8} \right)^2} \frac{s_{1,j}^{(k)} x_1 + s_{2,j}^{(k)} x_2}{8+\vert x \vert^2}  dx+\mu_{\rho,j}\frac{o(1)}{\rho}\\
&=-\frac{\pi}{3\sqrt{8}} \mu_{\rho,j} \frac{1}{\rho}\sum_{k=1}^m  \frac{\alpha_k }{\tau_k }\left(s_{1,j}^{(k)} \lambda_{1}^{(k)} + s_{2,j}^{(k)} \lambda_{2}^{(k)} \right)(1+ o(1))+ \mu_{\rho,j}\frac{o(1)}{\rho}\\
&\hbox{taking $\lambda^{(k)}_1$ and $\lambda^{(k)}_2$ such that $\sum_{k=1}^m  \frac{\alpha_k }{\tau_k }\left(s_{1,j}^{(k)} \lambda_{1}^{(k)} + s_{2,j}^{(k)} \lambda_{2}^{(k)} \right)=0$}\\
&=\mu_{\rho,j} \frac{o(1)}{\rho}.
\end{align*}
Then, if $m+1 \leqslant j \leqslant 3m-1$ we have
\begin{equation} \label{nume}
\left( \sum_{k=1}^m \left(\lambda_{1}^{(k)} \frac{\partial u_{\rho}}{\partial x_1} + \lambda_{2}^{(k)}\frac{\partial u_{\rho}}{\partial x_2} \right) \psi_{\rho,k}, v_{\rho,j} \right)_{H_0^1(\Omega)} = \mu_{\rho,j} \frac{o(1)}{\rho}.
\end{equation}
For the denominator of \eqref{ak}, for $m+1 \leqslant j \leqslant 3m-1$ we have
\begin{equation} \label{deno}
\begin{split}
&(v_{\rho,j},v_{\rho,j})_{H_0^1(\Omega)}\\
&=\mu_{\rho,j} \sum_{k=1}^m  \frac{1}{8 \tau_k^2}  \int_{B_{\frac{\sqrt{8}\epsilon}{\tau_k \rho}}(0)} \frac{e^{\ell_{\rho,k} \left(\frac{\tau_k \rho x}{\sqrt{8}} + \xi_{\rho,k} \right) + \varphi_{\rho}\left(\frac{\tau_k \rho x}{\sqrt{8}} + \xi_{\rho,k} \right)}}{\left(1+\frac{\vert x \vert^2}{8} \right)^2} \left(\tilde{v}^{(k)}_{\rho,j} (x)\right)^2 dx +\\
&+\mu_{\rho,j} \sum_{k=1}^m  \frac{\tau_k^6 \rho^8}{8}\int_{B_{\frac{\sqrt{8}\epsilon}{\tau_k \rho}}(0)} e^{-\ell_{\rho,k} \left(\frac{\tau_k \rho x}{\sqrt{8}} + \xi_{\rho,k} \right) - \varphi_{\rho}\left(\frac{\tau_k \rho x}{\sqrt{8}} + \xi_{\rho,k} \right)}\left(1+\frac{\vert x \vert^2}{8} \right)^2 \left(\tilde{v}^{(k)}_{\rho,j} (x)\right)^2 dx +o(1)\\
&\hbox{using Lebesgue theorem}\\
&=\mu_{\rho,j} \sum_{k=1}^m \int_{\R^2} \frac{1}{\left(1+\frac{\vert x \vert^2}{8} \right)^2} \left( \frac{s_{1,j}^{(k)} x_1 + s_{2,j}^{(k)} x_2}{8+\vert x \vert^2} \right)^2 dx +o(1)\\
&= \frac{\pi}{12} \mu_{\rho,j} \sum_{k=1}^m \left(\left(s_{1,j}^{(k)} \right)^2 + \left(s_{2,j}^{(k)} \right)^2\right)+o(1).
\end{split}
\end{equation}
Using \eqref{nume} and \eqref{deno} in \eqref{ak} we get that for $m+1 \leqslant j \leqslant 3m-1$, we have $\rho a_{\rho,j}=o(1)$.
By the formula for the Rayleigh quotient we get that
\begin{equation} \label{R2}
\begin{split}
\mu_{\rho,3m} &= \inf_{\substack{v \in H_0^1(\Omega), v \neq 0,\\ v \perp v_{\rho,1}, \ldots, v \perp v_{\rho,3m-1}}} \frac{\int_{\Omega} \vert \nabla v \vert^2 dx}{\rho^2 \int_{\Omega} \left(e^{u_{\rho}(x)} +e^{-u_{\rho}(x)} \right) v^2(x) dx}\\
&\leqslant \frac{\int_{\Omega} \vert \nabla \Psi_{\rho}(x) \vert^2 dx}{\rho^2 \int_{\Omega} \left(e^{u_{\rho}(x)} +e^{-u_{\rho}(x)} \right) \Psi^2_{\rho}(x)dx }\\
&=\frac{N}{D}.
\end{split}
\end{equation}
Where
\begin{align*}
N=\int_{\Omega} \left|\nabla \left( \sum_{k}^m \left(\lambda_{1}^{(k)} \frac{\partial u_{\rho}}{\partial x_1}(x) + \lambda_{2}^{(k)}\frac{\partial u_{\rho}}{\partial x_2}(x) \right) \psi_{\rho,k}(x) + \sum_{j=1}^{3m-1} a_{\rho,j} v_{\rho, j}(x) \right) \right|^2 dx
\end{align*}
and
\begin{align*}
D&=\rho^2 \int_{\Omega} \left(e^{u_{\rho}(x)} + e^{-u_{\rho}(x)} \right) \left( \sum_{k}^m \psi_{\rho,k}(x) \sum_{i=1,2}  \lambda_{i}^{(k)} \frac{\partial u_{\rho}}{\partial x_i}(x) + \sum_{j=1}^{3m-1} a_{\rho,j} v_{\rho, j}(x) \right)^2dx.
\end{align*}
Let us start by studying the numerator. Since
\begin{align*}
&\sum_{k=1}^m \int_{\Omega} \left| \nabla \left( \left(\lambda_{1}^{(k)} \frac{\partial u_{\rho}}{\partial x_1}(x) + \lambda_{2}^{(k)}\frac{\partial u_{\rho}}{\partial x_2}(x) \right) \psi_{\rho,k}(x)\right) \right|^2 dx \\
&=-2\sum_{k=1}^m \int_{\Omega} \left( \lambda_{1}^{(k)} \frac{\partial u_{\rho}}{\partial x_1}(x) + \lambda_{2}^{(k)}\frac{\partial u_{\rho}}{\partial x_2}(x) \right)  \psi_{\rho,k}(x)\nabla  \left(\lambda_{1}^{(k)} \frac{\partial u_{\rho}}{\partial x_1}(x) + \lambda_{2}^{(k)}\frac{\partial u_{\rho}}{\partial x_2}(x) \right) \nabla  \psi_{\rho,k}(x) dx +\\
&-\sum_{k=1}^m \int_{\Omega} \Delta \left(\lambda_{1}^{(k)} \frac{\partial u_{\rho}}{\partial x_1}(x) + \lambda_{2}^{(k)}\frac{\partial u_{\rho}}{\partial x_2}(x) \right) \left(\lambda_{1}^{(k)} \frac{\partial u_{\rho}}{\partial x_1}(x) + \lambda_{2}^{(k)}\frac{\partial u_{\rho}}{\partial x_2}(x) \right)\psi_{\rho,k}^2(x)  dx +\\ 
&+\sum_{k=1}^m \int_{\Omega} \left(\lambda_{1}^{(k)} \frac{\partial u_{\rho}}{\partial x_1}(x) + \lambda_{2}^{(k)}\frac{\partial u_{\rho}}{\partial x_2}(x) \right)^2 \left| \nabla \psi_{\rho,k}(x) \right|^2 dx +\\
&+2\sum_{k=1}^m \int_{\Omega} \left( \lambda_{1}^{(k)} \frac{\partial u_{\rho}}{\partial x_1}(x) + \lambda_{2}^{(k)}\frac{\partial u_{\rho}}{\partial x_2}(x) \right)  \psi_{\rho,k}(x)\nabla  \left(\lambda_{1}^{(k)} \frac{\partial u_{\rho}}{\partial x_1}(x) + \lambda_{2}^{(k)}\frac{\partial u_{\rho}}{\partial x_2}(x) \right) \nabla  \psi_{\rho,k}(x) dx \\
&=\sum_{k=1}^m \rho^2 \int_{B_{2\epsilon}(\xi_{\rho,k})} \left(e^{u_{\rho}(x)}+e^{-u_{\rho(x)}} \right)\left(\lambda_{1}^{(k)} \frac{\partial u_{\rho}}{\partial x_1}(x) + \lambda_{2}^{(k)}\frac{\partial u_{\rho}}{\partial x_2}(x) \right)^2 \psi_{\rho,k}^2(x)  dx +\\
&+\sum_{k=1}^m \int_{B_{2\epsilon}(\xi_{\rho,k}) \setminus B_{\epsilon}(\xi_{\rho,k})} \left(\lambda_{1}^{(k)} \frac{\partial u_{\rho}}{\partial x_1}(x) + \lambda_{2}^{(k)}\frac{\partial u_{\rho}}{\partial x_2}(x) \right)^2 \left| \nabla \psi_{\rho,k}(x) \right|^2 dx,
\end{align*}
we have that
\begin{equation} \label{numeratore}
\begin{split}
&\int_{\Omega} \left|\nabla \left( \sum_{k=1}^m \left(\lambda_{1}^{(k)} \frac{\partial u_{\rho}}{\partial x_1}(x) + \lambda_{2}^{(k)}\frac{\partial u_{\rho}}{\partial x_2}(x) \right) \psi_{\rho,k}(x) + \sum_{j=1}^{3m-1} a_{\rho,j} v_{\rho, j}(x) \right) \right|^2 dx\\
&=\sum_{k=1}^m \rho^2 \int_{B_{2\epsilon}(\xi_{\rho,k})} \left(e^{u_{\rho}(x)}+e^{-u_{\rho(x)}} \right)\left(\lambda_{1}^{(k)} \frac{\partial u_{\rho}}{\partial x_1}(x) + \lambda_{2}^{(k)}\frac{\partial u_{\rho}}{\partial x_2}(x) \right)^2 \psi_{\rho,k}^2(x)  dx +\\
&+\sum_{k=1}^m \int_{B_{2 \epsilon}(\xi_{\rho,k}) \setminus B_{\epsilon}(\xi_{\rho,k})} \left(\lambda_{1}^{(k)} \frac{\partial u_{\rho}}{\partial x_1}(x) + \lambda_{2}^{(k)}\frac{\partial u_{\rho}}{\partial x_2}(x) \right)^2 \left| \nabla \psi_{\rho,k}(x) \right|^2 dx+\\
&+ \sum_{j=m}^{3m-1} a_{\rho,j}^2 \int_{\Omega} \left|\nabla v_{\rho, j}(x) \right|^2 dx+\\
&+2 \sum_{j=m}^{3m-1} a_{\rho,j} \int_{\Omega} \nabla \left( \sum_{k=1}^m \left(\lambda_{1}^{(k)} \frac{\partial u_{\rho}}{\partial x_1}(x) + \lambda_{2}^{(k)}\frac{\partial u_{\rho}}{\partial x_2}(x) \right) \psi_{\rho,k}(x)\right) \nabla v_{\rho, j}(x) dx+o(1),\\
\end{split}
\end{equation}
since for $j \leqslant m$, by \eqref{risu1} we have
$$a_{\rho,j}^2 \int_{\Omega} \left|\nabla v_{\rho, j}(x) \right|^2 dx = o(1),$$
and for $j \leqslant m$, by \eqref{risu2} we have
\begin{align*}
&a_{\rho,j} \int_{\Omega} \nabla \left( \sum_{k}^m \left(\lambda_{1}^{(k)} \frac{\partial u_{\rho}}{\partial x_1}(x) + \lambda_{2}^{(k)}\frac{\partial u_{\rho}}{\partial x_2}(x) \right) \psi_{\rho,k}(x)\right) \nabla v_{\rho, j}(x) dx\\
&=\sum_{k=1}^m  \lambda_{1}^{(k)} a_{\rho,j} \left( \frac{\partial u_{\rho}}{\partial x_1} \psi_{\rho,k}, v_{\rho, j} \right)_{H_0^1(\Omega)}+\sum_{k=1}^m  \lambda_{2}^{(k)} a_{\rho,j} \left( \frac{\partial u_{\rho}}{\partial x_2} \psi_{\rho,k}, v_{\rho, j} \right)_{H_0^1(\Omega)}\\
&=o(1).
\end{align*}
For the denominator we have that
\begin{equation} \label{denomi}
\begin{split}
&\rho^2 \int_{\Omega} \left(e^{u_{\rho}(x)} + e^{-u_{\rho}(x)} \right) \left( \sum_{k}^m \left(\lambda_{1}^{(k)} \frac{\partial u_{\rho}}{\partial x_1}(x) + \lambda_{2}^{(k)}\frac{\partial u_{\rho}}{\partial x_2}(x) \right) \psi_{\rho,k}(x) + \sum_{j=1}^{3m-1} a_{\rho,j} v_{\rho, j}(x)\right)^2dx\\
&=\sum_{k=1}^m \rho^2 \int_{B_{2\epsilon}(\xi_{\rho,k})} \left(e^{u_{\rho}(x)} + e^{-u_{\rho}(x)} \right) \left(\lambda_{1}^{(k)} \frac{\partial u_{\rho}}{\partial x_1}(x) + \lambda_{2}^{(k)}\frac{\partial u_{\rho}}{\partial x_2}(x) \right)^2 \psi_{\rho,k}^2(x) dx + \\
&+\sum_{j=m}^{3m-1}\frac{a_{\rho,j}^2}{\mu_{\rho,j}}  \rho^2 \int_{\Omega} \vert \nabla v_{\rho, j} \vert^2(x) dx +\\
&+2 \sum_{j=m}^{3m-1} \frac{a_{\rho,j}}{\mu_{\rho,j}} \left( \sum_{k}^m \left(\lambda_{1}^{(k)} \frac{\partial u_{\rho}}{\partial x_1} + \lambda_{2}^{(k)}\frac{\partial u_{\rho}}{\partial x_2} \right)\psi_{\rho,k}, v_{\rho, j} \right)_{H_0^1(\Omega)} +o(1),
\end{split}
\end{equation}
since $j \leqslant m$, by \eqref{risu1} we have
$$\frac{a_{\rho,j}^2}{\mu_{\rho,j}} \int_{\Omega} \left|\nabla v_{\rho, j}(x) \right|^2 dx = o(1),$$
and by \eqref{risu2} we have
\begin{align*}
&\frac{a_{\rho,j}}{\mu_{\rho,j}} \left( \sum_{k}^m \left(\lambda_{1}^{(k)} \frac{\partial u_{\rho}}{\partial x_1} + \lambda_{2}^{(k)}\frac{\partial u_{\rho}}{\partial x_2} \right)\psi_{\rho,k}, v_{\rho, j} \right)_{H_0^1(\Omega)}\\
&=\sum_{k=1}^m  \lambda_{1}^{(k)} \frac{a_{\rho,j}}{\mu_{\rho,j}} \left( \frac{\partial u_{\rho}}{\partial x_1} \psi_{\rho,k}, v_{\rho, j} \right)_{H_0^1(\Omega)}+\sum_{k=1}^m  \lambda_{2}^{(k)} \frac{a_{\rho,j}}{\mu_{\rho,j}} \left( \frac{\partial u_{\rho}}{\partial x_2} \psi_{\rho,k}, v_{\rho, j} \right)_{H_0^1(\Omega)}\\
&=o(1).
\end{align*}
Substituting \eqref{numeratore} and \eqref{denomi} in \eqref{R2} we get
$$\mu_{2,\rho} \leqslant 1 +\frac{N_{\rho}(x)}{D_{\rho}(x)},$$
where 
\begin{align*}
N_{\rho}(x)&=\sum_{k=1}^m \int_{B_{2 \epsilon}(\xi_{\rho,k}) \setminus B_{\epsilon}(\xi_{\rho,k})} \left(\lambda_{1}^{(k)} \frac{\partial u_{\rho}}{\partial x_1}(x) + \lambda_{2}^{(k)}\frac{\partial u_{\rho}}{\partial x_2}(x) \right)^2 \left| \nabla \psi_{\rho,k}(x) \right|^2 dx+\\
&+ \sum_{j=m}^{3m-1} a_{\rho,j}^2 \frac{(\mu_{\rho,j}-1)}{\mu_{\rho,j}} \int_{\Omega} \left|\nabla v_{\rho, j}(x) \right|^2 dx+2 \sum_{j=m}^{3m-1} a_{\rho,j} \frac{(\mu_{\rho,j}-1)}{\mu_{\rho,j}}\\
&\int_{\Omega} \nabla \left( \sum_{k=1}^m \left(\lambda_{1}^{(k)} \frac{\partial u_{\rho}}{\partial x_1}(x) + \lambda_{2}^{(k)}\frac{\partial u_{\rho}}{\partial x_2}(x) \right) \psi_{\rho,k}(x)\right) \nabla v_{\rho, j}(x) dx+o(1),
\end{align*}
and
\begin{align*}
D_{\rho}(x)&=\sum_{k=1}^m \rho^2 \int_{B_{2\epsilon}(\xi_{\rho,k})} \left(e^{u_{\rho}(x)} + e^{-u_{\rho}(x)} \right) \left(\lambda_{1}^{(k)} \frac{\partial u_{\rho}}{\partial x_1}(x) + \lambda_{2}^{(k)}\frac{\partial u_{\rho}}{\partial x_2}(x) \right)^2 \psi_{\rho,k}^2(x) dx + \\
&+\sum_{j=m}^{3m-1}\frac{a_{\rho,j}^2}{\mu_{\rho,j}}  \rho^2 \int_{\Omega} \vert \nabla v_{\rho, j} \vert^2(x) dx +\\
&+2 \sum_{j=m}^{3m-1} \frac{a_{\rho,j}}{\mu_{\rho,j}} \left( \sum_{k}^m \left(\lambda_{1}^{(k)} \frac{\partial u_{\rho}}{\partial x_1} + \lambda_{2}^{(k)}\frac{\partial u_{\rho}}{\partial x_2} \right)\psi_{\rho,k}, v_{\rho, j} \right)_{H_0^1(\Omega)} +o(1).
\end{align*}
It remains to prove that $$\frac{N_{\rho}(x)}{D_{\rho}(x)} \leqslant C \rho^2.$$
Let us start by studying $N_{\rho}(x)$. We have
$$\int_{B_{2 \epsilon}(\xi_{\rho,k}) \setminus B_{\epsilon}(\xi_{\rho,k})} \left(\lambda_{1}^{(k)} \frac{\partial u_{\rho}}{\partial x_1}(x) + \lambda_{2}^{(k)}\frac{\partial u_{\rho}}{\partial x_2}(x) \right)^2 \left| \nabla \psi_{\rho,k}(x) \right|^2 dx < C,$$
by \eqref{deno}, and since $a_{\rho,j}=\frac{o(1)}{\rho}$ and $\mu_{\rho,j}-1 \leqslant C \rho^2$, we have
\begin{align*}
&a_{\rho,j}^2 \frac{(\mu_{\rho,j}-1)}{\mu_{\rho,j}} \int_{\Omega} \left|\nabla v_{\rho, j}(x) \right|^2 dx = \frac{\pi}{12} a_{\rho,j}^2 (\mu_{\rho,j}-1)\left( \sum_{k=1}^m \left(s_{1,j}^{(k)} \right)^2 + \left(s_{2,j}^{(k)} \right)^2 +o(1)\right)= o(1),
\end{align*}
and by \eqref{nume}, and since $a_{\rho,j}=\frac{o(1)}{\rho}$ and $\mu_{\rho,j}-1 \leqslant C \rho^2$, we have
\begin{align*}
&a_{\rho,j} \frac{(\mu_{\rho,j}-1)}{\mu_{\rho,j}} \int_{\Omega} \nabla \left( \sum_{k=1}^m \left(\lambda_{1}^{(k)} \frac{\partial u_{\rho}}{\partial x_1}(x) + \lambda_{2}^{(k)}\frac{\partial u_{\rho}}{\partial x_2}(x) \right) \psi_{\rho,k}(x)\right) \nabla v_{\rho, j}(x) dx\\
&=a_{\rho,j} \frac{(\mu_{\rho,j}-1)}{\mu_{\rho,j}} \left( \sum_{k=1}^m \left(\lambda_{1}^{(k)} \frac{\partial u_{\rho}}{\partial x_1} + \lambda_{2}^{(k)}\frac{\partial u_{\rho}}{\partial x_2} \right) \psi_{\rho,k}, v_{\rho,j} \right)_{H_0^1(\Omega)}\\
&=a_{\rho,j} (\mu_{\rho,j}-1) \frac{o(1)}{\rho}\\
&=o(1).
\end{align*}
So 
\begin{equation} \label{nonso1}
N_{\rho}(x)< C+ o(1).
\end{equation}
For the denominator we have
\begin{equation} \label{nonso2}
\begin{split}
D_{\rho}(x)& \geqslant \sum_{k=1}^m \rho^2 \int_{B_{2\epsilon}(\xi_{\rho,k})} \left(e^{u_{\rho}(x)} + e^{-u_{\rho}(x)} \right) \left(\lambda_{1}^{(k)} \frac{\partial u_{\rho}}{\partial x_1}(x) + \lambda_{2}^{(k)}\frac{\partial u_{\rho}}{\partial x_2}(x) \right)^2 \psi_{\rho,k}^2(x) dx + \\
&+2 \sum_{j=m}^{3m-1} \frac{a_{\rho,j}}{\mu_{\rho,j}} \left( \sum_{k}^m \left(\lambda_{1}^{(k)} \frac{\partial u_{\rho}}{\partial x_1} + \lambda_{2}^{(k)}\frac{\partial u_{\rho}}{\partial x_2} \right)\psi_{\rho,k}, v_{\rho, j} \right)_{H_0^1(\Omega)} +o(1)\\
&> \frac{C}{\rho^2}(1+o(1)).
\end{split}
\end{equation}
In fact, by computations analogous to those of \eqref{appox}, we have
\begin{align*}
&\rho^2 \int_{B_{2\epsilon}(\xi_{\rho,k})} \left(e^{u_{\rho}(x)} + e^{-u_{\rho}(x)} \right) \left(\lambda_{1}^{(k)} \frac{\partial u_{\rho}}{\partial x_1}(x) + \lambda_{2}^{(k)}\frac{\partial u_{\rho}}{\partial x_2}(x) \right)^2 \psi_{\rho,k}^2(x) dx > \frac{C }{\rho^2} (1 + o(1)),
\end{align*}
and by \eqref{nume}, and since $a_{\rho,j}=\frac{o(1)}{\rho}$, we have
$$\frac{a_{\rho,j}}{\mu_{\rho,j}} \left( \sum_{k}^m \left(\lambda_{1}^{(k)} \frac{\partial u_{\rho}}{\partial x_1} + \lambda_{2}^{(k)}\frac{\partial u_{\rho}}{\partial x_2} \right)\psi_{\rho,k}, v_{\rho, j} \right)_{H_0^1(\Omega)}=a_{\rho} \frac{o(1)}{\rho}=\frac{o(1)}{\rho^2}.$$
Using \eqref{nonso1} and \eqref{nonso2} we get that
$$\frac{N_{\rho}(x)}{D_{\rho}(x)} \leqslant C \rho^2,$$
so that (\textit{i}) is proved.

By (\textit{i}) the hypotheses of Proposition \ref{uno1} are satisfied. We know by the proof of Proposition \ref{uno1} that the first two eigenvalues are zero and 1, and we know that $\mu_{\rho,j} \neq 0$ for all $j>m$ and so (\textit{ii}) follows.

Let us now prove (\textit{iii}). By Theorem \ref{tre3} it is enough to prove that there exists $1 \leqslant k \leqslant m$ such that $t^{(k)}_j \neq 0$ for $j>3m$. Let us assume by contradiction that $t^{(k)}_{3m+1}=0$ for every $1 \leqslant k \leqslant m$. 
By the orthogonality of the eigenfunctions corresponding to distinct eigenvalues we have that
\begin{align*}
0=(v_{\rho,j}, v_{\rho,n})_{H_0^1(\Omega)}=\mu_{\rho,j} \rho^2 \int_{\Omega} \left(e^{u_{\rho}(x)} + e^{-u_{\rho}(x)} \right) v_{\rho,j}(x) v_{\rho,n}(x) dx\\
\end{align*}
for every $m+1 \leqslant j<n \leqslant 3m+1$. Then
\begin{align*}
0&=\rho^2 \int_{\Omega} \left(e^{u_{\rho}(x)} + e^{-u_{\rho}(x)} \right) v_{\rho,j}(x) v_{\rho,n}(x) dx\\
&=\sum_{k=1}^m \rho^2 \frac{\tau_k^2 \rho^2}{8} \frac{1}{\tau_k^4 \rho^4} \int_{B_{\frac{\sqrt{8}\epsilon}{\tau_k \rho}}(0)} \frac{e^{-\ell_{\rho,k} \left(\frac{\tau_k \rho x}{\sqrt{8}} + \xi_{\rho,k} \right) - \varphi_{\rho}\left(\frac{\tau_k \rho x}{\sqrt{8}} + \xi_{\rho,k} \right)}}{\left(1+\frac{\vert x \vert^2}{8} \right)^2} \tilde{v}^{(k)}_{\rho,j}(x) \tilde{v}^{(k)}_{\rho,n}(x) dx+\\
&+\sum_{k=1}^m \rho^2 \frac{\tau_k^2 \rho^2}{8} \tau_k^4 \rho^4\\
& \int_{B_{\frac{\sqrt{8}\epsilon}{\tau_k \rho}}(0)} e^{-\ell_{\rho,k} \left(\frac{\tau_k \rho x}{\sqrt{8}} + \xi_{\rho,k} \right) - \varphi_{\rho}\left(\frac{\tau_k \rho x}{\sqrt{8}} + \xi_{\rho,k} \right)}\left(1+\frac{\vert x \vert^2}{8} \right)^2 \tilde{v}^{(k)}_{\rho,j}(x) \tilde{v}^{(k)}_{\rho,n}(x) dx+o(1)\\
&\hbox{using Lebesgue theorem}\\
&=\sum_{k=1}^m \int_{\R^2} \frac{1}{\left(1+\frac{\vert x \vert^2}{8} \right)^2} \left( \frac{s_{1,j}^{(k)} x_1 + s_{2,j}^{(k)} x_2}{8+\vert x \vert^2} \right) \left( \frac{s_{1,n}^{(k)} x_1 + s_{2,n}^{(k)} x_2}{8+\vert x \vert^2} \right) dx + o(1)\\
&= \frac{\pi}{12} \sum_{k=1}^m s_{1,j}^{(k)}s_{1,n}^{(k)} + s_{2,j}^{(k)}s_{2,n}^{(k)}+o(1),
\end{align*}
Then, having set $s_j=(s_{1,j}^{1}, \ldots, s_{1,j}^{m}, s_{2,j}^{1}, \ldots, s_{2,j}^{m}) \in \R^{2m}$, we get
$$\sum_{k=1}^m s_{1,j}^{(k)}s_{1,3m+1}^{(k)} + s_{2,j}^{(k)}s_{2,3m+1}^{(k)} =s_j \cdot s_{3m+1}=0; \quad \forall m \leqslant j \leqslant 3m,$$
but in $\R^{2m}$ there are at most $2m$ orthogonal vectors, so that the condition cannot be satisfied, from which we get a contradiction.
\endproof

\section{Morse Index Computations}
In this chapter, we will give an asymptotic estimate for the eigenvalues and the eigenfunctions from $m+1$ to $3m$. While computing this estimate, we will also calculate the Morse index.

Let us firstly prove the following lemma.

\begin{lem} \label{quasi}
If $j$ is such that $\mu_{\rho,j} \to 1$ and $t_{j}=\textbf{0}$, then for every $k=1, \ldots,m$, we have
$$\Lambda^{j,k}_{\rho}:=\rho^2 \int_{B_{\epsilon}(\xi_{\rho,k})} \left(e^{u_{\rho}(x)} + e^{-u_{\rho}(x)} \right) v_{\rho,j}(x) dx = o(\rho).$$
\end{lem}
\proof
By \eqref{stima1}, for every $j$ we have that
\begin{align*}
v_{\rho,j}(x) =  \mu_{\rho,j}\sum_{k=1}^m G(x,\xi_{k}) \Lambda^{k,j}_{\rho}+o(1)
\end{align*}
in $C^1(\bar{\Omega} \setminus \lbrace \xi_1, \ldots, \xi_m \rbrace)$. 
Moreover, by Lemma \ref{lemma} we have that
\begin{align*}
&(1-\mu_{\rho,j}) \rho^2 \int_{B_{\epsilon}(\xi_k)} \left( e^{u_{\rho}(x)} + e^{-u_{\rho}(x)} \right) \frac{\partial u_{\rho}}{ \partial x_i}(x) v_{\rho,j}(x) dx \\
&= \int_{\partial B_{\epsilon}(\xi_k)}\left( \frac{\partial v_{\rho,j}}{\partial \nu }(x) \frac{\partial u_{\rho}}{ \partial x_i}(x) -v_{\rho,j}(x)\frac{\partial^2 u_{\rho}}{ \partial \nu \partial x_i}(x) \right) d\sigma_x.
\end{align*}
Let us start by studying the left hand side.
\begin{equation} \label{aiut}
\begin{split}
&\rho^2 \int_{B_{\epsilon}(\xi_k)} \left( e^{u_{\rho}(x)} + e^{-u_{\rho}(x)} \right) \frac{\partial u_{\rho}}{ \partial x_i}(x) v_{\rho,j}(x) dx\\
&=\rho^2 \frac{\tau_k^2 \rho^2}{8} \frac{1}{\tau_k^4 \rho^4}\int_{B_{\frac{\sqrt{8}\epsilon}{\tau_k \rho}}(0)} \frac{e^{\ell_{\rho,k}\left(\frac{\tau_k \rho x}{\sqrt{8}} +\xi_{\rho,k} \right) + \varphi_{\rho}\left(\frac{\tau_k \rho x}{\sqrt{8}} +\xi_{\rho,k} \right)}}{\left(1+\frac{\vert x \vert^2}{8} \right)^2} \tilde{v}_{\rho,j}^{(k)}(x) \\
&\left( -\frac{4 \alpha_k}{\sqrt{8} \tau_k \rho} \frac{x_i}{1+\frac{\vert x \vert^2}{8}} + \frac{\partial \ell_{\rho,k}}{\partial x_i} \left(\frac{\tau_k \rho x}{\sqrt{8}} +\xi_{\rho,k} \right) + \frac{\partial \varphi_{\rho}}{\partial x_i}\left(\frac{\tau_k \rho x}{\sqrt{8}} +\xi_{\rho,k} \right) \right) dx + \\
& + \rho^2 \frac{\tau_k^2 \rho^2}{8} \tau_k^4 \rho^4 \int_{B_{\frac{\sqrt{8}\epsilon}{\tau_k \rho}}(0)} e^{-\ell_{\rho,k}\left(\frac{\tau_k \rho x}{\sqrt{8}} +\xi_{\rho,k} \right) - \varphi_{\rho}\left(\frac{\tau_k \rho x}{\sqrt{8}} +\xi_{\rho,k} \right)}\left(1+\frac{\vert x \vert^2}{8} \right)^2 \tilde{v}_{\rho,j}^{(k)}(x) \\
&\left( -\frac{4 \alpha_k}{\sqrt{8} \tau_k \rho} \frac{x_i}{1+\frac{\vert x \vert^2}{8}} + \frac{\partial \ell_{\rho,k}}{\partial x_i} \left(\frac{\tau_k \rho x}{\sqrt{8}} +\xi_{\rho,k} \right) + \frac{\partial \varphi_{\rho}}{\partial x_i}\left(\frac{\tau_k \rho x}{\sqrt{8}} +\xi_{\rho,k} \right) \right) dx+o(1)\\
&\hbox{using Lebesgue theorem}\\
&=-\frac{4}{\sqrt{8} \rho} \frac{\alpha_k}{\tau_k} \int_{\R^2} \frac{x_i}{\left(1+\frac{\vert x \vert^2}{8} \right)^3} \left(\frac{s_{1,j}^{(k)} x_1 + s_{2,j}^{(k)} x_2}{8+\vert x \vert^2}  \right) dx(1+o(1))\\
&=-\frac{16}{3\sqrt{8} \rho} \pi \frac{\alpha_k}{\tau_k} s_{i,j}^{(k)}(1+o(1)),
\end{split}
\end{equation}
where we used \eqref{cauchy2} and \eqref{cauchy3}.
For the right hand side, by Lemma \ref{lemma2} we get
$$ \int_{\partial B_{\epsilon}(\xi_k)} \left( \frac{\partial v_{\rho,j}}{\partial \nu }(x) \frac{\partial u_{\rho}}{ \partial x_i}(x) -v_{\rho,j}(x)\frac{\partial^2 u_{\rho}}{ \partial \nu \partial x_i}(x) \right) d\sigma_x = o(1),$$
hence 
$$-\frac{16}{3\sqrt{8}\rho}\pi \frac{\alpha_k}{\tau_k}s_{i,j}^{(k)}(1+o(1))(1-\mu_{\rho,j}) =o(1),$$
and so
$$(1-\mu_{\rho,j}) = o(\rho).$$
By \ref{lemma3}, and proceeding as in Theorem \ref{tre3}, we get
\begin{align*}
&R\int_{\partial B_R(\xi_k)} \left( 2\frac{\partial u_\rho}{\partial \nu}(x)\frac{\partial v_{\rho,j}}{\partial \nu}(x)-\nabla u_{\rho}(x) \cdot \nabla v_{\rho,j}(x) \right) d \sigma(x)\\
&=\int_{B_R(\xi_k)} \left\lbrace \left[(x-\xi_k) \cdot \nabla u_\rho (x)\right] \Delta v_{\rho,j}(x) + \left[ (x-\xi_k) \cdot \nabla v_{\rho,j} (x)\right] \Delta u_{\rho}(x) \right\rbrace dx\\
&=-\rho^2 \int_{\partial B_R(\xi_k)} (x-\xi_k) \cdot \nu  \left( e^{u_{\rho}(x)} - e^{-u_{\rho}(x)} \right) v_{\rho,j}(x) d\sigma(x) +\\
&+ 2 \rho^2 \int_{B_R(\xi_k)} \left( e^{u_{\rho}(x)} - e^{-u_{\rho}(x)} \right) v_{\rho,j}(x) dx +\\
&+(1-\mu_{\rho,j}) \rho^2 \int_{B_R(\xi_k)} \left[(x-\xi_k) \cdot \nabla u_\rho (x)\right] \left( e^{u_{\rho}(x)} + e^{-u_{\rho}(x)} \right) v_{\rho,j}(x) dx.
\end{align*}
For the left hand side we have
\begin{align*}
&R\int_{\partial B_R(\xi_k)} \left( 2\frac{\partial u_\rho}{\partial \nu}(x)\frac{\partial v_{\rho,j}}{\partial \nu}(x)-\nabla u_{\rho}(x) \cdot \nabla v_{\rho,j}(x) \right) d \sigma(x)\\
&=R\int_{\partial B_R(\xi_k)}\Bigg[  2\frac{\partial }{\partial \nu}\left(8 \pi \sum_{i=1}^m G(x,\xi_{i}) +o(1) \right) \frac{\partial }{\partial \nu}\left(\mu_{\rho,j} \sum_{i=1}^m G(x,\xi_i)\Lambda^{i,j}_{\rho} + o(1) \right) +\\
&-\nabla \left(8 \pi \sum_{i=1}^m G(x,\xi_{i}) +o(1) \right) \cdot \nabla \left(\mu_{\rho,j} \sum_{i=1}^m G(x,\xi_i)\Lambda^{i,j}_{\rho} + o(1) \right) \Bigg] d \sigma(x)\\
&\hbox{using Lemma \ref{lll}}\\
&=4 \mu_{\rho,j} \Lambda^{k,j}_{\rho} + o_R(1).
\end{align*}
For the right hand side we have that
\begin{equation} \label{rug1}
\begin{split}
&\rho^2 \int_{B_R(\xi_k)} \left( e^{u_{\rho}(x)} - e^{-u_{\rho}(x)} \right) v_{\rho,j}(x) dx\\
&= \alpha_k \rho^2 \frac{\tau_k^2 \rho^2}{8} \frac{1}{\tau_k^4 \rho^4} \int_{B_{\frac{\sqrt{8}R}{\tau_k \rho}}(0)} \frac{e^{\ell_{\rho,k} \left(\frac{\tau_k \rho x}{\sqrt{8}} + \xi_{\rho,k} \right) + \varphi_{\rho}\left(\frac{\tau_k \rho x}{\sqrt{8}} + \xi_{\rho,k} \right)}}{\left(1+\frac{\vert x \vert^2}{8} \right)^2} \tilde{v}^{(k)}_{\rho,j}(x) dx+\\
&- \alpha_k \rho^2 \frac{\tau_k^2 \rho^2}{8} \tau_k^4 \rho^4 \int_{B_{\frac{\sqrt{8}R}{\tau_k \rho}}(0)} e^{-\ell_{\rho,k} \left(\frac{\tau_k \rho x}{\sqrt{8}} + \xi_{\rho,k} \right) - \varphi_{\rho}\left(\frac{\tau_k \rho x}{\sqrt{8}} + \xi_{\rho,k} \right)}\left(1+\frac{\vert x \vert^2}{8} \right)^2 \tilde{v}^{(k)}_{\rho,j}(x) dx\\
&= \alpha_k \rho^2 \frac{\tau_k^2 \rho^2}{8} \frac{1}{\tau_k^4 \rho^4}\int_{B_{\frac{\sqrt{8}R}{\tau_k \rho}}(0)} \frac{e^{\ell_{\rho,k} \left(\frac{\tau_k \rho x}{\sqrt{8}} + \xi_{\rho,k} \right) + \varphi_{\rho}\left(\frac{\tau_k \rho x}{\sqrt{8}} + \xi_{\rho,k} \right)}}{\left(1+\frac{\vert x \vert^2}{8} \right)^2} \tilde{v}^{(k)}_{\rho,j}(x) dx+ o(\rho).
\end{split}
\end{equation}
Furthermore,
\begin{equation} \label{rug2}
\begin{split}
&\rho^2 \int_{B_R(\xi_k)} \left( e^{u_{\rho}(x)} + e^{-u_{\rho}(x)} \right) v_{\rho,j}(x) dx\\
&= \rho^2 \frac{\tau_k^2 \rho^2}{8} \frac{1}{\tau_k^4 \rho^4} \int_{B_{\frac{\sqrt{8}R}{\tau_k \rho}}(0)} \frac{e^{\ell_{\rho,k} \left(\frac{\tau_k \rho x}{\sqrt{8}} + \xi_{\rho,k} \right) + \varphi_{\rho}\left(\frac{\tau_k \rho x}{\sqrt{8}} + \xi_{\rho,k} \right)}}{\left(1+\frac{\vert x \vert^2}{8} \right)^2} \tilde{v}^{(k)}_{\rho,j}(x) dx+\\
&+ \rho^2 \frac{\tau_k^2 \rho^2}{8} \tau_k^4 \rho^4 \int_{B_{\frac{\sqrt{8}R}{\tau_k \rho}}(0)} e^{-\ell_{\rho,k} \left(\frac{\tau_k \rho x}{\sqrt{8}} + \xi_{\rho,k} \right) - \varphi_{\rho}\left(\frac{\tau_k \rho x}{\sqrt{8}} + \xi_{\rho,k} \right)}\left(1+\frac{\vert x \vert^2}{8} \right)^2 \tilde{v}^{(k)}_{\rho,j}(x) dx\\
&= \rho^2 \frac{\tau_k^2 \rho^2}{8} \frac{1}{\tau_k^4 \rho^4}\int_{B_{\frac{\sqrt{8}R}{\tau_k \rho}}(0)} \frac{e^{\ell_{\rho,k} \left(\frac{\tau_k \rho x}{\sqrt{8}} + \xi_{\rho,k} \right) + \varphi_{\rho}\left(\frac{\tau_k \rho x}{\sqrt{8}} + \xi_{\rho,k} \right)}}{\left(1+\frac{\vert x \vert^2}{8} \right)^2} \tilde{v}^{(k)}_{\rho,j}(x) dx+ o(\rho),
\end{split}
\end{equation}
Putting \eqref{rug1} and \eqref{rug2} together we get that
$$\rho^2 \int_{B_R(\xi_k)} \left( e^{u_{\rho}(x)} - e^{-u_{\rho}(x)} \right) v_{\rho,j}(x) dx=\alpha_k \Lambda^{k,j}_{\rho} +o(\rho).$$
Proceeding as in Theorem \ref{tre3} we have that
$$-\rho^2 \int_{\partial B_R(\xi_k)} (x-\xi_k) \cdot \nu  \left( e^{u_{\rho}(x)} - e^{-u_{\rho}(x)} \right) v_{\rho,j}(x) d\sigma(x) = o_R(1),$$
and
$$\rho^2 \int_{B_R(\xi_k)} \left[(x-\xi_k) \cdot \nabla u_\rho (x)\right] \left( e^{u_{\rho}(x)} + e^{-u_{\rho}(x)} \right) v_{\rho,j}(x) dx=o(1).$$
Putting everything together we get 
$$4 \mu_{\rho,j} \Lambda^{k,j}_{\rho} + o_R(1) = o_R(1)+ 2 \alpha_k \Lambda^{k,j}_{\rho}+o(\rho) + (1-\mu_{\rho,j}) o (1).$$
Since this hold true for every $R$, we get
$$2(2\mu_{\rho,j}-\alpha_k)\Lambda^{k,j}_{\rho} =o(\rho) +(1-\mu_{\rho,j}) o (1),$$
the conclusion follows.
\endproof

We recall the following:

\begin{prop}[{\cite[Proposition 2.3]{GOS}}] \label{propo}
Let $z_1$, $z_2$ and $z_3$ be in $\Omega$ and let $R$ be such that $B_R(z_1) \subset \Omega$, and $z_i \notin B_R(z_1)$ for $i=2,3$, then
\begin{align*}
&\int_{\partial B_R(z_1)} \left( \frac{\partial^2 G}{\partial \nu_x \partial x_i}(x,z_2) \frac{\partial G}{\partial y_j}(x,z_3) - \frac{\partial G}{\partial x_i}(x,z_2) \frac{\partial^2 G}{\partial \nu_x \partial y_j}(x,z_3)\right) d \sigma_x\\
&=\begin{cases} 0 & (z_1 \neq z_2, z_1 \neq z_3)\\
\frac{1}{2} \frac{\partial^2 R}{\partial x_i \partial x_j} (z_1,z_1)  & (z_1 = z_2 = z_3)\\
\frac{\partial^2 G}{\partial x_i \partial y_j} (z_1, z_3) & (z_1 = z_2 \neq z_3)\\
\frac{\partial^2 G}{\partial x_i \partial x_j} (z_1, z_2) & (z_1 = z_3 \neq z_2)\\
\end{cases}
\end{align*}
\end{prop}

We will now give an estimate for the eigenfunctions away from the blow-up points. This estimate will hold precisely when the one provided by Lemma \ref{due2} doesn't hold, that is for the eigenfunctions from $m+1$ to $3m$. Moreover, we will identify the connection between the Morse index of $u_{\rho}$ and that of the Hamilton function $\mathcal{F}$.
\begin{prop} \label{t13}
Let $j$ be such that $t_j=\textbf{0}$. Then 
\begin{equation} \label{fine1}
\frac{v_{\rho,j}}{\rho} \to 2 \pi \sum_{k=1}^m \frac{\tau_k}{\sqrt{8}} \left( s_{1,j}^{(k)} \frac{\partial G}{\partial x_1} (x,\xi_k)+ s_{2,j}^{(k)} \frac{\partial G}{\partial x_2} (x,\xi_k) \right)
\end{equation} 
in $C^1_{\mathrm{loc}}(\bar{\Omega} \setminus \lbrace \xi_1, \ldots, \xi_m \rbrace)$.
Furthermore
$$(1-\mu_{\rho,j})=3 \pi \rho^2 \eta_j (1+o(1)),$$
where the $\eta_j$'s, for $j=1, \ldots, 2m$, are the eigenvalues of the matrix $D (\mathrm{Hess} \, \mathcal{F}) D$, where $D$ is the diagonal matrix $D= \mathrm{diag} (\tau_1, \ldots, \tau_m, \tau_1, \ldots, \tau_m )$, and $\mathrm{Hess} \, \mathcal{F}$ is the Hessian matrix of the Hamilton function $\mathcal{F}$.
\end{prop}
\proof
Let $K \subset \bar{\Omega} \setminus \lbrace \xi_1, \ldots, \xi_m \rbrace$ be compact. Let
$$\epsilon'=\min_{k} \inf_{\rho} \mathrm{dist}(K,\xi_{\rho,k}),$$
and let $\lambda=\min(\epsilon, \epsilon')$.
Since $G(\cdot , \cdot)$ is smooth on $K \times B_{\lambda}(\xi_{\rho,k})$, we can consider its Taylor expansion to get
\begin{align*}
G(x,y)&= G(x, \xi_{\rho,k}) + \sum_{i=1,2} \frac{\partial G}{\partial y_i} (x,\xi_{\rho,k}) (y-\xi_{\rho,k})_i \\
&+ \frac{1}{2} \sum_{i,h=1,2} \frac{\partial^2 G}{\partial y_i \partial y_h} (x, \eta_{\rho}) (y-\xi_{\rho,k})_i(y-\xi_{\rho,k})_h,
\end{align*}
where $\eta_{\rho}$ is a point lying on the segment connecting $y$ to $\xi_{\rho,k}$, and hence contained in $B_{\lambda}(\xi_{\rho,k})$.
Taking $\lambda=\rho^\gamma$ with $\gamma <\frac{1}{4}$ we get
\begin{align*}
&\left| \mu_{\rho,j} \rho^2 \int_{\Omega \setminus \cup_{k=1}^m B_{\lambda}(\xi_{\rho,k})} \left(e^{u_{\rho}(y)} + e^{-u_{\rho}(y)} \right) v_{\rho,j}(y) G(x,y) dy \right|\\
&\leqslant C\left( \frac{\rho^2}{\left(\rho^2 + \lambda^2 \right)^2}+\rho^2 \right) \int_{\Omega \setminus \cup_{k=1}^m B_{\lambda}(\xi_{\rho,k})} \vert G(x,y) \vert dy\\
&\hbox{since $G(x, \cdot) \in L^1(\Omega)$}\\
&=o(\rho).
\end{align*}
Using Green's representation formula, for each $x \in K$, we have
\begin{align*}
v_{\rho,j}(x)&= \mu_{\rho,j} \rho^2 \int_{\Omega} \left(e^{u_{\rho}(y)} + e^{-u_{\rho}(y)} \right) v_{\rho,j}(y) G(x,y) dy\\
&=\mu_{\rho,j}\sum_{k=1}^m \rho^2 \int_{B_{\lambda}(\xi_{\rho,k})} \left(e^{u_{\rho}(y)} + e^{-u_{\rho}(y)} \right) v_{\rho,j}(y)\\
&\left( G(x, \xi_{\rho,k}) + \sum_{i=1,2} \frac{\partial G}{\partial y_i} (x,\xi_{\rho,k}) (y-\xi_{\rho,k})_i \right) dy +R_{\rho} + o(\rho),
\end{align*}
having set
\begin{align*}
R_{\rho}&=\frac{\mu_{\rho,j}}{2} \sum_{k=1}^m \rho^2 \int_{B_{\lambda}(\xi_{\rho,k})} \left(e^{u_{\rho}(y)} + e^{-u_{\rho}(y)} \right) v_{\rho,j}(y) \\
&\sum_{i,h=1,2} \frac{\partial^2 G}{\partial y_i \partial y_h} (x, \eta_{\rho}) (y-\xi_{\rho,k})_i(y-\xi_{\rho,k})_h  dy.
\end{align*}
Since $\eta_{\rho} \in B_{\lambda}(\xi_{\rho,k})$ and $x \notin B_{\lambda}(\xi_{\rho,k})$, we have
$$\left|\frac{\partial^2 G}{\partial y_i \partial y_h} (x, \eta_{\rho}) \right| \leqslant \sup_{i,h=1,2; z \in B_{\lambda}(\xi_{\rho,k})} \left|\frac{\partial^2 G}{\partial y_i \partial y_h} (x, z) \right| \leqslant C.$$
Since $\mu_{\rho,j} = 1+o(1)$, we have
\begin{align*}
\vert R_{\rho} \vert &\leqslant C \sum_{k=1}^m \rho^2 \int_{B_{\lambda}(\xi_{\rho,k})} \left(e^{u_{\rho}(y)} + e^{-u_{\rho}(y)} \right) \vert y-\xi_{\rho,k} \vert^2 dy \\
&= C \sum_{k=1}^m \frac{\rho^2}{64} \int_{B_{\frac{\sqrt{8}\lambda}{\tau_k \rho}}(0)} \frac{e^{\ell_{\rho,k}\left(\frac{\tau_k \rho y}{\sqrt{8}} + \xi_{\rho,k} \right) + \varphi_{\rho}\left(\frac{\tau_k \rho y}{\sqrt{8}} + \xi_{\rho,k} \right)}}{\left(1+\frac{\vert y \vert^2}{8} \right)^2} \vert y \vert^2 dy + \\
&+ C \sum_{k=1}^m \frac{\tau_k^8 \rho^{10}}{64} \int_{B_{\frac{\sqrt{8}\lambda}{\tau_k \rho}}(0)} e^{-\ell_{\rho,k}\left(\frac{\tau_k \rho y}{\sqrt{8}} + \xi_{\rho,k} \right) - \varphi_{\rho}\left(\frac{\tau_k \rho y}{\sqrt{8}} + \xi_{\rho,k} \right)}\left(1+\frac{\vert y \vert^2}{8} \right)^2 \vert y \vert^2 dy\\
&=o(\rho),
\end{align*}
where we used the fact that
\begin{align*}
&\rho^2 \left| \int_{B_{\frac{\sqrt{8}\lambda}{\tau_k \rho}}(0)} \frac{e^{\ell_{\rho,k}\left(\frac{\tau_k \rho y}{\sqrt{8}} + \xi_{\rho,k} \right) + \varphi_{\rho}\left(\frac{\tau_k \rho y}{\sqrt{8}} + \xi_{\rho,k} \right)}}{\left(1+\frac{\vert y \vert^2}{8} \right)^2} \vert y \vert^2 dy \right| \leqslant C\rho^2 \left(\log(\rho) + 1 \right)=o(\rho).
\end{align*}
Therefore
\begin{align*}
v_{\rho,j}(x) &=\mu_{\rho,j} \sum_{k=1}^m \left(G(x, \xi_{k})+o(1) \right) \rho^2 \int_{B_{\lambda}(\xi_{\rho,k})} \left(e^{u_{\rho}(y)} + e^{-u_{\rho}(y)} \right) v_{\rho,j}(y) dy+\\
&+\mu_{\rho,j} \sum_{i=1,2} \frac{\partial G}{\partial y_i} (x,\xi_{k})\left(1+o(1) \right) \\
&\sum_{k=1}^m \rho^2 \int_{B_{\lambda}(\xi_{\rho,k})} \left(e^{u_{\rho}(y)} + e^{-u_{\rho}(y)} \right) v_{\rho,j}(y) (y-\xi_{\rho,k})_i dy+ o(\rho)\\
&\hbox{using Lemma \ref{quasi} }\\
&=\mu_{\rho,j} \sum_{i=1,2} \frac{\partial G}{\partial y_i} (x,\xi_{k})\left(1+o(1) \right) \\
&\sum_{k=1}^m \rho^2 \frac{\tau_k^2 \rho^2}{8} \frac{1}{\tau_k^4 \rho^4} \frac{\tau_k \rho}{8}\int_{B_{\frac{\sqrt{8}\lambda}{\tau_k \rho}}(0)} \frac{e^{\ell_{\rho,k}\left( \frac{\tau_k \rho y}{\sqrt{8}} + \xi_{\rho,k} \right) + \varphi_{\rho}\left( \frac{\tau_k \rho y}{\sqrt{8}} + \xi_{\rho,k} \right)}}{\left(1 + \frac{\vert y \vert^2}{8} \right)^2}  \tilde{v}^{(k)}_{\rho,j}(y) y_i dy+\\
&+\mu_{\rho,j} \sum_{i=1,2} \frac{\partial G}{\partial y_i} (x,\xi_{k})\left(1+o(1) \right) \sum_{k=1}^m \rho^2 \frac{\tau_k^2 \rho^2}{8} \tau_k^4 \rho^4 \frac{\tau_k \rho}{8}\\
&\int_{B_{\frac{\sqrt{8}\lambda}{\tau_k \rho}}(0)} e^{-\ell_{\rho,k}\left( \frac{\tau_k \rho y}{\sqrt{8}} + \xi_{\rho,k} \right) - \varphi_{\rho}\left( \frac{\tau_k \rho y}{\sqrt{8}} + \xi_{\rho,k} \right)}\left(1 + \frac{\vert y \vert^2}{8} \right)^2  \tilde{v}^{(k)}_{\rho,j}(y) y_i dy +o(\rho)\\
&=\mu_{\rho,j} \sum_{i=1,2} \frac{\partial G}{\partial y_i} (x,\xi_{k})(1+o(1)) \\
&\sum_{k=1}^m \frac{1}{8\tau_k^2} \frac{\tau_k \rho}{\sqrt{8}}\int_{B_{\frac{\sqrt{8}\lambda}{\tau_k \rho}}(0)} \frac{e^{\ell_{\rho,k}\left( \frac{\tau_k \rho y}{\sqrt{8}} + \xi_{\rho,k} \right) + \varphi_{\rho}\left( \frac{\tau_k \rho y}{\sqrt{8}} + \xi_{\rho,k} \right)}}{\left(1 + \frac{\vert y \vert^2}{8} \right)^2}  \tilde{v}^{(k)}_{\rho,j}(y) y_i dy+ o(\rho).
\end{align*}
Therefore
\begin{align*}
\frac{v_{\rho,j}(x)}{\rho} &= \mu_{\rho,j} \sum_{i=1,2} \frac{\partial G}{\partial y_i} (x,\xi_{k})(1+o(1))\\
& \sum_{k=1}^m \frac{1}{8 \tau_k^2} \frac{\tau_k}{\sqrt{8}}\int_{B_{\frac{\sqrt{8}\lambda}{\tau_k \rho}}(0)} \frac{e^{\ell_{\rho,k}\left( \frac{\tau_k \rho y}{\sqrt{8}} + \xi_{\rho,k} \right) + \varphi_{\rho}\left( \frac{\tau_k \rho y}{\sqrt{8}} + \xi_{\rho,k} \right)}}{\left(1 + \frac{\vert y \vert^2}{8} \right)^2}  \tilde{v}^{(k)}_{\rho,j}(y) y_i dy+ o(1)\\
&\hbox{using Lebesgue}\\
&=\mu_{\rho,j} \sum_{i=1,2} \frac{\partial G}{\partial y_i} (x,\xi_{k})(1+o(1)) \sum_{k=1}^m \frac{\tau_k}{\sqrt{8}}\\
&\int_{\R^2} \frac{y_i}{\left(1 + \frac{\vert y \vert^2}{8} \right)^2} \frac{s_{1,j}^{(k)} y_1 + s_{2,j}^{(k)} y_2}{8+\vert y \vert^2}  dy+o(1)\\
&=2\pi \mu_{\rho,j} \sum_{i=1,2} \sum_{k=1}^m \frac{\partial G}{\partial y_i} (x,\xi_{k}) \frac{\tau_k s_{i,j}^{(k)}}{\sqrt{8}}+o(1),
\end{align*}
and so \ref{fine1} is proved.
 
By Lemma \ref{lemma}, we get
\begin{align*}
&(1-\mu_{\rho,n}) \rho^2 \int_{B_{\epsilon}(\xi_k)} \left( e^{u_{\rho}(x)} + e^{-u_{\rho}(x)} \right) \frac{\partial u_{\rho}}{ \partial x_j}(x) v_{\rho,n}(x) dx\\
&= \int_{\partial B_{\epsilon}(\xi_k)}\left( \frac{\partial v_{\rho,n}}{\partial \nu }(x) \frac{\partial u_{\rho}}{ \partial x_j}(x) -v_{\rho,n}(x)\frac{\partial^2 u_{\rho}}{ \partial \nu \partial x_j}(x)\right) d\sigma_x.
\end{align*}
For the left hand side, using \eqref{aiut}, we have
$$\rho^2 \int_{B_{\epsilon}(\xi_{\rho,k})} \left( e^{u_{\rho}(x)} + e^{-u_{\rho}(x)} \right) \frac{\partial u_{\rho}}{ \partial x_j}(x) v_{\rho,n}(x) dx=-\frac{16}{3\sqrt{8} \rho} \pi \frac{\alpha_k}{\tau_k} s_{j,n}^{(k)}(1+o(1)).$$
For the right hand side we have
\begin{align*}
&\int_{\partial B_{\epsilon}(\xi_k)} \left( \frac{\partial v_{\rho,n}}{\partial \nu }(x) \frac{\partial u_{\rho}}{ \partial x_j}(x) -v_{\rho,n}(x)\frac{\partial^2 u_{\rho}}{ \partial \nu \partial x_j}(x)\right) d\sigma_x\\
&=\rho \int_{\partial B_{\epsilon}(\xi_k)} \frac{\partial}{\partial \nu_x}\left( 2\pi \mu_{\rho,n} \sum_{i=1,2} \sum_{l=1}^m \frac{\partial G}{\partial y_i} (x,\xi_{l}) \frac{\tau_l s_{i,n}^{(l)}}{\sqrt{8}}+o(1) \right)\\
&\frac{\partial}{\partial x_j} \left(8 \pi \sum_{h=1}^m \alpha_h G(x, \xi_h)+o(1) \right) d\sigma_x +\\
&-\rho \int_{\partial B_{\epsilon}(\xi_k)} \left( 2\pi \mu_{\rho,n} \sum_{i=1,2} \sum_{l=1}^m \frac{\partial G}{\partial y_i} (x,\xi_{l}) \frac{\tau_l s_{i,n}^{(l)}}{\sqrt{8}}+o(1) \right)\\
&\frac{\partial^2}{\partial x_j \partial \nu_x} \left(8 \pi \sum_{h=1}^m \alpha_h G(x, \xi_h)+o(1) \right) d\sigma_x \\
&=16 \pi^2 \rho \mu_{\rho,n} \sum_{l,h=1}^m \sum_{i=1,2} \alpha_h \frac{\tau_l s_{i,n}^{(l)}}{\sqrt{8}}\\
&\int_{\partial B_{\epsilon}(\xi_k)} \left( \frac{\partial^2 G}{\partial y_i \partial \nu_x} (x, \xi_l) \frac{\partial G}{\partial x_j} (x,\xi_h) - \frac{\partial G}{\partial y_i} (x,\xi_l)\frac{\partial^2 G}{\partial x_j \partial \nu_x}(x,\xi_h)\right) d\sigma_x + o(\rho) \\
&\hbox{using Proposition \ref{propo}}\\
&= -16 \pi^2 \alpha_k \rho \mu_{\rho,n} \sum_{l=1}^m \sum_{i=1,2} \frac{\tau_l s_{i,n}^{(l)}}{\sqrt{8}} \frac{\partial^2 \mathcal{F}}{\partial x_i^{(l)} \partial x_j^{(k)}}(\xi_1, \ldots, \xi_m)+o(\rho),
\end{align*}
Hence
\begin{align*}
&-\frac{16}{3\sqrt{8} \rho} \pi \frac{\alpha_k}{\tau_k} s_{j,n}^{(k)}(1+o(1))(1-\mu_{\rho,n})\\
&= -16 \pi^2 \alpha_k \rho \mu_{\rho,n} \sum_{l=1}^m \sum_{i=1,2} \frac{\tau_l s_{i,n}^{(l)}}{\sqrt{8}} \frac{\partial^2 \mathcal{F}}{\partial x_i^{(l)} \partial x_j^{(k)}}(\xi_1, \ldots, \xi_m)+o(\rho).
\end{align*}
Having set
$$\eta_n:= \frac{\sum_{l=1}^m \sum_{i=1,2} \tau_l \tau_k s_{i,n}^{(l)} \frac{\partial^2 \mathcal{F}}{\partial x_i^{(l)} \partial x_j^{(k)}}(\xi_1, \ldots, \xi_m)}{s_{j,n}^{(k)}},$$
we get
$$s_{j,n}^{(k)} \eta_n = \sum_{l=1}^m \sum_{i=1,2} \tau_l \tau_k s_{i,n}^{(l)} \frac{\partial^2 \mathcal{F}}{\partial x_i^{(l)} \partial x_j^{(k)}}(\xi_1, \ldots, \xi_m),$$
so that $s_n= \left(s_{1,n}^{(1)}, \ldots, s_{1,n}^{(m)},s_{2,n}^{(1)}, \ldots, s_{2,n}^{(m)} \right)$ and $\eta_n$ are respectively the eigenvectors and the eigenvalues of the matrix $D (\mathrm{Hess} \, \mathcal{F}) D$. 
We can therefore conclude that 
$$(1-\mu_{\rho,n})=3 \pi \rho^2 \eta_n +o(\rho^2).$$
\endproof
Finally we can prove the following
\proof [Proof of Theorem \ref{main2}].

Proposition \ref{t13} proves point 1 and point 3, while Theorem \ref{tre3} jointly with point (i) of Proposition \ref{penultimo} proves point 2.
\endproof
And

\proof [Proof of Theorem \ref{main3}].

With similar calculus of those in Proposition \ref{penultimo} we can prove that for $3m+1<j \leqslant 4m$, $\mu_{\rho,j}$ still goes to one. In the proof of point (iii) in Proposition \ref{penultimo} we showed that for $3m+1<j \leqslant 4m$, $t^{(k)}_j \neq 0$,  so Lemma \ref{due2} holds and in such a case at the end of Theorem \ref{tre3} we proved \eqref{due}, therefore point 1  and point 3 are proved. Finally, considering Proposition \ref{uno1} and the fact that the functions $\frac{8-\vert x \vert^2}{8+\vert x \vert^2}$ and $\frac{x_i}{8+\vert x \vert^2}$, for $i=1,2$,  are orthogonal and that eigenfunctions relative to different eigenspaces have to be orthogonal we get point 2.
\endproof

We are now ready to prove the main result. 
\proof [Proof of Theorem \ref{finefinefine}]. 

For $\rho$ sufficiently small, by Proposition \ref{t3} we have that $\mu_{\rho,j} < 1$ for every $j=1, \ldots, m$. By Proposition \ref{penultimo} we have that $\mu_{\rho,j} >1$ for $j>3m$ and for every $m+1 \leqslant j \leqslant 3m$ we have that $\mu_{\rho,j} \leqslant 1+C \rho^2$, so by Theorem \ref{tre3} we get $t_j=\textbf{0}$. Therefore, since the hypotheses of Proposition \ref{t13} are satisfied, we get
$$\mu_{\rho,j}=1-3 \pi \rho^2 \eta_j (1+o(1)),$$
where the $\eta_j$'s, for $j=1, \ldots, 2 m$, are the eigenvalues of the matrix \\$D (\mathrm{Hess} \,\mathcal{F})D$. Then, to any positive eigenvalue $\eta_j$ of the matrix $D( \mathrm{Hess} \, \mathcal{F})D$ there corresponds a negative $\mu_{\rho,j}$.
Since $D$ is a diagonal, positive definite matrix, the signature of the matrix $D (\mathrm{Hess} \, \mathcal{F})D$ is equal to the signature of the matrix $\mathrm{Hess} \,\mathcal{F}$, and hence the theorem is proved.
\endproof

We deduce the following corollary.

\proof [Proof of Corollary \ref{corol}].

It is enough to notice that $\mathrm{Hess} \,\mathcal{F}$ is a $2m \times 2m$ matrix, and so  $0 \leqslant \mathcal{M}(\mathrm{Hess} \,\mathcal{F}) \leqslant 2m$.
\endproof

\bibliographystyle{abbrv}
\bibliography{bibliografia}

\end{document}